\documentclass{amsart}
\usepackage[utf8]{inputenc}

\usepackage[english]{babel}
\usepackage[T1]{fontenc}
\usepackage{lmodern}
\usepackage{amsmath}
\usepackage{amsthm}
\usepackage{amsfonts}
\usepackage{amssymb}
\usepackage{dsfont}
\usepackage{stmaryrd}
\usepackage[a4paper, scale=0.75,centering]{geometry}
\usepackage{csquotes}

\usepackage{hyperref}
\usepackage{xcolor}

\usepackage{enumitem} 
\setlist[1]{itemsep=1pt,topsep=2pt}
\setlist[enumerate]{label={(\roman*)}}

\usepackage[backend=biber,
style=alphabetic,
sorting=nyt,maxcitenames=99,
maxbibnames=99]{biblatex}
\numberwithin{equation}{section}

\theoremstyle{plain}

\newtheorem{theorem}{Theorem}[section]
\newtheorem{proposition}[theorem]{Proposition}
\newtheorem{lemma}[theorem]{Lemma}
\newtheorem{corollary}[theorem]{Corollary}
\newtheorem{remark}[theorem]{Remark}

\theoremstyle{definition}

\newtheorem{example}{Example}

\DeclareMathSymbol{\leqslant}{\mathalpha}{AMSa}{"36} 
\DeclareMathSymbol{\geqslant}{\mathalpha}{AMSa}{"3E} 
\DeclareMathSymbol{\eset}{\mathalpha}{AMSb}{"3F}     
\renewcommand{\le}{\, \leqslant\,} 
\renewcommand{\ge}{\, \geqslant\,}

\newcommand{\cA}{{\ensuremath{\mathcal A}} }
\newcommand{\cB}{{\ensuremath{\mathcal B}} }
\newcommand{\cC}{{\ensuremath{\mathcal C}} }
\newcommand{\cD}{{\ensuremath{\mathcal D}} }

\newcommand{\cF}{{\ensuremath{\mathcal F}} }

\newcommand{\cK}{{\ensuremath{\mathcal K}} }
\newcommand{\cL}{{\ensuremath{\mathcal L}} }
\newcommand{\cM}{{\ensuremath{\mathcal M}} }

\newcommand{\cT}{{\ensuremath{\mathcal T}} }

\newcommand{\cZ}{{\ensuremath{\mathcal Z}}}

\newcommand{\bA}{{\ensuremath{\mathbf A}} }

\newcommand{\bE}{{\ensuremath{\mathbf E}} }

\newcommand{\bL}{{\ensuremath{\mathbf L}} }

\newcommand{\bP}{{\ensuremath{\mathbf P}} }
\newcommand{\bQ}{{\ensuremath{\mathbf Q}} }

\newcommand{\bX}{{\ensuremath{\mathbf X}} }

\newcommand{\bbE}{{\ensuremath{\mathbb E}} }

\newcommand{\bbN}{{\ensuremath{\mathbb N}} }

\newcommand{\bbP}{{\ensuremath{\mathbb P}} }

\newcommand{\bbR}{{\ensuremath{\mathbb R}} }

\newcommand{\bbZ}{{\ensuremath{\mathbb Z}} }

\newcommand{\bpsi}{\boldsymbol\psi}
\newcommand{\bphi}{\boldsymbol\phi}
\newcommand{\bzeta}{\boldsymbol\zeta}
\newcommand{\bPhi}{\boldsymbol\Phi}

\newcommand{\deff}{d_{\mathrm{eff}}}
\newcommand{\bff}{\ensuremath{\mathbf{f}}}

\newcommand{\Es}{\mathds{E}}
\newcommand{\Pro}{\mathds{P}}
\newcommand{\Ind}{\mathds{1}}

\newcommand{\E}{\mathrm{E}}
\renewcommand{\L}{\mathrm{L}}
\renewcommand{\P}{\mathrm{P}}

\newcommand{\expo}{\mathrm{e}}
\newcommand{\norme}[1]{\left\Vert #1\right\Vert}
\newcommand{\floor}[1]{\lfloor #1 \rfloor}

\newcommand{\sumtwo}[2]{\sum_{\substack{#1 \\ #2}}} 

\newcommand{\intg}[2]{\llbracket #1,#2 \rrbracket}

\newcommand{\tomega}{\tilde{\omega}}
\newcommand{\homega}{\hat{\omega}}
\newcommand{\bomega}{\check{\omega}}
\newcommand{\hbeta}{\hat{\beta}}
\newcommand{\hh}{\hat{h}}

\newcommand{\dd}{\mathrm{d}}
\newcommand{\R}{\mathbb{R}}

\newcommand{\gep}{\varepsilon}

\title[Disordered systems and polynomial chaos with heavy-tail disorder]{Disordered systems and (subcritical) polynomial chaos\\ with heavy-tail disorder}

\author{Gaspard Gomez}
\address{DMA, École normale supérieure, Université PSL, CNRS, 75005 Paris, France}
\email{gaspard.gomez@ens.psl.eu}

\begin{document}

\begin{abstract}
	We study discrete statistical mechanics systems perturbed by a random environment without a finite second moment.
	Specifically, we consider a random environment whose tail distribution satisfies $\Pro[\omega > x] \sim x^{-\gamma}$ as $x \to +\infty$ for some $\gamma \in (1,2)$.
	Inspired by the seminal work of Caravenna, Sun and Zygouras \cite{csz_2016}, we adopt a general framework that encompasses as key examples both the disordered pinning model and the long-range directed polymer model.
	We provide some \textit{subcriticality} condition under which we prove that the discrete disordered system possesses a non-trivial scaling limit.
	We also interpret the subcriticality condition in terms of a generalized Harris criterion without second moment, which gives a prediction for \textit{disorder relevance} depending on the parameters of the system.
	Our analysis relies on the study of multilinear polynomials of independent heavy-tailed random variables known as \textit{polynomial chaos} and their continuous analogue, given by multiple integrals with respect to a $\gamma$-stable Lévy white noise.
	We develop precise and flexible moments estimates adapted to the heavy-tailed setting.
\end{abstract}

\maketitle

\setcounter{tocdepth}{1}
\tableofcontents

\section{Introduction and presentation of the results}{\label{section: introduction}}

In this article we study statistical mechanics models defined on a lattice which are perturbed by an external random field. 
Our work focuses on the study of the scaling limit of such disordered system, in the so-called intermediate disorder regime.
The approach we follow is in great part inspired by the work of Caravenna, Sun and Zygouras in \cite{csz_2016}.
Their work focuses on the case where the perturbation lies in the bassin of attraction of a Gaussian white noise and their techniques strongly rely on $\mathds{L}^2$ computations.

We will instead focus on an environment which lies in the domain of attraction of a (\(\gamma\)-stable) Lévy white noise. 
In that case, $\mathds{L}^2$ computations are not available and we need to develop  flexible~$\mathds{L}^p$ estimates for $p \in (1,2)$.  
We obtain a general criterion for \textit{subcriticality} of disordered systems, and we prove that such systems admit a non-trivial scaling limit. 
This is a generalization of the results presented in \cite{berger_lacoin_2021,berger_lacoin_2022,berger_chong_lacoin_2023} which focused on a specific model known as the directed polymer in a random environment. 
In the spirit of~\cite[Section~1.3]{csz_2016}, it can also be seen as a criterion for \textit{disorder relevance} in the case of a heavy-tail disorder, a question considered for instance in \cite{lacoin_sohier_2017} in the case of the disordered pinning model.
We then apply our general results to several statistical mechanics models such as the disordered pinning model or the long-range directed polymer model.

\subsection{Setting and first notations}

Our work deals with disordered systems, which are made of two ingredients: a (homogeneous) \textit{reference model}, perturbed by a random environment, also called \textit{disorder}.
We now introduce the general framework we work with; we give a few examples that will serve as a common thread.

\subsubsection{About the reference model}

Let us first introduce the homogeneous model which will be perturbed latter on. One should think about it as the law of a rescaled object that is very well understood.

Let $\Omega \subset \mathbb{R}^D$ be a \textit{bounded} open set, with \(D\) the dimension of the ambient space.
Consider for every $\delta > 0$ a discretization $\Omega_{\delta}$ of $\Omega$.
Let $\deff$ be the effective dimension of this discretization, that is the unique real number $\deff$ such that for any test function $f$ on $\Omega$, 
\begin{equation}
	\label{eq: Omega delta}
	\delta^{\deff} \sum\limits_{x \in \Omega_{\delta}} f(x) \xrightarrow{\delta \downarrow 0} \int_{\Omega} f(x) \dd x \,.
\end{equation}
Note that $\deff$ could be different from the ambient dimension $D$ (see the examples below in Section~\ref{section: statistical mechanics}). 
Suppose that a reference probability measure $\P_{\delta}^{\rm{ref}}$ is given on $\{0,1\}^{\Omega_{\delta}}$, which describes a field $(\sigma_x)_{x \in \Omega_{\delta}}$.
The law $\P_{\delta}^{\rm{ref}}$ is referred to as the \textit{homogeneous model}.

\begin{example}[\bf Renewal processes and pinning model]
	\label{ex:pinning}
	Let $\tau = (\tau_k)_{k \ge 0}$ be a renewal process on $\bbN$ with inter-arrival law $\P(\tau_1 = n) \sim c_0 n^{-(1+\alpha)}$ as \(n\to\infty\), for some \(\alpha\in (0,1)\). 
	In other words \(\tau_0=0\) and the \((\tau_i-\tau_{i-1})_{i\geq 1}\) are i.i.d.\ \(\bbN\)-valued random variables; with a small abuse of notation, we also interpret the renewal process \(\tau\) as a subset of \(\bbN\).
	Then, the dimension is \(D=1\) and we take $\Omega = (0,1)$. 
	We let $\Omega_{\delta} = (\delta \bbZ) \cap \Omega$ where we have set $\delta = \frac{1}{N}$, with \(N\in \mathbb{N}\) the length of the system.
	Here, the effective dimension of the discretization is $\deff = 1=D$. 

	For \(x \in \Omega_{\delta}\), we can define \(\sigma_x = \Ind_{\{ xN \in \tau\}}\) and consider $\P_{\delta}^{\rm ref} $ the law of the random field $(\Ind_{\tau_{\delta}}(x))_{x \in \Omega_{\delta}}$ where $\tau_{\delta} = (\delta \tau) \cap \Omega$.
	Then, the law of $\P_{\delta}^{\rm ref}$ admits a scaling limit as \(\delta\downarrow 0\), which is the law of the $\alpha$-stable regenerative set (see e.g.\ \cite[App.~A.5.4]{giacomin_2007} and Section \ref{section: pinning} below).
\end{example}

Although we assume that the domain $\Omega$ is bounded (this assumption plays an important role, see the comments after Theorem \ref{thm: continuous polynomial chaos}), our results can be extended to several models defined on an unbounded domain such as random walks. 

\begin{example}[\bf Long-range random walk]
\label{ex:longrange}
Let $(S_n)_{n \ge 1}$ be a random walk on $\bbZ^{d}$, \(d\geq 1\), with i.i.d.\ increments in the (strict) domain of attraction of an $\alpha$-stable law with $\alpha\in (0,2]$; in other words, assume that \((n^{-1/\alpha} S_n)_{n\geq 0}\) converges in distribution towards some \(d\)-dimensional \(\alpha\)-stable law.
In that case, we take \(D=1+d\) and $\Omega = (0,1) \times \bbR^{d}$. 
Then, for \(N\in \mathbb{N}\) large, let the mesh size be $\delta = \frac{1}{N}$ and consider the discretization $\Omega_{\delta} = \left((\delta \bbZ) \times (\delta^{1/\alpha} \bbZ^{d})\right) \cap \Omega$. 
Here the effective dimension of the discretization is $\deff = 1 + \frac{d}{\alpha}$ and differs from \(D=1+d\). 

Then, for \(x = (t,u) \in \Omega_{\delta}\), define \(\sigma_x= \Ind_{\{S_{tN}= u N^{1/\alpha}\}}\) so that the random walk trajectory in time-space up to time~\(N\) can be interpreted through its occupation field in \(\{0,1\}^{\Omega_{\delta}}\).
Thus, we define $\P_{\delta}^{\rm ref}$ as the law of the random field $(\Ind_{\cA_{\delta}}(x))_{x \in \Omega_{\delta}}$ where $\cA_{\delta} = \{ (\frac{n}{N}, \frac{S_n}{N^{1/\alpha}})\}_{1 \le n \le N}$. 
The law $\P_{\delta}^{\rm ref}$ admits a scaling limit as \(\delta\downarrow0\), which is the law of the random field $(\Ind_{\bA}(x))_{x \in \Omega}$ where $\bA = \{ (t,X_t) \}_{t \ge 0}$ with $(X_t)_{t \ge 0}$ an $\alpha$-stable Lévy process (a Brownian motion when $\alpha =2$).
\end{example}

\subsubsection{Introducing disorder}

Let $\omega = (\omega_x)_{x \in \Omega_{\delta}} $ be i.i.d.\ \textit{centered} random variables such that~$\omega_x > -1$\  almost surely (the reason for this choice will be clear a few lines below).
Their law will be denoted $\Pro$ and expectation with respect to $\omega$ will be written $\Es$.
Depending on the context, $\omega$ will either stand for a generic random variable which has the same law as $\omega_x$ or the full environment $(\omega_x)_{x \in \Omega_{\delta}}$, and will be referred to as \textit{disorder}. 
We insist on the fact that the randomness of $\omega$ is not the same as the randomness of the homogeneous model~$\sigma$, and in fact we assume that \(\omega\) and \(\sigma\) are independent.

Given $\beta \in  [0,1]$ (the disorder intensity) and a  \emph{quenched} realization of $\omega$, we define the disordered probability measure $\P_{\Omega_{\delta}}^{\omega,\beta}$ for the field $\sigma = (\sigma_x)_{x \in \Omega_{\delta}}$ as a Gibbs modification of \(\P_{\delta}^{\rm{ref}}\) as follows:
\begin{equation}{\label{eq: disordered model}}
	\P_{\Omega_{\delta}}^{\omega,\beta}(\dd\sigma) := \frac{1}{Z_{\Omega_{\delta}}^{\omega,\beta}} \prod\limits_{x \in \Omega_{\delta}} (1+\beta \omega_x \sigma_x) 
	 \ \P_{\delta}^{\rm{ref}}(\dd\sigma) .
\end{equation}
The normalizing constant $Z_{\Omega_{\delta}}^{\omega,\beta}$ that makes $\P_{\Omega_{\delta}}^{\omega,\beta}(\dd\sigma)$ a probability measure is called the partition function, and is defined by 
\begin{equation}
	\label{def: partition function}
	Z_{\Omega_{\delta}}^{\omega,\beta} := \E_{\delta}^{\rm ref}\bigg[\prod\limits_{x \in \Omega_{\delta}}(1+\beta \omega_x \sigma_x)\bigg] 
	.
\end{equation}
We will refer to $\P_{\Omega_{\delta}}^{\omega,\beta}$ as the \textit{disordered system}.
The only reason we require $\beta \le 1$ is to guarantee that the density of $\P_{\Omega_{\delta}}^{\omega,\beta}$ with respect to $\P_{\delta}^{\rm{ref}}$ is positive (recall we assume that \(\omega > -1\) a.s.). 

\begin{example}[\bf Disordered pinning model]
	\label{ex:pinning2}
	Coming back to Example~\ref{ex:pinning}, the disordered version of the model is known as the \textit{disordered pinning model on a defect line}.
	We come back to this example more in depth in Section~\ref{section: pinning} below; see \cite{giacomin_2007,giacomin_2011} for comprehensive references.
\end{example}

\begin{example}[\bf Long-range directed polymer]
	\label{ex:longrange2}
	As far as Example~\ref{ex:longrange} is concerned, the disordered system based on (long-range) random walks is known as the (long-range) directed polymer model.
	We come back to this example more in depth in Section~\ref{section: directed polymer} below; we refer to~\cite{comets_2017} for a comprehensive overview of the directed polymer model.
\end{example}

\begin{remark}[Product vs.\ exponential weights]
	Readers familiar with disordered systems might be slightly surprised by our definition of the disordered probability measure.
	Usually the Gibbs weight is indeed written in the exponential form $\exp(\beta  \sum_{x \in \Omega_{\delta}} \tilde{\omega}_x \sigma_x)$.
	First of all, let us stress that, for fixed $\beta$, this is the same up to the change of variable $\omega = \frac{1}{\beta}(\expo^{\beta \tilde{\omega}} -1)$. 	
	There are several reasons for which we make this choice.
	The most obvious one is certainly that the product form allows for a more direct polynomial chaos expansion\footnote{Note that this product form is originally used in~\cite{huse_henley_1985} who introduced the directed polymer model, and it also appears explicitly in the seminal work~\cite{akq_2014}, see Equation~(4), as an intermediary step when considering the intermediate disorder scaling limit of the directed polymer model.}, 
	as presented in Section~\ref{section: L2 case} below: in particular, the relevant variables in the expansion are then the~\(\omega\)'s and \textit{not} the \(\tilde{\omega}\)'s.  
	Another reason to consider the product form is that, when considering heavy-tailed random variables, the exponential form gives an overwhelming importance to the high values of \(\tilde \omega\) and the tail of \(\expo^{\beta \tilde{\omega}}\) changes with \(\beta\); in particular, after rescaling, \(\omega\) possesses a scaling limit whereas \(\expo^{\beta \tilde{\omega}}\) never does.
	This is why \cite{lacoin_sohier_2017} and then \cite{berger_lacoin_2021} considered a product form random perturbation (see in particular Remark~1.2 in \cite{berger_lacoin_2021} for more motivations of this choice), which is in fact most natural when investigating the relation between disordered systems and stochastic PDEs, see Sections~\ref{subsec: remark SPDE pinning}-\ref{subsec: remark SPDE polymer} below.
\end{remark}

The question we address in this article is that of the scaling limit of the disordered system in the limit $\delta\downarrow0$, choosing $\beta=\beta_{\delta} \downarrow 0$ in some appropriate scaling window, known as the intermediate disorder regime.

\subsection{Intermediate disorder in the \texorpdfstring{$\mathds{L}^2$}{L2} case: a quick overview}{\label{section: L2 case}}

For pedagogical purposes, before introducing our main results on heavy-tail random variables and Lévy white noises, let us briefly present and comment the case where the random variables~$\omega$ have a finite $\mathds{L}^2$ moment. 
This section reproduces heuristics from Section~1.2. of \cite{csz_2016} and allows us to introduce some important notation and ideas, in particular the notion of \textit{subcritical polynomial chaos}.
Without loss of generality we assume in this section that $\Es[\omega^2] = 1$. 

The starting point of the discussion is the polynomial chaos expansion of the partition function. 
It is helpful for heuristically identifying what we mean by \textit{subcritical regime}, what the assumptions on the homogeneous system are, and what the scaling limit of the disordered system should be. 
Let us focus for now on the partition function \(Z_{\Omega_{\delta}}^{\omega,\beta}\) in~\eqref{def: partition function}.
Expanding the product $\prod_{x \in \Omega_{\delta}}(1+\beta \omega_x \sigma_x)$ and taking the expectation yields
\begin{equation}{\label{eq: first chaos expansion of the partition function}}
	  Z_{\Omega_{\delta}}^{\omega,\beta} = 1 + \sum_{k \ge 1} \beta^k \sum\limits_{\{x_1,\dots,x_k\} \subset \Omega_{\delta}} \E_{\delta}^{\rm ref}\bigg[\prod\limits_{j=1}^k \sigma_{x_j}\bigg] \prod\limits_{j=1}^k \omega_{x_j}  \,.
\end{equation}
In this expansion, note that the $k$-point correlation functions of the homogeneous system appear.
Since the homogeneous system has a scaling limit, it is natural (and satisfied for the systems presented in the above examples) to assume that the correlation functions themselves have a scaling limit. 
More precisely, define for \(x_1,\dots,x_k \in \Omega_{\delta}\) and $\lambda \in (0,+\infty)$,
\begin{equation}{\label{eq: rescaled correlation function}}
	\psi_{\Omega_{\delta}}^{(k)}(x_1,\dots,x_k) := (\delta^{-\lambda})^k \; \E_{\delta}^{\rm ref}\bigg[\prod\limits_{j=1}^k \sigma_{x_j}\bigg],
\end{equation}
with by convention \(\psi_{\Omega_{\delta}}^{(k)}(x_1,\dots,x_k) =0\) on the diagonals, \textit{i.e.}\ if the \(x_i\) are not distinct.
Then, we assume that the scaling exponent \(\lambda\) is such that, for every $k \ge 1$,
\begin{equation}{\label{eq: convergence of the rescaled correlation function}}
	\psi_{\Omega_{\delta}}^{(k)}(x_1,\dots,x_k) \xrightarrow{\ \delta \downarrow 0\ } \bpsi_{\Omega}^{(k)}(x_1,\dots,x_k) \,,
\end{equation}
for some non-trivial function \(\bpsi_{\Omega}^{(k)}\), interpreted as the \(k\)-point correlation function of the scaling limit of the reference model; we also set \(\bpsi_{\Omega}^{(k)}=0\) on the diagonals.
For the time being, we are voluntarily remaining evasive about the precise meaning of the convergence~\eqref{eq: convergence of the rescaled correlation function}. 

\begin{example}[\bf Pinning model]
	For the pinning model of Example~\ref{ex:pinning}, we have that \(\psi_{\Omega_{\delta}}^{(k)}(x_1,\dots,x_k) = \prod_{i=1}^k \P( (x_i-x_{i-1})N \in \tau)\) for any \(0<x_1<\cdots <x_k <1\), with by convention \(x_0=0\).
	Then, we have that \eqref{eq: convergence of the rescaled correlation function} holds with \(\lambda = 1-\alpha\) thanks to \cite{doney_1997}, and \(\bpsi_{\Omega}^{(k)}(x_1,\dots,x_k) = \prod_{i=1}^k(x_i-x_{i-1})^{\alpha-1}\).
\end{example}

\begin{example}[\bf Long-range directed polymers]
	For the long-range directed polymer of Example~\ref{ex:longrange}, we have that \(\psi_{\Omega_{\delta}}^{(k)}( (t_1,u_1),\dots,(t_k,u_k)) = \prod_{i=1}^k \P( S_{ (t_i-t_{i-1})N} = (u_i-u_{i-1}) N^{1/\gamma})\) for any \(0<t_1<\cdots <t_k <1\), with by convention \(t_0=0\), \(x_0=0\).
	Then, by the local limit theorem, we have that  \eqref{eq: convergence of the rescaled correlation function} holds with \(\lambda = \frac{d}{\alpha}\), and \(\bpsi_{\Omega}^{(k)}(x_1,\dots,x_k) = \prod_{i=1}^k \rho_{t_i-t_{i-1}}(u_i-u_{i-1})\) where \(\rho_t(u)\) is the transition density of the limiting \(\alpha\)-stable Lévy process \((X_t)_{t\geq 0}\).
\end{example}

With the definition~\eqref{eq: rescaled correlation function} and using the symmetry of the \(k\)-point correlation functions, this allows us to rewrite the partition function as 
\begin{equation}{\label{eq: chaos expansion for L2 computation}}
	Z_{\Omega_{\delta}}^{\omega,\beta} 
	= 1 + \sum\limits_{k \ge 1}  \frac{1}{k!} (\beta \delta^{\lambda})^{k} \sum\limits_{x_1,\dots,x_k \in \Omega_{\delta}} \psi_{\Omega_{\delta}}^{(k)}(x_1,\dots,x_k) \prod\limits_{j=1}^k \omega_{x_j}.
\end{equation}
We can now compute the $\mathds{L}^2$ moment of $Z_{\Omega_{\delta}}^{\omega,\beta}$. 
Observe that in \ref{eq: chaos expansion for L2 computation}, the terms $\prod_{i=1}^k \omega_{x_i}$ and $\prod_{j=1}^l \omega_{y_j}$ are orthogonal in $\mathds{L}^2$ except if the sets $\{x_1,\ldots,x_k\}$ and $\{y_1,\dots,y_l\}$ coincide.
Therefore, 
\begin{equation*}
	\Es\big[(Z_{\Omega_{\delta}}^{\omega,\beta})^2\big] 
	= 1 + \sum\limits_{k \ge 1}\frac{1}{k!} (\beta \delta^{\lambda})^{2k}  \sum\limits_{x_1,\dots,x_k \in \Omega_{\delta}} \psi_{\Omega_{\delta}}^{(k)}(x_1,\dots,x_k)^2 \,.
\end{equation*}
Then, at least heuristically, for every $k \ge 0$,
\begin{equation*}
		\sum\limits_{x_1,\dots,x_k \in \Omega_{\delta}} \psi_{\Omega_{\delta}}^{(k)}(x_1,\dots,x_k)^2 \approx \delta^{-k\deff} \int_{\Omega^k} \bpsi_{\Omega}^{(k)}(x_1,\dots,x_k)^2 \dd x_1 \dots \dd x_k \,,
\end{equation*}
recalling that \(\deff\) is the effective dimension of the discretization \(\Omega_{\delta}\), see~\eqref{eq: Omega delta}.
This yields
\begin{equation}{\label{eq: identification of Harris criterion}}
	\Es\big[(Z_{\Omega_{\delta}}^{\omega,\beta})^2\big] \approx 1 + \sum\limits_{k \ge 1} \frac{1}{k!} \left(\beta \,\delta^{\lambda-\deff/2}\right)^{2k} \int_{\Omega^k} \bpsi_{\Omega}^{(k)}(x_1,\dots,x_k)^2 \dd x_1 \dots \dd x_k.
\end{equation}

\noindent
Now, there are two cases to distinguish:
\begin{enumerate}
	\item \textit{If $\deff/2 - \lambda  < 0$}.
	For $\beta$ small enough, one may show that $Z_{\Omega_{\delta}}^{\omega,\beta}$ has a uniformly bounded~$\mathds{L}^2$ moment. 
	In that case, any weak disorder limit (\(\beta=\beta_{\delta} \downarrow0\)) is trivial and \textit{disorder is irrelevant} (in the sense of \cite[\S1.3]{csz_2016}).
	We also dub this case \textit{supercritical}, since the terms in the chaos expansion~\eqref{eq: identification of Harris criterion} grow with \(k\); it also corresponds to a case where the integrals \(\int_{\Omega^k} \bpsi_{\Omega}^{(k)}(x_1,\dots,x_k)^2 \dd x_1 \dots \dd x_k\) are divergent.
	
	\item \textit{If $ \deff/2 - \lambda  > 0$.}
	In particular, for a fixed $\beta > 0$, the $\mathds{L}^2$ moment of $Z_{\Omega_{\delta}}^{\omega,\beta}$ diverges. 
	One may hope to scale \(\beta=\beta_{\delta} \downarrow 0\) appropriately to obtain a non-trivial scaling limit and \textit{disorder is relevant} (in the sense of \cite[\S1.3]{csz_2016}). 
	We also dub this case \textit{subcritical}, since the terms in the chaos expansion~\eqref{eq: identification of Harris criterion} vanish with \(k\); it also corresponds to a case where the integrals \(\int_{\Omega^k} \bpsi_{\Omega}^{(k)}(x_1,\dots,x_k)^2 \dd x_1 \dots \dd x_k\) are convergent.
\end{enumerate}
This disjunction of cases is linked to the \textit{Harris criterion} for disorder relevance; we will come back to this in Section~\ref{section: Harris} below.
We do not mention the marginal case where $\lambda = \frac{\deff}{2}$ because this is not the focus of the present article. 
Instead, we refer to \cite{csz_2017} for an extensive discussion, or~\cite{zygouras_2024,CSZ_review25} for an overview of the recent progress in the marginally relevant case.

\smallskip
Assume that we are in the subcritical (disorder relevant) case, \textit{i.e.} $\lambda < \deff/2$. 
From \eqref{eq: identification of Harris criterion}, we may identify the correct \textit{intermediate disorder regime}.
Indeed, if one scales \(\beta=\beta_{\delta} \downarrow 0\) as \(\delta\downarrow0\), in such a way that
\begin{equation}{\label{eq: intermediate window L2}}
	\lim_{\delta\downarrow0} \,\beta_{\delta} \, \delta^{\deff/2-\lambda}  = \hbeta \in [0,\infty) \,,
\end{equation}
then we obtain that
\(
\Es[(Z_{\Omega_{\delta}}^{\omega,\beta})^2] \approx 1 + \sum_{k \ge 1} \frac{\hbeta^{2k}}{k!} \int_{\Omega^k} \bpsi_{\Omega}^{(k)}(x_1,\dots,x_k)^2 \dd x_1 \dots \dd x_k \,. 
\)
Hence, a very natural assumption on the \(k\)-point correlation functions is that
\begin{equation}
	\sum\limits_{k \ge 1} \frac{\hbeta^{2k}}{k!} \int_{\Omega^k} \bpsi_{\Omega}^{(k)}(x_1,\dots,x_k)^2 \dd x_1 \dots \dd x_k < +\infty
	.
\end{equation}
This is the main assumption (written in a slightly less general set up) of \cite[Theorem~2.3]{csz_2016}, whose conclusion is that $Z_{\Omega_{\delta}}^{\omega,\beta}$ converges in distribution as $\delta \to 0$. 
We postpone the precise description of the limit to Section~\ref{section: main results}, where we will compare this with the results we obtain.

\subsection{The case of heavy-tail disorder: subcriticality and disorder relevance}

From now on, we focus on the case of an environment $\omega$ whose tail distribution has a power-law decay. 
More precisely, assume that \(\Es[\omega]=0\) and that there exist  $\gamma \in (1,2)$ and a constant $C_0$ such that 
\begin{equation}
	\label{def: omega tail}
	\Pro[\omega > t] \sim C_0 t^{-\gamma}\,, \quad \text{ as } t \to +\infty\,.
\end{equation}
By the so-called layer-cake representation, this implies that $\omega$ has a finite expectation and an infinite variance. 
We could consider a slightly more general class of environments, allowing for slowly-varying functions in the tail of $\omega$.
The general setting will be presented in Section \ref{section: convergence polynomial chaos}(see \eqref{eq: definition of a general heavy-tailed random variable}).
In this introduction, for pedagogical purposes, we restrict ourselves to the pure-power tail \eqref{def: omega tail}.

Let us now identify the subcritical regime and the intermediate disorder window as we did in the $\mathds{L}^2$ case. 
We keep the notation of the rescaled correlation function introduced in~\eqref{eq: rescaled correlation function}.
We can no longer calculate the $\mathds{L}^2$ moment of the partition, instead we will use the scaling limit of the environment. 
Let us define the discrete noise
\begin{equation}{\label{eq: definition of the discrete noise}}
\zeta_{\delta} := \delta^{\deff/\gamma} \sum\limits_{x \in \Omega_{\delta}} \omega_x \delta_x.
\end{equation}
Under assumption~\eqref{def: omega tail}, we have that, as $\delta \downarrow 0$, $\zeta_\delta$ converges (in some specific Sobolev space) to a \(\gamma\)-stable Lévy white noise $\bzeta$, which is a random distribution which is the equivalent of Gaussian white noise in a non-$\mathds{L}^2$ setting. 
We will give a more precise description of it below, see Section~\ref{section: subcritical polynomial chaos};  we also refer to Appendix~\ref{thm: functional convergence for the noise} for details. 

With the above notation, we can rewrite our partition function as follows (recall~\eqref{eq: chaos expansion for L2 computation}):
\begin{equation}{\label{eq: chaos expansion of the partition function}}
	\begin{split}
		Z_{\Omega_{\delta}}^{\omega,\beta} &= 1 + \sum\limits_{k \ge 1} \frac{1}{k!}\left(\beta \delta^{\lambda-\deff/\gamma}\right)^k \sum\limits_{x_1,\dots,x_k \in \Omega_{\delta}} \psi_{\Omega_{\delta}}^{(k)}(x_1,\dots,x_k) \prod\limits_{j=1}^k \big(\delta^{\deff/\gamma} \omega_{x_j} \big) \\
		&= 1 + \sum\limits_{k \ge 1} \frac{1}{k!}\left(\beta \delta^{\lambda-\deff/\gamma} \right)^k \int_{\Omega^k} \overline{\psi}_{\Omega_{\delta}}^{(k)}(x_1,\dots,x_k) \prod\limits_{j=1}^k \zeta_{\delta}(\dd x_j),
	\end{split}
\end{equation}
where $\overline{\psi}_{\Omega_{\delta}}^{(k)}$ denotes the piecewise constant extension of $\psi_{\Omega_{\delta}}^{(k)}$ to $\Omega^k$.
Under this form, it is now clear that \textit{subcriticality} or \textit{supercriticality} (that can also be interpreted as disorder \textit{relevance} or \textit{irrelevance}) depends on whether $\lambda > \deff/\gamma$ or \(\lambda < \deff/\gamma\).
Furthermore, in analogy with~\eqref{eq: intermediate window L2} in the \(\mathds{L}^2\) case, the intermediate disorder regime now corresponds to
\begin{equation}{\label{eq: intermediate disorder window}}
	\lim_{\delta\downarrow 0} \, \beta_{\delta}\, \delta^{\deff/\gamma-\lambda} = \hbeta \in [0,\infty) \,.
\end{equation} 
Overall, the limit of the rescaled partition function clearly should be the expression~\eqref{eq: chaos expansion of the partition function} with \(\psi_{\Omega_{\delta}}^{(k)}\) and \(\zeta_{\delta}\) replaced by their respective limits \(\bpsi_{\Omega}^{(k)}\) and \(\bzeta\). 
However, without even mentioning summability issues, the meaning of the iterated integrals appearing in the sum is not clear: indeed $\bzeta$ lives in some negative Sobolev space and the correlation functions could be too singular.

\subsection{Main results in a nutshell}

Let us now briefly present our main results to give an overview of the scope of the paper.
We purposely remain evasive on some definitions since they would require the introduction of further notation: for the full statements of our results, we refer to Section~\ref{section: main results} below.

First of all, we show that, under appropriate (subcritical) integrability condition on the continuum \(k\)-point correlation functions, the candidate scaling limit is well defined.

\begin{theorem}[Well-posedness of the continuous partition function]
\label{thm: continuum Z introduction}
Let \(\bzeta\) be a \(\gamma\)-stable Lévy white noise with \(\gamma \in (1,2)\).
Assume that there is some \(q\in (\gamma,2]\) such that for every $C \ge 0$,
\begin{equation}
	\label{eq: integrability phi}
	\sum_{k\geq 1} C^k \|\bpsi^{(k)}\|_{q} <+\infty \,,
\end{equation}
where \(\|\cdot\|_{q}\) denotes the symmetric \(L^q \)-norm (see its definition \eqref{eq: def of symmetric norms} below).
Then, the continuum partition function 
\begin{equation}
	\label{eq: limiting chaos expansion}
	\cZ_{\Omega}^{\bzeta,\hbeta} = 1 + \sum\limits_{k \ge 1} \frac{\hbeta^k}{k!} \int_{\Omega^k} \bpsi^{(k)}(x_1,\dots,x_k) \prod\limits_{j=1}^k \bzeta(\dd x_j)
\end{equation}
is a well-defined (and non-degenerate) \(\bzeta\)-measurable random variable in \(\mathbb{L}^p\) for every $ p \in [0,\gamma) $.
\end{theorem}

Our second result is that, under an appropriate \textit{subcritical} condition, the limit of the partition function in the intermediate disorder regime~\eqref{eq: intermediate disorder window} is the random variable constructed in Theorem~\ref{thm: continuum Z introduction}. 

\begin{theorem}[Intermediate disorder limit]
	\label{thm: convergence introduction}
	Assume that the noise \(\omega\) is centered and has a tail distribution given by~\eqref{def: omega tail} with \(\gamma \in (1,2)\). 
	Assume also that there exists some \(q\in (\gamma,2]\) such that the correlation functions \((\psi_{\delta}^{(k)})_{k\geq 1}\) satisfy the following:
	\begin{enumerate}
		\item For every \(k\in \mathbb{N}\) there exists some continuous \(k\)-point correlation function \(\bpsi^{(k)}\) such that 
		\[
		\lim_{\delta\downarrow 0} \big\| \overline{\psi}_{\delta}^{(k)} - \bpsi^{(k)} \big\|_q=0 \,;
		\]

		\item The following uniform integrability condition is verified: for every \(C>0\),
		\[
		\limsup_{\delta \downarrow 0} \sum_{k \ge 0} C^k \| \overline{\psi}_{\delta}^{(k)} \|_{q} < +\infty
		\quad \text{ and } \quad
		\lim\limits_{M \to +\infty} \displaystyle\displaystyle\limsup_{\delta \downarrow 0} \sum\limits_{k > M} C^k \| \overline{\psi}_{\delta}^{(k)} \|_{q} = 0.
		\]
	\end{enumerate}
	Then, under the \emph{subcritical} condition \(\lambda < \frac{\deff}{\gamma}\), in the intermediate disorder regime of~\eqref{eq: intermediate disorder window}, \textit{i.e.}\ if \(\lim_{\delta\downarrow 0}   \beta_{\delta} \delta^{\deff/\gamma - \lambda} = \hbeta \in [0,\infty)\), we have the following convergence in distribution:
	\[
	Z_{\Omega_{\delta}}^{\omega,\beta_{\delta}} \xrightarrow{\ (d)\ }  \cZ_{\Omega}^{\bzeta,\hbeta} \,,
	\]
	where \(\cZ_{\Omega}^{\bzeta,\hbeta}\) is defined in~\eqref{eq: limiting chaos expansion}, with \(\bzeta = \lim_{\delta \downarrow0} \zeta_{\delta}\) a \(\gamma\)-stable Lévy white noise. 
\end{theorem}

In particular, these results can be applied to the disordered pinning model and the directed polymer model.
We develop in detail these two examples in Section~\ref{section: statistical mechanics}, but let us quickly summarize the results we obtain.
In particular, in the subcritical regime \(\lambda<\frac{\deff}{\gamma}\), additionally to the convergence of the disordered partition function, we establish the convergence of the disordered Gibbs measures towards their continuum counterparts.
They are the analogue of the results of \cite{berger_lacoin_2021,berger_lacoin_2022} for the directed polymer model.

\subsubsection{For the pinning model}

For the disordered pinning model of Example~\ref{ex:pinning} and~\ref{ex:pinning2}, we prove the following. 
We refer to Theorem~\ref{thm: convergence of the partition function of the pinning model}-\ref{thm: convergence of the  pinning measure} for precise (and slightly more general) statements.

\begin{theorem}
	\label{thm: pinning-1}
	Assume that \( \alpha > 1-\frac{1}{\gamma} \) in Example~\ref{ex:pinning} and that \(\beta_{\delta} \downarrow 0\) with \(\lim_{\delta \downarrow 0} \beta_{\delta} \delta^{\alpha-1+\frac{1}{\gamma}} = \hat \beta \in (0,\infty)\).
	Then, we have the following convergence in distribution for the disordered partition function
	\[
	Z_{\Omega_{\delta}}^{\omega,\beta_{\delta}} \xrightarrow[\ \delta\downarrow 0\ ]{ (d) } \cZ_{\Omega}^{\bzeta,\hbeta} \,,
	\]
	where \(\cZ_{\Omega}^{\bzeta,\hbeta}\) a \(\bzeta\) -measurable non-degenerate random variable expressed as a polynomial chaos~\eqref{eq: limiting chaos expansion}, with \(\bzeta = \lim_{\delta \downarrow 0} \zeta_{\delta}\) a \(\gamma\)-stable Lévy white noise on \(\Omega=(0,1)\).

	Additionally, there exists a (disordered) continuum measure $\bQ_0^{\bzeta,\hbeta}$ on the space \(\mathcal{C}_{\infty}\) of closed subsets of \((0,1)\) (with the Matheron topology) such that we have the following convergence in distribution\footnote{on the set \(\mathcal{M}_1(\mathcal{C}_{\infty})\) of probability distributions on \(\mathcal{C}_{\infty}\) endowed with the weak topology.}:
	\[
	\P_{\Omega_{\delta}}^{\omega,\beta} \xrightarrow[\ \delta\downarrow 0\ ]{ (d) } \bQ_0^{\bzeta,\hbeta} \,.
	\]
\end{theorem}

The disordered measure \(\bQ_0^{\bzeta,\hbeta}\) is referred to as the \emph{continuum directed pinning model with \(\gamma\)-stable Lévy noise}.
In particular, Theorem~\ref{thm: pinning-1} establishes the conjecture of \cite[\S 2.5.3-(B)]{berger_lacoin_2022} also mentioned in \cite[\S 2.2]{berger_lacoin_2021} (note that the roles of \(\alpha\) and \(\gamma\) are reversed in \cite{berger_lacoin_2021,berger_lacoin_2022}). 
Our proofs build on recent results from~\cite{faugere_lacoin}, where the random measure \(\bQ_0^{\bzeta,\hbeta}\) is also studied.

\subsubsection{For the (long-range) directed polymer model}

For the long-range directed polymer model of Example~\ref{ex:longrange} and~\ref{ex:longrange2}, we prove the following. 
We refer to Theorems~\ref{thm: convergence of the polymer partition function}-\ref{thm: convergence of the polymer measure} for precise statements.

\begin{theorem}
	\label{thm: polymer-1}
	Assume that \( \frac{\alpha}{d} >  \gamma-1\) in Example~\ref{ex:longrange} and that \(\beta_{\delta} \downarrow 0\) with \(\lim_{\delta \downarrow 0} \beta_{\delta} \delta^{\frac{1}{\alpha \gamma}(\alpha - d (\gamma-1))} = \hat \beta \in (0,\infty)\).
	Then, we have the following convergence in distribution for the disordered partition function
	\[
	Z_{\Omega_{\delta}}^{\omega,\beta_{\delta}} \xrightarrow[\delta\downarrow 0]{ (d) } \cZ_{\Omega}^{\bzeta,\hbeta} \,,
	\]
	where \(\cZ_{\Omega}^{\bzeta,\hbeta}\) a \(\bzeta\) -measurable non-degenerate random variable expressed as a polynomial chaos~\eqref{eq: limiting chaos expansion}, with \(\bzeta = \lim_{\delta \downarrow 0} \zeta_{\delta}\) a \(\gamma\)-stable Lévy white noise on \(\Omega=(0,1) \times \mathbb{R}^d\).

	Additionally, there exists a (disordered) continuum measure $\bQ^{\bzeta,\hbeta}$ on the space \(\mathcal{D}([0,1],\mathbb{R}^{d})\) of càdlàg functions with values in \(\mathbb{R}^{d}\) (with the Skorokhod topology) such that we have the following convergence in distribution\footnote{on the set \(\mathcal{M}_1(\mathcal{D}([0,1],\mathbb{R}^{d}))\) of probability distributions on \(\mathcal{D}([0,1],\mathbb{R}^{d})\) endowed with the weak topology.}:
	\[
	\P_{\Omega_{\delta}}^{\omega,\beta} \xrightarrow[\delta\downarrow 0]{ (d) } \bQ^{\bzeta,\hbeta} \,.
	\]
\end{theorem}

The disordered measure \(\bQ^{\bzeta,\hbeta}\) is referred to as the \emph{continuum long-range polymer model with \(\gamma\)-stable Lévy noise}.
In particular, Theorem~\ref{thm: polymer-1} establishes the conjecture of \cite[\S 2.5.3-(A)]{berger_lacoin_2022} also mentioned in \cite[\S 2.2]{berger_lacoin_2021} (again, the roles of \(\alpha\) and \(\gamma\) are reversed in \cite{berger_lacoin_2021,berger_lacoin_2022}).

\subsection{Disorder relevance and Harris criterion without second moment}
\label{section: Harris}

As mentioned above, the disjunctions of cases we have obtained are linked to the Harris criterion introduced in \cite{harris_1974}. 
Disorder is said to be \textit{irrelevant} if for small enough disorder intensity \(\beta\) the disordered and homogeneous system share the same features; disorder is said to be \emph{relevant} if the disordered system drastically differs from the homogeneous one, for any disorder intensity \(\beta>0\).
Harris criterion~\cite{harris_1974} states that, for random environments with a finite second moment, disorder relevance/irrelevance can predicted through the \textit{correlation length exponent} \(\nu\) of the homogeneous system; let us stress that the \(k\)-point correlation exponent \(\lambda\) should be related to \(\nu\) by the relation \(\nu = \frac{1}{\deff-\lambda}\) (see \cite[Section~1.3]{csz_2016} for a discussion). 

More precisely, Harris' predictions state that, in the \(\mathds{L}^2\) case, disorder should be irrelevant if \( \deff\nu  > 2\) (\textit{i.e.}\ if \(\lambda < \frac{\deff}{2}\)) and relevant if \( \deff\nu < 2\) (\textit{i.e.}\ if \(\lambda < \frac{\deff}{2}\)); the remaining case \(\nu \deff =2\) is dubbed marginal and disorder relevance should depend on finer properties of the model.
This corresponds exactly to the dichotomy of supercriticality/subcriticality described above in the~\(\mathds{L}^2\) case: in particular, disorder is relevant if the disordered system admits a non-trivial, \textit{i.e.}\ disordered, scaling limit.

Harris criterion has been devised for i.i.d.\ disorder with a finite variance and it does not hold anymore in the heavy-tailed case~\eqref{def: omega tail}.
In that case, we recall that subcriticality corresponds to \(\lambda < \frac{\deff}{\gamma}\) and supercriticality to \(\lambda > \frac{\deff}{\gamma}\) (see~\eqref{eq: chaos expansion of the partition function} and the discussion that follows).
We therefore arrive at the following criterion:
\begin{equation}
	\label{eq: Harris heavy-tail}
	\text{disorder is } 
	\begin{cases}
	\, \text{\textit{relevant}}& \text{ if }\ {\displaystyle  \deff \nu  < \frac{\gamma}{\gamma-1} }  \qquad ( \lambda < \frac{\deff}{\gamma})\,,\\[4pt]
	\, \text{\textit{irrelevant}}& \text{ if }\  {\displaystyle \deff\nu > \frac{\gamma}{\gamma-1}}  \qquad ( \lambda > \frac{\deff}{\gamma})\,.		
	\end{cases}
\end{equation}
This criterion has been established in the case of the disordered pinning model of Example~\ref{ex:pinning2} in~\cite{lacoin_sohier_2017} (see also \cite{lacoin_2017} for a marginal case) or for the directed polymer model of Example~\ref{ex:longrange2} (with the simple random walk, so in particular \(\alpha=2\)) in \cite{viveros_2021}.
As far as the construction of a disordered scaling limit is concerned, only the case of the directed polymer has been treated (again with the simple random walk) in~\cite{berger_lacoin_2021}, under the subcritical condition \(\lambda < 1+ \frac{d}{2}\).


Our main result put the criterion~\eqref{eq: Harris heavy-tail} to rigorous ground in a very broad context.
We refer to Section~\ref{section: statistical mechanics} below for a more thorough discussion on the different examples.

\section{Main results: polynomial chaos with heavy-tail disorder}
\label{section: main results}

In this section, we properly state our main results, introducing all the notations we need.
Notice that once stated as in~\eqref{eq: chaos expansion of the partition function}-\eqref{eq: limiting chaos expansion}, the problem we deal with is more general and addresses the question of the scaling limit of polynomial chaos with asymptotically Lévy noise under some \textit{subcriticality} condition. 
We therefore start by defining in Section~\ref{section: subcritical polynomial chaos} continuum polynomial chaos with \(\gamma\)-stable Lévy white noise under a subcriticality condition.
We then state in Section~\ref{section: convergence polynomial chaos} the convergence of discrete polynomial chaos with heavy-tailed random variables to their continuous counterpart.
Then, in Section~\ref{section: convergence statistical mechanics}, we come back to our original motivation and state our results in the context of scaling limit of disordered statistical mechanics systems: we give conditions for the convergence of the partition function and of the disordered Gibbs measure, and we also discuss further properties.

\subsection{Continuum polynomial chaos with Lévy white noise}
{\label{section: subcritical polynomial chaos}}

\subsubsection{Lévy noise and truncated approximation of continuum polynomial chaos}

First of all, let us give a precise definition of a Lévy white noise $\bzeta$ on \(\Omega\). 
Let~$\Lambda$ be a Poisson point process on $\Omega \times \mathbb{R}$, of intensity $\mu(\dd x \dd z) = \dd x \lambda(\dd z)$ where $\lambda (\dd z)$ is the so-called Lévy measure.
Even if the construction is more general, we focus on the $\gamma$-stable case, which corresponds to taking $\lambda(\dd z) = (c_+ \Ind_{z > 0} + c_- \Ind_{z < 0} ) \gamma |z|^{-1-\gamma} \dd z$ where the constants $c_+,c_-$ tune the asymmetry of the noise and verify $c_+,c_-\ge 0$ and $c_+ + c_- =1$.

For $a > 0$, define the truncated noise
\begin{equation}{\label{eq: definition of the general truncated noise}}
	\bzeta^{(a)} := \sum\limits_{ (x,z) \in \Lambda} z \Ind_{|z| > a} \delta_{x} -\kappa(a) \dd x,
\end{equation}
with $\kappa(a) = \int_{\mathbb{R}} z \Ind_{|z| > a} \lambda(dz)$.
When $a \downarrow 0$, $\bzeta^{(a)}$ converges almost surely in any negative local Sobolev space $H^{-s}_{\rm loc}(\Omega)$ with $s >\frac{D}{2}$ to a random distribution denoted $\bzeta$ called $\gamma$-stable Lévy white noise.
We refer to Theorem~\ref{thm: construction of the continuous noise} in appendix for a statement.

\smallskip
Consider a general family of functions $(\bpsi^{(k)})_{k \ge 0}$ such that $\bpsi^{(k)}$ is defined on $\Omega^k$, is symmetric and vanishes on the diagonals.
By convention, $\bpsi^{(0)}$ is a constant.  
The goal is to properly define the following \textit{Lévy chaos expansion} (in analogy with \textit{Wiener chaos expansion} in the Gaussian case), \textit{i.e.} the continuous chaos expansion with respect to a $\gamma$-stable Lévy white noise $\bzeta$:
\begin{equation}{\label{eq: the continuum limit}}
	\bPhi = \sum\limits_{k \ge 0} \frac{1}{k!} \int_{\Omega^k} \bpsi^{(k)}(x_1,\dots,x_k) \prod\limits_{j=1}^k \bzeta(\dd x_j) \,,
\end{equation}

\begin{remark}
Note that we do not require any positivity or monotonicity (on the noise or the functions~$\bpsi$) which were key features of the previous work of Berger and Lacoin \cite{berger_lacoin_2021,berger_lacoin_2022}. 
\end{remark}

The way to define~\eqref{eq: the continuum limit} is via truncation (analogously to what is done in~\cite{berger_lacoin_2022}).
We introduce, for $a > 0$, the truncated chaos
\begin{equation}
	\label{def: Phia}
	\bPhi^{(a)} := \sum\limits_{k \ge 0} \frac{1}{k!} \int_{x_1,\dots,x_k \in \Omega} \bpsi^{(k)}(x_1,\dots,x_k) \prod\limits_{j=1}^k \bzeta^{(a)}(\dd x_j) .
\end{equation}
Let $\cF = (\cF_a)_{a > 0}$ be the filtration, where $\cF_a= \sigma\{\Lambda \cap (\Omega \times \mathbb{R}\setminus (-a,a))\}$. 
We will see below that $(\bPhi^{(a)})_{a > 0}$ is a time-reversed martingale with respect to the filtration $\cF$. 
Therefore, if the martingale is uniformly integrable, it converges almost surely and in $\mathbb{L}^1$. 
Our first main result is to derive conditions for this to hold. 

\subsubsection{Well-posedness of Lévy chaos expansion}

Before we state our result, let us introduce symmetric norms.
Define for $q > 1$, $k \in \bbN$ and $\bff : \Omega^k \to \R$ a symmetric function, the \textit{symmetric} $L^q$ norm:
\begin{equation}{\label{eq: def of symmetric norms}}
	\big\| \bff \|_q = \Bigg(\frac{1}{k!} \int_{\Omega^k} |\bff(x_1,\dots,x_k)|^q \prod\limits_{j=1}^k \dd x_j \Bigg)^{1/q}.
\end{equation}
The space of symmetric functions on $\Omega^k$ with finite symmetric $L^q$ norm is denoted $L^q_s(\Omega^k)$. 

\begin{remark}
It may seem surprising to include a factorial $k!$ in the definition of the $\L^q$ norm but this is a fairly common issue encountered in Malliavin calculus. 
Indeed, consider the case where the white noise $\bzeta$ is Gaussian and $\Omega = (0,1)$. 
Then if $\bff$ is an $L^2$ symmetric function of $(0,1)^k$, the random variable
\begin{equation}
	\label{eq: Xbff}
	\bX(\bff) := \frac{1}{k!} \int_{(0,1)^k} \bff(x_1,\dots,x_k) \prod\limits_{j=1}^k \bzeta(\dd x_j) = \int_{0 < t_1 < \dots < t_k < 1} \bff(t_1,\dots,t_k) \prod\limits_{j=1}^k \bzeta(\dd t_j)
\end{equation}
is well-defined and 
\[
\Es[  \bX(\bff)^2 ] =  \int_{0 < t_1 < \dots < t_k < 1} \bff(t_1,\dots,t_k)^2 \prod\limits_{j=1}^k \dd x_j = \frac{1}{k!} \int_{(0,1)^k} \bff(x_1,\dots,x_k)^2 \prod\limits_{j=1}^k \dd x_j.
\]
Therefore \( \|\bX(\bff) \|_{2} = \|\bff\|_{2}\) in the sense of symmetric norms, and this shows that the condition $\bff \in L^2_s$ is actually necessary for $\bX(\bff)$ to be well-defined. 
\end{remark}

The question of finding a necessary and sufficient condition for the multiple integral $\bX(\bff)$ in~\eqref{eq: Xbff} to be well-defined when $\bzeta$ is a $\gamma$-stable Lévy noise is known to be more delicate than in the Gaussian setting. 
For instance, it is known that the single integral 
\(
\int_{\Omega} \bff(x) \bzeta(\dd x) 
\)
is well-defined if and only if $\int_{\Omega} |\bff(x)|^{\gamma} \dd x < +\infty$ but this does not extend to multiple integrals (see for example \cite{rosinski_woyczynski_1986}).
We will show that a sufficient condition for the $k$-multiple integral $\bX(\bff)$ to be well-defined is $\bff \in L^q_s(\Omega^k)$ for some $q > \gamma$ but this is known to not be optimal (see for example \cite{surgailis_1985}).

\begin{theorem}[Well-posedness of \(\gamma\)-stable Lévy chaos]
	\label{thm: continuous polynomial chaos}
	Let \(\bzeta\) be a \(\gamma\)-stable Lévy white noise with \(\gamma \in (1,2)\).
	Assume that there exist $q\in (\gamma,2]$ such that family of functions $(\bpsi^{(k)})_{k \ge 0}$ verifies the following: for every $C > 0$, 
	\begin{equation}{\label{hyp: infinite radius of convergence for Lq norms}}
		\sum\limits_{k \ge 0} C^k \big\| \bpsi^{(k)} \big\|_{q} < \infty \,.
	\end{equation}
	Then $\Phi^{(a)}$ is well-defined for every $a \in (0,1]$ and $(\bPhi^{(a)})_{a \in (0,1]}$ is a time-reversed martingale with respect to the filtration $\cF =(\cF_a)_{a > 0}$.

	Additionally, for every $p \in (0,\gamma)$ , $(\bPhi^{(a)})_{a \in (0,1]}$ is bounded in $\mathds{L}^p$ and thus admits a limit \(\bPhi := \lim_{a\downarrow 0} \bPhi^{(a)}\) a.s.\ and in \(\mathds{L}^p\).
	Furthermore, for every $p \in (0,\gamma)$, there exist a constant $C_0 = C_0(p,q,\gamma,\Omega)$ such that
	\begin{equation}
		\label{eq: moments constant C0}
		\big\| \bPhi \big\|_p \le \sum\limits_{k \ge 0} (C_0)^k \big\| \bpsi^{(k)} \big\|_{q} \,.
	\end{equation}
\end{theorem}

Let us stress that it is crucial here to have some room between $\gamma$ and the borders of the interval~$[1,2]$: indeed, both limit cases $\gamma =1$ and $\gamma =2$ would require logarithmic corrections. 
Additionally, the bounds on the moments of $\bPhi$ strongly rely on the fact that the domain $\Omega$ is bounded. 
A convincing way to understand why, is the fact that the $L^q$ norm does not dominate the $L^p$ norm if $\Omega$ is unbounded.  

\begin{remark}
	\label{rem: constant C0}
	It may be useful in some application to know how the constant~$C_0=C_0(p,q,\gamma,\Omega)$ appearing in Theorem~\ref{thm: continuous polynomial chaos}-\eqref{eq: moments constant C0} depends on $p$,$q$,$\gamma$ and $\Omega$. 
	One can deduce from the proof that, for $p \in (1,\gamma)$,
	\begin{equation}{\label{eq: the constant C}}
	C_0(p,q,\gamma,\Omega) = \frac{C_1}{(p-1)} \max\{(\gamma-p)^{-1/p},(q-\gamma)^{-1/q}\}\, 
	(|\Omega|\vee 1)^{\frac1p - \frac1q}\,,
	\end{equation}
	where the constant $C_1$ is universal. 
\end{remark}

\subsection{Scaling limit of discrete polynomial chaos}
\label{section: convergence polynomial chaos}

Let us now go back to our initial problem of the scaling limit of the discrete polynomial chaos expansion of the partition function~\eqref{eq: chaos expansion of the partition function}. 
We will now allow for a slightly more general setting.

\subsubsection{Discretization of \(\Omega\) and of \(L^q_s(\Omega)\)}

First, let us give a precise definition of what we mean when we say that $\Omega_{\delta}$ is a discretization of $\Omega$.
For every $\delta \in (0,1)$, $\Omega_{\delta}$ is a finite subset of $\Omega$.
Suppose that to each $\Omega_{\delta}$ is associated a \textit{tesselation} $\cC_{\delta}$ of $\Omega$, that is a map $\cC_{\delta}: \Omega_{\delta} \to \cB(\Omega)$ (with $\cB(\Omega)$ the Borel subsets of $\Omega$) that verifies the following conditions:
\begin{enumerate}{\label{def of a tesselation}}
	\item the ``cells'' $(\cC_{\delta}(x))_{x \in \Omega_{\delta}}$ form a partition of $\Omega$;
	\item for every $x \in \Omega_{\delta}$, $x \in \cC_{\delta}(x)$;
	\item the volume of the cells $\cC_{\delta}(x)$ does not depend on the point $x \in \Omega_{\delta}$ and is denoted $v_{\delta}$. 
\end{enumerate}
We assume that $v_{\delta} \to 0$ as $\delta \downarrow 0$, meaning that the discrete sets $\Omega_{\delta}$ asymptotically approximate $\Omega$.
Note that in Section \ref{section: introduction}, we had assumed $v_{\delta}$ had the specific form $\delta^{\deff}$.
This is actually not necessary and, in particular, relaxing it will allow us to consider more general random walks for the (long-range) directed polymer model (see Section~\ref{section: directed polymer} below). 

Then, for every $\delta \in (0,1)$, any symmetric function $f : \Omega_{\delta}^k \to \bbR$ which vanishes on the diagonals can be extended into a piecewise constant function $\overline{f} : \Omega^k \to \bbR$ by assigning value 
\begin{equation}{\label{eq: piecewise extension of f}}
	\overline{f}(y_1,\dots,y_k) = f(x_1,\dots,x_k),
\end{equation}
for every $(y_1,\dots,y_k) \in \cC_{\delta}(x_1) \times \dots \times \cC_{\delta}(x_k)$. 
Note that $\overline{f}$ is symmetric and vanishes on the diagonals and that its symmetric \(L^q\) norm is given by
\begin{equation}{\label{eq: norm of the piecewise extension of f}}
	\big\| \overline{f} \big\|_q = \bigg(\frac{(v_{\delta})^k}{k!} \sum\limits_{x_1,\dots,x_k \in \Omega_{\delta}} |f(x_1,\dots,x_k)|^q\bigg)^{1/q}. 
\end{equation}

For every $\delta \in (0,1)$, let $(\omega_x)_{x \in \Omega_{\delta}}$ be i.i.d.\ \textit{centered} random variables in the domain of attraction of a $\gamma$-stable law with $\gamma\in (1,2)$. 
In particular (see \cite[XVII.5]{feller_1966}), there exists a slowly-varying function~\(\varphi\)\footnote{A function \(\varphi\) is said to be slowly-varying if \(\lim_{t\to\infty} \frac{\varphi(at)}{\varphi(t)} =1\) for any \(a>0\), see \cite{bingham_goldie_teugels_1987}.} such that for every $t > 0$, 
\begin{equation}{\label{eq: definition of a general heavy-tailed random variable}}
	\Pro[|\omega| > t] = \frac{\varphi(t)}{t^{\gamma}},
\end{equation}
and the following limit exists $c_+ = \lim_{t \to +\infty} \frac{\Pro[\omega \ge t]}{\Pro[|\omega| \ge t]}$ and $c_- = 1 -c_+$.
Note that this is slightly more general than the pure power-law tail introduced in \eqref{def: omega tail}, and also removes the condition that \(\omega>-1\) a.s.\ (we do not require any positivity assumption in the section).

Let us define $V_{\delta} \to +\infty$ up to asymptotic equivalence by
\begin{equation}{\label{eq: general scale of the noise}}
	\Pro[|\omega| > V_{\delta}] \sim v_{\delta} \quad \text{ as } \delta \downarrow 0\,.
\end{equation} 
One can actually show that there is a slowly-varying function $\hat{\varphi}$ such that $V_{\delta} = v_{\delta}^{-1/\gamma} \hat{\varphi}(v_{\delta})$, see \cite[\S1.5.7]{bingham_goldie_teugels_1987}.
(Note that in Section~\ref{section: introduction} we had $V_{\delta} \sim C_0^{-1/\gamma} \delta^{-\deff/\gamma}$.)
The definition~\eqref{eq: definition of the discrete noise} of the discrete noise~$\zeta_{\delta}$ has to be adapted and becomes
\begin{equation}{\label{eq: general definition of the discrete noise}}
	\zeta_{\delta} := \frac{1}{V_{\delta}} \sum\limits_{x \in \Omega_{\delta}} \omega_x \delta_x \,.
\end{equation}

\smallskip
For any \(\delta\in (0,1)\), we let $(\psi_{\delta}^{(k)})_{k \in \bbN}$ be functions defined on $\Omega_{\delta}^k$, symmetric and vanishing on the diagonals.
We then consider the following \textit{discrete} polynomial chaos associated with $(\psi_{\delta}^{(k)})_{k \in \bbN}$, with disorder \(\omega\):
\begin{equation}
	\Phi_{\delta} = \sum\limits_{k \in \bbN} \frac{1}{k!} \sum\limits_{x_1,\dots,x_k \in \Omega_{\delta}} \psi_{\delta}^{(k)}(x_1,\dots,x_k) \prod\limits_{j=1}^k \frac{\omega_{x_j}}{V_{\delta}} = \sum\limits_{k \in \bbN} \frac{1}{k!} \int_{\Omega^k} \overline{\psi}_{\delta}^{(k)}(x_1,\dots,x_k) \prod\limits_{j=1}^k \zeta_{\delta}(\dd x_j) \,,
\end{equation}
where we recall that \(\overline{\psi}_{\delta}\) is the piecewise constant extension of \(\psi_{\delta}\).

Our second main theorem shows that, under relevant conditions on the functions $(\psi_{\delta}^{(k)})_{k \in \bbN}$, the discrete polynomial chaos $\Phi_{\delta}$ converges in law as $\delta \to 0$ to its continuous counterpart.

\begin{theorem}{\label{thm: scaling limit of discrete polynomial chaos}}
Assume that \eqref{eq: definition of a general heavy-tailed random variable} holds with \(\gamma \in (1,2)\) and let \(V_{\delta}\) be defined by \eqref{eq: general scale of the noise}.
Assume that there exist $q\in (\gamma,2]$ such that:
\begin{enumerate}[label={(\roman*)}]
	\item \label{hyp: convergence of correlation functions} For every $k \in \bbN$, there exists some $\bpsi^{(k)}\in L_s^q(\Omega^k)$ such that $\lim\limits_{\delta\downarrow 0} \| \overline{\psi}_{\delta}^{(k)}- \bpsi^{(k)}\|_q=0$;
	
	\item \label{hyp: truncation} 
	The following summability conditions are satisfied: for every $C \ge 0$,
	\begin{equation*}
		\limsup_{\delta \downarrow 0} \sum_{k \ge 0} C^k \| \overline{\psi}_{\delta}^{(k)} \|_{q} < +\infty
		\quad \text{ and } \quad
		\lim\limits_{M \to +\infty} \displaystyle\displaystyle\limsup_{\delta \downarrow 0} \sum\limits_{k > M} C^k \| \overline{\psi}_{\delta}^{(k)} \|_{q} = 0.
	\end{equation*}	
\end{enumerate}
Then we have the following joint convergence in distribution in $H_{\rm loc}^{-s}(\Omega) \times \mathbb{R}$, $s > \frac{D}{2}$
\begin{equation*}
	\left(\zeta_{\delta}, \Phi_{\delta}\right) \xrightarrow{(d)} (\bzeta,\bPhi) \quad \text{ as } \delta\downarrow 0 \,.
\end{equation*}
Furthermore, for every $1 < p < \gamma$ and $\delta \in (0,1)$, for the same constant $C_0 = C_0(p,q,\gamma,\Omega)$ appearing in Theorem \ref{thm: continuous polynomial chaos}-\eqref{eq: moments constant C0}, we have
\(
	\| \Phi_{\delta}\|_p \le \sum\limits_{k \in \bbN} (C_0)^k \|\overline{\psi}_{\delta}^{(k)}\|_q
\)
and $\|\Phi_{\delta}\|_p \to \|\bPhi\|_p$ as $\delta \downarrow 0$.
\end{theorem}

\subsection{Going back to disordered statistical mechanics models}
\label{section: convergence statistical mechanics}

Let us note that the setting of the introduction, \textit{i.e.}\ partition functions of statistical mechanics disordered systems, fits that of Section~\ref{section: convergence polynomial chaos} above: the only specificity is the fact that $c_-=0$ and $c_+=1$ (recall \(\omega >-1\)) and also that the functions \(\psi_{\delta}^{(k)}\) are non-negative.

We have also slightly generalized the setting of Section~\ref{section: introduction} since we can now consider disorder satisfying \eqref{eq: definition of a general heavy-tailed random variable} instead of the pure power-law \eqref{def: omega tail}.
In addition, we can consider a more general rescaling of the correlations functions: let \(J_{\delta}\) be a normalization factor and, for every $k \ge 1$ and $x_1,\dots,x_k \in \Omega_{\delta}$, define
\begin{equation}{\label{eq: general rescaled correlation function}}
	\psi_{\Omega_{\delta}}^{(k)}(x_1,\dots,x_k) := (J_{\delta})^k \,\E_{\delta}^{\mathrm{ref}}\Big[\prod\limits_{j=1}^k \sigma_{x_j}\Big]\,. 
\end{equation} 
From \eqref{eq: chaos expansion for L2 computation}, we can now rewrite the partition function as a polynomial chaos:
\begin{equation}
	Z_{\Omega_{\delta}}^{\omega,\beta} = 1 + \sum\limits_{k \ge 1} \Big( \frac{\beta V_{\delta}}{J_{\delta}} \Big)^k \sum\limits_{x_1,\dots,x_k \in \Omega_{\delta}} \psi_{\Omega_{\delta}}^{(k)}(x_1,\dots,x_k) \prod\limits_{j=1}^k \frac{\omega_{x_j}}{V_{\delta}}\,.
\end{equation}
Then, the condition for \textit{subcriticality} of polynomial chaos can be translated as having $\lim_{\delta \downarrow 0} \frac{V_{\delta}}{J_{\delta}} = +\infty$; the criterion \eqref{eq: Harris heavy-tail} becomes, in this case, that disorder should be \textit{relevant} when $\lim_{\delta \downarrow 0} \frac{V_{\delta}}{J_{\delta}} = +\infty$. 
We now present several results under this subcriticality condition.

\subsubsection{Intermediate disorder for the partition function}

An immediate consequence of Theorems~\ref{thm: continuous polynomial chaos} and~\ref{thm: scaling limit of discrete polynomial chaos} above is that under adequate conditions on the rescaled correlation functions \eqref{eq: general rescaled correlation function}, the partition function~$Z_{\Omega_{\delta}}^{\omega,\beta}$ converges in law as $\delta \to 0$, in the intermediate disorder scaling limit $\beta_{\delta} V_{\delta}/J_{\delta} \to \hbeta \in (0,\infty)$.

\begin{theorem}[Scaling limit of the partition function]
	{\label{thm: convergence of the partition function}}
	Assume that \eqref{eq: definition of a general heavy-tailed random variable} holds with \(\gamma \in (1,2)\) and let \(V_{\delta}\) be defined by \eqref{eq: general scale of the noise}.
	Assume that there exist $q \in (\gamma,2]$ such that the correlations functions~$\psi_{\Omega_{\delta}}^{(k)}$ in~\eqref{eq: general rescaled correlation function} satisfy the hypotheses of Theorem~\ref{thm: scaling limit of discrete polynomial chaos}.
	Then, in the intermediate disorder regime $\lim_{\delta\downarrow 0} \beta_{\delta}\, \frac{V_{\delta}}{J_{\delta}}  = \hbeta\in [0,\infty) $, we have the following joint convergence in distribution:
	\[
	(\zeta_{\delta},Z_{\Omega_{\delta}}^{\omega,\beta_{\delta}}) \xrightarrow{\delta\to 0} (\bzeta,\cZ_{\Omega}^{\bzeta,\hbeta}) \,.
	\]
	Here, $\cZ_{\Omega}^{\bzeta,\hbeta}$ is the \emph{continuum partition function} defined in \eqref{eq: limiting chaos expansion}, which is well-defined thanks to Theorem~\ref{thm: continuous polynomial chaos}.
\end{theorem}

\subsubsection{Intermediate disorder for the disordered Gibbs measure}

Another very natural result is the convergence of the disordered probability \(\P_{\Omega_{\delta}}^{\omega,\beta}\) defined in \eqref{eq: disordered model}. 
For such a result to hold, it makes sense to require that the homogeneous model itself converges in a stronger sense than only through correlation functions. 

Let $E$ be a Polish space such that for every $\delta > 0$, finite fields $\{\sigma_x\}_{x \in \Omega_{\delta}}$ can be viewed as elements of~$E$.
(The pinning model and long-range polymers are such examples, we refer to Section~\ref{section: statistical mechanics} for further details.)
Denote $\cM_1(E)$ the topological space of probability measures on $E$ endowed with the topology of convergence in distribution. 
For any function $G$ defined on $E$, $k \ge 1$ and $x_1,\dots,x_k \in \Omega_{\delta}$, denote
\begin{equation}
	\label{def: generalized correlation functions}
	\psi_{\Omega_{\delta}}^{(k)}(x_1,\dots,x_k;G) := (J_{\delta})^k \: \E_{\delta}^{\mathrm{ref}}\Big[G(\sigma) \prod\limits_{j=1}^k \sigma_{x_j} \Big] \,.
\end{equation}
We denote by \(\mathcal{C}_b(E)\) the set of bounded, continuous functions \(G : E \to \mathbb{R}\).
The following result shows the convergence of the disordered model in the intermediate disorder regime at the level of measures, and in particular construct a \textit{continuum disordered model}.

\begin{theorem}[Scaling limit of the Gibbs measure]
	{\label{thm: convergence of the disordered probability}}
Assume that the conditions of Theorem~\ref{thm: convergence of the partition function} hold.
If furthermore,
\begin{enumerate}[label={(\alph*)}]
	\item \label{cond: invariance principle}
	$\P_{\Omega_{\delta}}^{\mathrm{ref}}$ converges in $\cM_1(E)$ to a probability distribution $\bP_{\Omega}$;
	\item \label{cond: convergence general correlations}
	for every function $G \in \cC_b(E)$ and every $k \in \bbN^*$, $\overline{\psi}_{\Omega_{\delta}}^{(k)}( \cdot,G)$ converges in $L^q_s(\Omega^k)$;
	\item \label{cond: positivite}
	for any $\hbeta \ge 0$, the continuum partition function $\cZ_{\Omega}^{\bzeta,\hbeta}$ is almost surely strictly positive.
\end{enumerate}
Then we have the following joint convergence in distribution in $H_{\rm loc}^{-s}(\Omega) \times \cM_1(E)$, $s > \frac{D}{2}$,
\begin{equation*}
	(\zeta_{\delta}, \P_{\Omega_{\delta}}^{\omega,\beta_{\delta}}) \xrightarrow{(d)} (\bzeta, \bP_{\Omega}^{\bzeta,\hbeta}) \quad\text{ as } \delta\downarrow 0 \,,
\end{equation*} 
where \(\bP_{\Omega}^{\bzeta,\hbeta}\) is the \emph{disordered} version of \(\bP_{\Omega}\); part of the statement is that the measure \(\bP_{\Omega}^{\bzeta,\hbeta}\) is well-defined. 
\end{theorem}

Assumptions~\ref{cond: invariance principle} and~\ref{cond: convergence general correlations} are very natural. 
One could hope that Assumption~\ref{cond: positivite} is actually a consequence of the convergence of $Z_{\Omega_{\delta}}^{\omega,\beta_{\delta}}$ to $\cZ_{\Omega}^{\bzeta,\hbeta}$. 
Unfortunately, we did not find any way to prove this in a general set up.
What one should try to prove is that the event $\{\cZ_{\Omega}^{\bzeta,\hbeta}=0\}$ is actually trivial, \textit{i.e.}\ it has probability $0$ or $1$, so that since $\Es[\cZ_{\Omega}^{\zeta,\hbeta}]=1$ it necessarily has probability $1$. 
This is what is done for instance in~\cite[Proposition~4.4]{berger_lacoin_2022}.
We comment further on this issue in the context of the disordered pinning model and the long-range directed polymer model, in Section~\ref{section: statistical mechanics}.

\begin{remark}[Universality of the scaling limit]{\label{rem: universality}}
	It is remarkable that the limiting object constructed in the above theorems depends only on a few parameters. 
	For instance, $\bP_{\Omega}^{\zeta,\hbeta}$ does not depend on the specific discretization $\Omega_{\delta}$ of $\Omega$ nor on the slowly-varying function~$\varphi$ that appears in the tail~\eqref{eq: definition of a general heavy-tailed random variable} of the discrete environment. 
	In particular, the rescaling $\beta_{\delta}$ depends on these microscopic details but the scaling limit is universal. 
\end{remark}

\subsubsection{About disorder relevance}
\label{subsec: disorder relevance}
One could wonder if it was truly necessary to rescale $\beta$ to have a scaling limit, or if the intermediate disorder scale $\beta_{\delta} \sim \hat \beta J_{\delta}/V_{\delta}$ is truly the relevant scale. 
First of all, let us give the following result, which is a direct corollary of Theorem~\ref{thm: convergence of the partition function} and shows that $\lim_{\delta\downarrow 0} \beta_{\delta} \, \frac{V_{\delta}}{J_{\delta}} = 0$ lies in the so-called \emph{weak disorder} regime.

\begin{corollary}[Weak disorder]
	Assume that the conditions of Theorem~\ref{thm: convergence of the partition function} hold.
	If $\lim_{\delta\downarrow 0} \beta_{\delta} \, \frac{V_{\delta}}{J_{\delta}} = 0$, then $Z_{\Omega_{\delta}}^{\omega,\beta_{\delta}}$ converges to $\cZ_{\Omega}^{\zeta,\hbeta=0} =1$ in \(\mathds{L}^p\) for every $p \in (0,\gamma)$.
\end{corollary}

On the other hand, one would like to show that if $\beta_{\delta} V_{\delta} /J_{\delta}  \to \infty$, then the partition function $Z_{\Omega_{\delta}}^{\omega,\beta_{\delta}}$ converges to \(0\) in probability. 
This property is a simplified instance of disorder relevance. 
As in the case of the strict positivity of the continuum partition function, it actually depends on finer properties of the homogeneous system than simply the scaling form of correlation functions and it is delicate to answer positively with such generality to this question.  
We have decided to include a sufficient condition (verified in practice in many models, including the pinning and directed polymer model, see Section~\ref{section: convergence statistical mechanics} below), to illustrate that for subcritical systems, this property is generic. 

\begin{theorem}[Strong disorder]
	{\label{thm: disorder relevance}}
Assume that $J_{\delta} v_{\delta} \sum_{x \in \Omega_{\delta}} \sigma_x$ converges in distribution under \(\P_{\delta}^{\mathrm{ref}}\) as~\(\delta\downarrow0\), to a strictly positive random variable. 
Then, if $\lim_{\delta\downarrow 0} \frac{\beta_{\delta}V_{\delta}}{J_{\delta}}  = +\infty$, the partition function~$Z_{\Omega_{\delta}}^{\omega,\beta_{\delta}}$ goes to $0$ in probability.
\end{theorem}

Note that one can check that
\[
\E_{\delta}^{\mathrm{ref}}\Big[ J_{\delta} v_{\delta} \sum\limits_{x \in \Omega_{\delta}} \sigma_x \Big ] =  v_{\delta} \sum_{x \in \Omega_{\delta}} \psi_{\Omega_{\delta}}^{(1)}(x) \xrightarrow{\delta\downarrow 0} \int_{\Omega} \bpsi_{\Omega}^{(1)}(x) \dd x \,.
\]
This at least indicates that $J_{\delta} v_{\delta}$ is the correct normalization for the random variable $ \sum_{x \in \Omega_{\delta}} \sigma_x$.

\subsection{Further comments}
\label{section: comments}

\subsubsection*{About the infinite mean case}

In constrast with the approach of \cite{berger_lacoin_2021, berger_lacoin_2022} for the directed polymer model, our approach does not cover the case of a $\gamma$-stable noise with \( \gamma \in (0,1] \).
This is due to the fact that polynomial chaos with respect to such noises would not have any finite $\mathds{L}^p$ moments for $p \ge 1$.  
We leave open the question of finding a general and necessary condition for the analogues of Theorems~\ref{thm: continuous polynomial chaos} and~\ref{thm: scaling limit of discrete polynomial chaos} in this case.

\subsubsection*{The case of a general Lévy noise}
	  
Another restriction compared to \cite{berger_lacoin_2021, berger_lacoin_2022} is that we only consider stable noises instead of general Lévy noises. 
Given a measure $\lambda$ on $\bbR$ satisfying \(\int_{\bbR} (x^2 \wedge 1) \lambda(\dd x) < +\infty\), one could define the Lévy white noise $\bzeta$ on $\Omega$ with Lévy measure $\lambda$ as the almost sure limit as $a \downarrow 0$ of the truncated noise 
\[
\bzeta^{(a)} = \sum\limits_{(x,z) \in \Lambda} z \Ind_{|z| > a} \delta - \kappa(a) \dd x\,,
\]
where $\Lambda$ is the Poisson point process on $\Omega \times \bbR$ of intensity $\dd x \lambda(\dd z)$, and $\kappa(a) = \int_{\bbR} z \Ind_{|z| > a} \lambda(\dd z)$.  
In this case, relevant quantities are the truncated moments of the Lévy measure given by 
\[ m_p(a,b) = \int_{\bbR} |z|^p \Ind_{a < |z| \le b} \lambda(\dd z)\,,\]
and our proofs provide the following generalization of Theorem~\ref{thm: continuous polynomial chaos} (with the same notations):
	
\begin{theorem}{\label{thm: continuous polynomial chaos for general noise}}
	Let $\bzeta$ be a Lévy white noise with Lévy measure $\lambda$ such that there exist $1 < p < q < 2$ such that $m_p(1,+\infty) < +\infty$ and $m_q(0,1) < +\infty$.
	Suppose the family of functions $(\bpsi^{(k)})_{k \ge 0}$ verifies the following: for every constant $C \ge 0$, 
	\begin{equation}
		\sum\limits_{k \ge 0} C^k \big\|\bpsi^{(k)}\big\|_q < +\infty\,.
	\end{equation}
	Then $\Phi^{(a)}$ is well-defined for every $a \in (0,1]$ and $(\bPhi^{(a)})_{a \in (0,1]}$ is a time-reversed martingale with respect to the filtration $\cF =(\cF_a)_{a > 0}$.
		
	Additionally, $(\bPhi^{(a)})_{a \in (0,1]}$ is bounded in $\mathds{L}^{p}$ and thus admits a limit \(\bPhi := \lim_{a\downarrow 0} \bPhi^{(a)}\) a.s.\ and in \(\mathds{L}^p\).
	Furthermore, there exist a universal constant $C_0'$ such that
	\begin{equation}
		\big\| \bPhi \big\|_p \le \sum\limits_{k \ge 0} \Big(\frac{C_0'}{p-1}(m_p(1,+\infty)+m_q(0,1))\Big)^k \big\| \bpsi^{(k)} \big\|_{q} \,.
	\end{equation}
\end{theorem}
We refer to Proposition~\ref{prop: Lp estimates for simple functions} for further details. 
	
A general Lévy noise $\bzeta$ does not satisfy any scale invariance property, and consequently a polynomial chaos with respect to $\bzeta$ cannot arise as the scaling limit of a discrete polynomial chaos with fixed disorder distribution. 
This obstruction can be circumvented by allowing the law of the discrete disorder random variables $\omega$ to depend on $\delta$. 
For instance, for every $\delta > 0$ and every $x \in \Omega_{\delta}$, define 
\[\omega_x^{(\delta)} = V_{\delta} \int_{\cC_{\delta}(x)} \zeta(\dd y)\,,\] 
where the constant $V_{\delta}$ is chosen so that $\Pro[|\omega_x^{(\delta)}| > 1 ] = 1/2 $ (the particular choice of the constants $1$ and $1/2$ is immaterial). 
For every fixed $\delta$, the random variables $(\omega_x^{(\delta)})$ are i.i.d.\ and centered.
It is therefore natural to conjecture that an analogue of Theorem~\ref{thm: scaling limit of discrete polynomial chaos} holds in this setting for Lévy noises satisfying the assumptions of Theorem~\ref{thm: continuous polynomial chaos for general noise}.

\subsubsection*{Relation to Stochastic PDEs}
An important motivation behind the study of continuum partition functions is the construction of solutions to stochastic partial differential equations.
Suppose the homogeneous system posseses a time direction (with finite horizon) and is Markovian i.e.\ $\Omega = (0,T) \times \bbR^d$ for some $T > 0$ (when $d = 0$, $\Omega = (0,T)$) and there is a function $\bpsi:[0,T]\times \bbR^d \times \bbR^d \to \bbR_+$ (interpreted as the density of the transition kernel) such that the continuum correlation functions are given by
\begin{equation*}
	\bpsi^{(k)}\big((t_1,x_1),\dots,(t_k,x_k) \big) = \bpsi(t_1,0,x_1) \bpsi(t_2-t_1,x_1,x_2) \dots \bpsi(t_k-t_{k-1},x_{k-1},x_k)\,,
\end{equation*} 
for every $0 < t_1 < \dots < t_k < T$ and $x_1,\dots,x_k \in \bbR^d$.
Then the point-to-point partition functions defined by the chaos expansions
\begin{equation*}
	\begin{split}
		&\cZ_{st}^{\hbeta}(x,y) = \bpsi(t-s,y-x) \sum\limits_{k \ge 1} \hbeta^k \int_{s < t_1 < \dots < t_k < t} \int_{ x_1,\dots,x_k \in \bbR^d} \bpsi(t_1-s,x,x_1) \bpsi(t_2-t_1,x_1,x_2) \\ 
		& \qquad \qquad \qquad \dots \bpsi(t_k-t_{k-1},x_{k-1},x_k) \bpsi(t-t_k,x_k,y) \prod\limits_{j=1}^k \bzeta(\dd t_j \dd x_j)\,,
		\end{split}
\end{equation*}
formally satisfy the flow equation
\begin{equation}
	{\label{eq: general flow spde}}
	\cZ_{st}^{\hbeta}(x,y) = \bpsi(t-s,x,y) + \hbeta \int_s^t \int_{\bbR^d} \cZ_{su}^{\hbeta}(x,z) \bpsi(t-u,z,y) \bzeta(\dd u \dd z)\,.
\end{equation}
If $\bpsi$ is the Green function of a (time-translation invariant) differential operator $\mathcal{L}$, then formally this is the flow equation associated to 
\begin{equation}{\label{eq: general spde with differential operator}}
	\mathcal{L} u(t,x) = \hbeta \, u(t,x) \cdot \bzeta(t,x)\,, \quad t \ge 0,x \in \bbR^d\,.
\end{equation}
Note that this equation is linear in $u$, but involves multiplicative noise, which may lead to singuliar behavior and ill-posedness.
	
Previous works (\cite{saintloube_1998, balan_2014,balan_ndongo_2017,balan_2023,balan_jimenez_2024}) have addressed equations of the form~\eqref{eq: general spde with differential operator}, primarly in the context of two classical operators: the heat operator $\partial_t - \Delta_x$ and the wave operator $\partial_{tt} - \Delta_x$.
Of particular relevance to the present work is the recent study~\cite{balan_jimenez_2024},
which focuses on symmetric noises ($c_+ = c_-$ in~\eqref{eq: definition of the general truncated noise}) and relies on the Lepage representation. 
Assumptions $1.2$ and $1.3$ in that work are closely related in spirit to the condition~\eqref{hyp: infinite radius of convergence for Lq norms} of Theorem~\ref{thm: continuous polynomial chaos}, as both impose integrability requirements slightly above $L^{\gamma}$.

\subsection{Organisation of the rest of the article}

The remainder of the article is organized as follows. 
In Section~\ref{section: statistical mechanics}, we apply our general results to the two specific systems already introduced in Examples~\ref{ex:pinning2}-\ref{ex:longrange2}. 
Section~\ref{section: Lp estimates} is devoted to the proof of the Theorems~\ref{thm: Lp estimates for discrete chaos} and~\ref{thm: Lp estimates for continuous chaos}, which provide general $\mathds{L}^p$ estimates for discrete and continuous chaos.
These results are then applied in Section~\ref{section: general polynomial chaos} to the proof of our main results (Theorems~\ref{thm: continuous polynomial chaos} and~\ref{thm: scaling limit of discrete polynomial chaos}). 
Section~\ref{section: statistical mechanics proofs} is devoted to the proof of the results stated in Section~\ref{section: convergence statistical mechanics}.
The proofs of the results of Section~\ref{section: statistical mechanics} are presented in Sections~\ref{section: proof pinning}-\ref{section: proofs polymer}. 
Finally, the Appendix collects some further technical results. 
In Appendix~\ref{appendix: noises and Poisson convergence}, we collect results on the discrete noise $\zeta_{\delta}$ and its convergence towards the continuum noise $\bzeta$. 
In Appendix~\ref{appendix: homogeneous partition functions}, we prove a useful inequality that compares a multivariate series with its continuous integral version.

\bigskip

{\noindent \bf Acknowledgements:} The author is deeply grateful to Quentin Berger for suggesting this problem, for his constant support, and for the time he spent reading numerous versions of this manuscript. 
The author also thanks Pierre Faugère and Hubert Lacoin for discussions related to Theorem~\ref{thm: convergence of the pinning measure}.

\section{Applications to specific statistical mechanics models}{\label{section: statistical mechanics}}

Let us now provide more details on how our results of Section~\ref{section: subcritical polynomial chaos} apply in our two guideline examples: the disordered pinning model (from Example~\ref{ex:pinning} and \ref{ex:pinning2}) and the long-range directed polymer model (from Example~\ref{ex:longrange} and \ref{ex:longrange2}).

\subsection{The (disordered) pinning model}
{\label{section: pinning}}

Before we introduce our results fo the pinning model, let us review the definition of the pinning model, which generalizes the setting of Example~\ref{ex:pinning}.

\subsubsection{The homogeneous pinning model and its scaling limit}
\label{subsec: homogeneous pinning}

Let $\tau = (\tau_k)_{k \ge 1}$ be a recurrent and aperiodic renewal process on $\mathbb{N}$ with inter-arrival law given by
\begin{equation}{\label{def: inter-arrival law}}
	\P[\tau_1 = n] := \frac{\ell(n)}{n^{1+\alpha}},
\end{equation}
with $\alpha \in (0,1)$ and $\ell(\cdot)$ a slowly-varying function\footnote{A function \(\ell\) is said to be slowly-varying if \(\lim_{t\to\infty} \frac{\ell(at)}{\ell(t)} =1\) for any \(a>0\), see \cite{bingham_goldie_teugels_1987}.}. 
We may also interpret $\tau$ as a random subset of~$\bbN$. 
Let us now introduce the renewal mass function $u(N) := \P[N \in \tau]$, which plays a crucial role.
In \cite[Theorem B]{doney_1997}, it is proved that 
\begin{equation}{\label{eq: renewal theorem}}
	u(N) := \P[N \in \tau] \underset{N \to +\infty}{\sim} \frac{C_{\alpha}}{\ell(N) N^{1-\alpha}}, \qquad \text{ with } \quad C_{\alpha} = \frac{\alpha \sin(\pi \alpha)}{\pi} \,.
\end{equation}

Let us stress that in \cite[Section A.5.4]{giacomin_2007}, it is shown that the rescaled subset $\frac{1}{N} \tau \subset \bbR_+$ converges in distribution as $N \to +\infty$ to~$\cA_{\alpha}$, the regenerative set of index $\alpha$.
The convergence takes place in the space $\cC_{\infty}$ of closed subsets of $\bbR$, endowed with the Matheron topology, which makes it a Polish space.
The regenerative set $\cA_{\alpha}$ can be defined as the range of the $\alpha$-stable subordinator $(\sigma_t^{(\alpha)})_{t \ge 0}$, characterized by its Laplace transform 
\[
\bE[\expo^{-\lambda \sigma_t^{(\alpha)}}] = \expo^{-t |\lambda|^{\alpha}}, \quad \text{for } \lambda \ge 0,\, t \ge 0 \,,
\]  
that is $\cA_{\alpha} := \overline{\{ \sigma_t^{(\alpha)}, t \ge 0\}}$.

\smallskip
The \textit{homogeneous pinning model} is the probability measure on subsets of $\bbN$ defined for every $h \in \mathbb{R}$ (the pinning parameter) and $N \ge 1$ by
\begin{equation}
	\frac{\dd \P_{N,h}}{\dd \P}(\tau) := \frac{1}{Z_{N,h}} \exp \bigg(h \sum\limits_{n=1}^N \Ind_{n \in \tau} \bigg)\quad \text{ with } \quad Z_{N,h} := \E\bigg[\exp\bigg(h \sum\limits_{n=1}^N \Ind_{n \in \tau}\bigg)\bigg] \,.
\end{equation}
Intuitively, the parameter \(h\) gives a reward (if \(h>0\)) or a penalty (if \(h<0\)) to renewal points in \(\tau\).
A classical reference for this model is \cite[Chapter 2]{giacomin_2007}.

As far as the scaling limit of the homogeneous model is concerned, it is shown in~\cite{sohier_2009} that one needs to rescale the pinning parameter \(h\) as 
\begin{equation}{\label{eq: the scale for h}}
h_N := \frac{\hh}{N u(N)}, \quad \hh \in \bbR \,.
\end{equation}
Then, \cite[Thm 3.1]{sohier_2009} states that when the closed subset $\frac{1}{N} \tau \cap \intg{1}{N}$ is sampled according to $\P_{N,h_N}$ with $h_N$ as in \eqref{eq: the scale for h}, it converges in distribution as $N \to +\infty$ (for the Matheron topology) to a random closed subset $\cB_{\alpha,\hh} \subset [0,1]$.
The law \(\bP_{\alpha,\hh}\) of $\cB_{\alpha,\hh}$ is absolutely continuous with respect to the law of~$\cA_{\alpha}$, with Radon-Nikodym derivative equal to 
$\expo^{\hh L_1^{(\alpha)}}/\bE[\expo^{\hh L_t^{(\alpha)}}]$
where $L_t^{(\alpha)} = \inf \{ s \ge 0, \sigma_s^{(\alpha)} \ge t \}$ is the local time of the subordinator $\sigma^{(\alpha)}$.
In particular, this shows that Assumption~\ref{cond: invariance principle} in Theorem~\ref{thm: convergence of the disordered probability} holds for the homogeneous reference measure $\P_{N,h_N}$.

\smallskip
In order to define the limiting correlation function \(\bpsi_{\hh}^{(k)}\) of the continuum homogeneous pinning model, let us define the \textit{continuum homogeneous partition function} and its \textit{conditioned} analogue. 
For every $t \ge 0$, let
\begin{equation}{\label{eq: def of the continuum homogeneous partition function}}
	\cZ_{t,\hh} := \bE\Big[\expo^{\hh L_t^{(\alpha)}} \Big] \qquad \text{ and } \qquad \cZ_{t,\hh}^c := \alpha\, t^{\alpha -1} \bE\Big[\expo^{\hh L_t^{(\alpha)}} \;\Big|\; t \in \cA_{\alpha} \Big].
\end{equation}
Let us stress that the second definition is only formal since $\bP[t \in \cA_{\alpha}] = 0$.
A consequence of \cite{sohier_2009} is that for every $t \ge 0$, $(Z_{Nt,h_N})_{N \ge 1}$ converges as $N \to +\infty$  to $\cZ_{t,\hh}$.
On the other hand, if one considers the discrete conditioned partition function 
$
Z_{N,h}^c := \E\big[\exp\big(h \sum_{n=1}^N \Ind_{n \in \tau}\big) \Ind_{N \in \tau} \big],
$
then~$\cZ_{t,\hh}^c$ is rigorously defined as the limit of $\P[N \in \tau]^{-1} Z_{Nt,h_N}^{c}$ as $N \to +\infty$. 
We provide a proof of the existence of this limit in Lemma~\ref{lemma: scaling limit of the homogeneous partition functions}, as well as \textit{explicit} expressions of the conditioned and unconditioned continuum homogeneous partition functions in terms of Mittag-Leffler functions.

With the definition \eqref{eq: def of the continuum homogeneous partition function} at hand, the \textit{continuum \(k\)-point correlation function} (for \(k\geq 1\)) is the symmetric function $\bpsi_{\hh}^{(k)} : (0,1)^k \to \bbR $ defined for $t_0 := 0 < t_1 < \dots < t_k < 1$ by 
\begin{equation}{\label{eq: continuum correlation functions of the pinning model}}
	\bpsi_{\hh}^{(k)}(t_1,\dots,t_k) = \frac{1}{\cZ_{1,\hh}} \prod\limits_{j=1}^k \big( \cZ_{t_j-t_{j-1},\hh}^{c} \big) \cZ_{1-t_k,\hh} \,.
\end{equation}

\subsubsection{The disordered pinning model and its scaling limit}

Let $(\omega_n)_{n \ge 1}$ be i.i.d.\ random variables such that \(\omega>-1\) a.s.
Then, for $N \in \bbN$, $\beta \in [0,1]$ and $h \in \bbR$, let us define the disordered pinning measure as follow:
\begin{equation}
	\frac{\dd \P_{N,h}^{\omega,\beta}}{\dd \P_{N,h}}(\tau) := \frac{1}{Z_{N,h}^{\omega,\beta}} \prod\limits_{n=1}^N (1+\beta \omega_n \Ind_{n \in \tau}),
	\quad \text{ with } \quad Z_{N,h}^{\omega,\beta} := \E_{N,h}\Big[\prod\limits_{n=1}^N (1+\beta \omega_n \Ind_{n \in \tau})\Big] \,.
\end{equation}

We can then apply Theorem~\ref{thm: convergence of the partition function} to the partition function \(Z_{N,h}^{\omega,\beta}\) of the pinning model to obtain the following.
We assume that \(\bbE[\omega]=0\) and that \(\Pro[\omega >t] \sim \varphi(t) t^{-\gamma}\) as \(t\to\infty\) for some \(\gamma\in (1,2)\), as in \eqref{eq: definition of a general heavy-tailed random variable}.
Also, as in~\eqref{eq: general scale of the noise}, let $V_N$ be such that 
\begin{equation}{\label{eq: scale of the noise for the pinning}}
	\Pro[\omega > V_N] \sim N^{-1} \quad \text{ as } \quad N \to +\infty ,
\end{equation}
and note that \(V_N\) is regularly varying with index \(1/\gamma\).

\begin{theorem}[Intermediate disorder for the partition function]
	{\label{thm: convergence of the partition function of the pinning model}}
	Let the renewal process $\tau$ satisfy \eqref{def: inter-arrival law} and assume that $1-\alpha <\frac1\gamma$. Note in particular that this guarantees that \(u(N)V_N \to \infty\) as \(N\to\infty\).
	Assume also the following scaling for the parameters \(\beta_N,h_N\): 
	\begin{equation}
		\label{eq: scaling pinning}
		\lim_{N\to\infty} \beta_N \, u(N) V_N = \hbeta  \in (0,+\infty)\,, \qquad  \lim_{N\to\infty} h_N \, u(N) N = \hh \in \bbR \,.
	\end{equation}
	Then we have the following convergence in distribution:
	\begin{equation}{\label{eq: the pinning c.p.f.}}
	Z_{N,h_N}^{\omega,\beta_N} \xrightarrow[\ N\to\infty \ ]{(d)}\cZ_{\hh}^{\bzeta,\hbeta} = 1 + \sum\limits_{k \ge 1} \frac{\hbeta^k}{k!} \int_{(0,1)^k} \bpsi_{\hh}^{(k)}(t_1,\dots,t_k) \prod\limits_{j=1}^k \bzeta(\dd t_j),
	\end{equation}
	where $\bzeta$ is a $\gamma$-stable noise on $(0,1)$ and for every $k \ge 1$, and $\bpsi_{\hh}^{(k)}$ is defined in \eqref{eq: continuum correlation functions of the pinning model}. 
	This convergence holds jointly with the convergence of the discrete noise to $\bzeta$.  
	Furthermore, for every $p \in [1,\gamma)$, $\cZ_{\hh}^{\bzeta,\hbeta}$ has a finite moment of order~$p$, and 
	$\lim_{N \to +\infty} \Es[(Z_{N,h_N}^{\omega,\beta_N})^p] = \Es[(\cZ_{\hh}^{\bzeta,\hbeta})^p]$.
\end{theorem}


Note that the multiple integrals appearing in \eqref{eq: the pinning c.p.f.} can be expressed as iterated stochastic integrals as follows.
Define the (spectrally positive) $\gamma$-stable Lévy process $(\bL_t)_{t \ge 0}$ by setting \(\bL_t = \int_{0}^t \bzeta(\dd t)\) for every $t \ge 0$. 
Then, we can write 
\begin{equation*}
	\cZ_{\hh}^{\bzeta,\hbeta} = 1 + \sum\limits_{k \ge 1} \hbeta^k \int_{0 < t_1 < \dots < t_k < 1} \bpsi_{\hh}^{(k)}(t_1,\dots,t_k) \prod\limits_{j=1}^k \dd \bL_{t_j} \,,
\end{equation*}
where the integrals over the simplex $\{0 < t_1 < \dots < t_k < 1 \}$ can be viewed as iterated Itô integrals with respect to the process $\bL$. 

\begin{remark}
Let us stress that there is an alternative representation of the continuum partition function, closer to the limit obtained in \cite[Thm 3.1]{csz_2016}, given by: 
\begin{equation}{\label{eq: alternative representation of the pinning c.p.f.}}
\cZ_{\hh}^{\bzeta,\hbeta} = (\cZ_{1,\hh})^{-1} \bigg(1 + \sum\limits_{k \ge 1} \frac{1}{k!} \int_{(0,1)^k} \prod\limits_{j =1}^k \bphi^{(k)}(t_1,\dots,t_k) \prod\limits_{j =1}^k (\hbeta \bzeta(\dd t_j) + \hh \dd t_j)\bigg), 
\end{equation}
where for every $k \ge 1$, $\bphi^{(k)}$ is a symmetric function (corresponding to the \(hh=0\) continuum homogeneous pinning model, \textit{i.e.}\ $\bphi^{(k)}= \bpsi_{\hh=0}^{(k)}(t_1,\dots,t_k)$), defined for $0=: t_0 < t_1 < \dots < t_k < 1$ by 
\begin{equation*}
\bphi^{(k)}(t_1,\dots,t_k) := \prod_{i=1}^{k} (t_i-t_{i-1})^{-(1-\alpha)} \,. 
\end{equation*}
The idea is simply to expand the product $\prod_{j=1}^k (\hbeta \bzeta(\dd t_j) + \hh \dd t_j)$ in \eqref{eq: alternative representation of the pinning c.p.f.} and regroup the integration variables with respect to $\dd t_j$.
Using the explicit expression of the continuum homogeneous partition functions (see Lemma \ref{lemma: scaling limit of the homogeneous partition functions}), one reconstructs the correlation functions \eqref{eq: continuum correlation functions of the pinning model} and recovers \eqref{eq: the pinning c.p.f.}. 
We refer to \cite[\S5.1]{berger_legrand_2024}, where this procedure is explained in detail in a very similar context.
\end{remark}

We interpret here the discrete pinning measure \(\P_{N,h}^{\omega,\beta}\) as a measure on the space $\cC_{\infty}$ of closed subsets of $\bbR$, endowed with the Matheron topology, by considering the closed subset $\frac{1}{N} \tau$.
Let us denote \(\mathcal{M}_1(\cC_{\infty})\) the set of probability measures on \(\cC_{\infty}\), endowed with the topology of the convergence in distribution.

\begin{theorem}[Intermediate disorder for the disordered pinning]
	{\label{thm: convergence of the pinning measure}}
	Under the same assumptions as Theorem \ref{thm: convergence of the partition function of the pinning model}, for every $\hbeta \in [0,\infty)$ and \(\hh \in \bbR\), there exists a random probability measure $\bQ_{\hh}^{\bzeta,\hbeta}$ on $\cC_{\infty}$ such that the following convergence in distribution holds in $\cM_1(\cC_{\infty})$,
	\begin{equation*}
		\P_{N,h_N}^{\beta_N,\omega} \xrightarrow[N \to +\infty]{(d)} \bQ_{\hh}^{\bzeta,\hbeta}\,.
	\end{equation*}
	This convergence holds jointly with the convergence of the discrete noise to $\bzeta$. 
	We refer to the measure~$\bQ_{\hh}^{\bzeta,\hbeta}$ as the \emph{continuum disordered pinning model with \(\gamma\)-stable Lévy noise}.
\end{theorem}

The proof of this result follows essentially from Theorem~\ref{thm: convergence of the partition function of the pinning model} and existing literature on the pinning model.
The only additional ingredient is the strict positivity of the continuum partition function \(\cZ_{\hh}^{\bzeta,\hbeta}\), which is necessary to apply our general Theorem~\ref{thm: convergence of the disordered probability}. 
This delicate problem is solved in~\cite{faugere_lacoin}. 

Let us conclude by mentioning the following results in the cases where in \eqref{eq: scaling pinning} one has \(\hat{\beta}=0\) (weak disorder) or \(\hbeta=+\infty\) (strong disorder).

\begin{theorem}{\label{thm: disorder relevance for the pinning model}}
	Let the aperiodic and recurrent renewal process $\tau$ satisfy \eqref{def: inter-arrival law} for some $\alpha \in (\frac{\gamma-1}{\gamma},1)$.
	Assume that $\lim_{N\to\infty} h_N u(N)N = \hh \in \bbR$, as in~\eqref{eq: scaling pinning}.
	\begin{enumerate}
		\item If $\lim_{N \to +\infty} \beta_N u(N) V_N = 0$, then $Z_{N,h_N}^{\omega,\beta_N}$ goes to $1$ in probability (and in $\mathds{L}^p$ for every $1 \le p < \gamma$).
		
		\item If $\lim_{N \to +\infty} \beta_N u(N) V_N = +\infty$, then $Z_{N,h_N}^{\omega,\beta_N}$ goes to $0$ in probability.
	\end{enumerate}
\end{theorem}

\noindent
In particular, the above shows that the disorder relevance properties discussed in Section \ref{subsec: disorder relevance} hold for the disordered pinning model.

\subsubsection{Further remarks and perspectives}
\label{subsec: remark SPDE pinning}

The point-to-point disordered partition function is defined for every $ 1 \le M \le N$ by
\begin{equation*}
	Z_{M,N,h}^{\omega,\beta,c} := \E_{N,h}\Big[\prod\limits_{n=M+1}^N (1+\beta \omega_n \Ind_{n \in \tau}) \Ind_{ N \in \tau} \mid M \in \tau \Big] \,.
\end{equation*}
Our main result also applies to this case, and we can show that if  $\lim_{N \to +\infty} \beta_N u(N) V_N = \hbeta \in \bbR_+$ and $\lim_{N \to +\infty} h_N u(N) N = \hh \in \bbR$, then for every $0 \le s \le t \le 1$, \(u(N)^{-1} Z_{Ns,N t,h_N}^{\omega,\beta_N,c}\) converges in distribution to a random variable denoted \(\cZ_{s,t,\hh}^{\hbeta}\).
In fact, one can show that 
\[ \big(u(N)^{-1} Z_{\floor{Ns},\floor{Nt},h_N}^{\omega,\beta_N,c}\big)_{0 \le s \le t \le 1} \xrightarrow[N \to +\infty]{(d)} \big(\cZ_{s,t,\hh}^{\hbeta}\big)_{0 \le s \le t \le 1} \,,\]
in the sense of finite-dimensional distributions.
It would be interesting to strengthen this convergence to convergence in distribution in a suitable functional space, as in \cite{csz_2016_continuum}. 
Since the limit is expressed as a polynomial chaos with respect to a Lévy white noise, a natural candidate is the space $\mathcal{D}([0,1]^{2,\le},\mathbb{R}_+)$, consisting of functions defined on the simplex $\{ 0 \le s \le t \le 1 \}$ that are càdlàg in both variables.

This question is motivated by the fact that when $\hh = 0$, the limiting random process is a candidate solution to the following \textit{stochastic Volterra flow equation}: 
\[ 
	\cZ_{s,t}^{\hbeta}= (t-s)^{-(1-\alpha)} + \hbeta \int_s^t (t-u)^{-(1-\alpha)} \cZ_{u,t}^{\hbeta} \bzeta(\dd u),
\]
where $\bzeta$ is the $\gamma$-stable Lévy white noise.
Inspired by the recent work~\cite{csz_2023}, a major breakthrough was achieved for this equation by \cite{wei_yu_2024} in the case where $\bzeta$ is Gaussian and in the critical regime $\alpha = 1/2$. 
We refer to their introduction for further background and references.

\subsection{The long-range directed polymer}
{\label{section: directed polymer}}
Before we introduce our results fo the directed polymer model, let us review our assumptions on the underlying random walk, which slightly generalizes the setting of Example~\ref{ex:longrange}.

\subsubsection{The long-range random walk and its scaling limit}
\label{subsec: longrange RW}

We consider a random walk $S = (S_{n})_{n\geq 0}$ on~\(\bbZ^{d}\), \(d\geq 1\), that belongs to the domain of attraction of an $\alpha$-stable law, with $\alpha \in (0,2]$, without centering.
In other words, we assume that there exists a sequence \((a_n)_{n\geq 1}\) such that $(\frac{1}{a_n} S_n)_{n\geq 0}$ converges in distribution as $n \to +\infty$ to a \textit{non-degenerate} \(d\)-dimensional $\alpha$-stable law~$\bX_1$. (We say that the law of \(\bX_1\) is non-degenerate if it cannot be supported on any proper subspace of \(\mathbb{R}^d\).) 
For simplicity, we assume in the following that the random walk is aperiodic; this excludes for instance the case of the simple random walk (considered in \cite{berger_lacoin_2021}) but makes the exposition much easier.

The characterization of \(d\)-dimensional domains of attractions is very well-known in dimension \(d=1\), see the seminal reference \cite[IX.8]{feller_1966}, and in dimension \(d\geq 2\), this dates back to \cite{dHOR_1984,resnick_greenwood_1979}; we refer to \cite{resnick_2004} for a general overview.
In particular, the scaling sequence \((a_N)_{N\geq 1}\) is necessarily regularly varying with exponent \(1/\alpha\); the case \(\alpha=2\) corresponds to the Gaussian case.
Note that since we consider a domain of attraction without centering, this imposes that \(\E[S_1]=0\) when \(\alpha \in (1,2]\); the case \(\alpha=1\) is more delicate.

Let us stress that we also have a Donsker's invariance principle as well as a local limit theorem for \((S_N)_{N\geq 0}\). 
This will ensure in particular that the conditions \eqref{cond: invariance principle} and \eqref{cond: convergence general correlations} of Theorem~\ref{thm: convergence of the disordered probability} hold.

Let us start with a statement of the local limit theorem, see e.g.\ \cite{doney_1991,griffin_1986}: under the aperiodicity condition, we have 
\begin{equation*}
	\sup\limits_{x \in \bbZ^{d}} \Big| (a_n)^{d}\, \P(S_n=x) - g_1\Big(\frac{x}{a_n}\Big) \Big| \xrightarrow[\ n \to +\infty\ ]{} 0 \,,
\end{equation*} 
where \(g_1(\cdot)\) is the density of \(\bX_1\).

As far as the invariance principle is concerned, let us introduce $\cD:=\cD([0,1],\bbR^{d})$ the space of càdlàg functions with values in $\bbR^{d}$, endowed with the Skorokhod topology (also called the $J_1$ topology).
For every $N \ge 1$ and $t \in [0,1]$, we consider the rescaled random walk \(S_t^{(N)} = \frac{1}{a_N} S_{\floor{N t}}\), where \((S_t^{(N)})_{t\in [0,1]}\) is interpreted as an element of~\(\cD\).
Then, the invariance principle (see e.g.\ \cite[Thm.~7.1]{resnick_2007}) tells that we have the following convergence in distribution in~\(\cD\):
\[
(S_t^{(N)})_{t\in [0,1]} \xrightarrow[\ N\to\infty\ ]{(d)} (\bX_t)_{t\in [0,1]} \,,
\]
where $(\bX_t)_{t \in [0,1]}$ is a \(d\)-dimensional $\alpha$-stable process.
For every $t \in (0,1]$, we denote by \(g_t(\cdot)\) the density of $\bX_t$: it is given by $g_t(x) = t^{-d/\alpha} g_1(x\, t^{-1/\alpha})$.  

In particular, with this at hand, the \textit{continuum \(k\)-point correlation function} (for \(k\geq 1\)) is the symmetric function $\bpsi^{(k)}:\big((0,1)\times \bbR^{d}\big)^k \to \bbR$ defined by
\begin{equation}
	{\label{eq: continuum correlation functions of the polymer model}}
	\bpsi^{(k)}\big((t_1,x_1),\dots,(t_k,x_k)\big) = \prod\limits_{j=1}^k g_{t_j-t_{j-1}}(x_j-x_{j-1})
\end{equation} 
for $t_0 := 0 < t_1 < \dots < t_k < 1 $ and $x_0 := 0, x_1,\dots,x_k \in \bbR^{d}$.

\subsubsection{The long-range directed polymer model and its scaling limit}

We are now ready to introduce the disordered version of the model.
Let $(\omega_{n,x})_{n \in \bbN, x \in \bbZ^{d}}$ be i.i.d.\ and centered heavy-tailed random variables, satisfying \eqref{eq: definition of a general heavy-tailed random variable} for some \(\gamma\in (1,2)\); assume also that \(\omega>-1\) a.s. 
Then, the (long-range) directed polymer model is the probability measure defined as a Gibbs modification of the law of the random walk \((S_n)_{n\geq 1}\) defined above: for every $N \in \mathbb{N}$ and $\beta \in [0,1]$, we let
\begin{equation}
	\frac{\dd \P_N^{\omega,\beta}}{\dd \P}(S) = \frac{1}{Z_N^{\omega,\beta}} \prod\limits_{n=1}^N (1+ \beta \omega_{n,S_n}) \,,
	\quad \text{ with } \quad
	Z_N^{\omega,\beta} = \E \Big[\prod\limits_{n=1}^N (1+\beta \omega_{n,S_n}) \Big] \,.
\end{equation}

Let us define $V_N$ up to asymptotic equivalence by 
\begin{equation}{\label{eq: scale of the noise for polymers}}
	\Pro[|\omega| > V_N] \sim \frac{1}{N\, (a_N)^{d}} \quad \text{ as } N \to +\infty \,,
\end{equation}
and note that \(V_N\) is regularly varying with index \(\frac1\gamma ( 1+ \frac{d}{\alpha})\).

Let us mention that we cannot apply directly Theorems~\ref{thm: convergence of the partition function} and~\ref{thm: convergence of the disordered probability} to obtain results on the directed polymer model, since the ambient space is here \(\Omega = (0,1) \times \mathbb{R}^{d}\), which is not bounded. 
We are however able to prove the following results.

\begin{theorem}[Intermediate disorder for the partition function]
	{\label{thm: convergence of the polymer partition function}}
	Let $S=(S_{n})_{n\geq 0}$ be a random walk on~$\bbZ^{d}$ satisfying the assumptions of Section~\ref{subsec: longrange RW}, and assume that $\gamma < 1+\frac{\alpha}{d}$. 
	Note in particular that it guarantees that \( V_N (a_N)^{-d}\to \infty\) as \(N\to\infty\).
	Assume also the following scaling for the parameter~\(\beta_N\): 
	\begin{equation}
		\label{eq: scaling polymer}
	 \lim_{N\to\infty} \beta_N V_N (a_N)^{-d}= \hbeta \in [0,\infty) \,.
	\end{equation}
	Then, we have the following convergence in distribution:
	\begin{equation}
	Z_N^{\omega,\beta_N} \xrightarrow[\ N\to\infty\ ]{(d)}	
	\cZ^{\bzeta,\hbeta} = 1 + \sum\limits_{k \ge 1} \frac{\hbeta^k}{k!} \int_{(0,1)^k \times (\bbR^{d})^k} \prod\limits_{j=1}^k  \bpsi^{(k)}\big((t_1,x_1),\dots,(t_k,x_k)\big) \prod\limits_{j=1}^k \bzeta(\dd t_j \dd x_j) \,,
	\end{equation}
	where $\bzeta$ is a $\gamma$-stable noise on $(0,1) \times \bbR^{d}$ and \(\bpsi^{(k)}\) is defined in~\eqref{eq: continuum correlation functions of the polymer model}.
	This convergence holds jointly with the convergence of the discrete noise to $\bzeta$.
	Furthermore, for every $p \in [1,\gamma)$, $\cZ^{\bzeta,\hbeta}$ has a finite moment of order~$p$, and 
	$\lim_{N \to +\infty} \Es[(Z_{N}^{\omega,\beta_N})^p] = \Es[(\cZ^{\bzeta,\hbeta})^p]$.
\end{theorem}

Note that the polymer measure can be viewed as a probability measure on $\cD:=\cD([0,1],\bbR^{d})$ by setting: for every continuous and bounded function $G:\mathcal{D} \to \bbR$, 
\begin{equation*}
	\E_N^{\omega,\beta}[G] := \frac{1}{Z_N^{\omega,\beta}} \E\Big[ G(S^{(N)}) \prod\limits_{n=1}^N(1+\beta \omega_{n,S_n}) \Big]\,.
\end{equation*}
We then have the following convergence for the directed polymer measure towards a continuum object.
 
\begin{theorem}[Intermediate disorder for the long-range directed polymer]
	{\label{thm: convergence of the polymer measure}}
	Under the same assumptions as Theorem \ref{thm: convergence of the polymer partition function}, for every $\hbeta \in [0,\infty)$, there exists a random probability measure $\bQ^{\bzeta,\hbeta}$ on $\mathcal{D}$ such that the following convergence in distribution holds in $\cM_1(\mathcal{D})$,
	\begin{equation*}
		\P_{N}^{\omega,\beta_N} \xrightarrow[\ N \to +\infty\ ]{(d)} \bQ^{\bzeta,\hbeta}\,.
	\end{equation*}
	This convergence holds jointly with the convergence of the discrete noise to $\bzeta$. 
	We refer to the measure~$\bQ^{\bzeta,\hbeta}$ as the \emph{continuum long-range directed polymer model with \(\gamma\)-stable Lévy noise}.
\end{theorem}

Let us mention that Theorems~\ref{thm: convergence of the polymer partition function}-\ref{thm: convergence of the polymer measure} are the analogous of Theorems~2.1 and~2.4 in \cite{berger_lacoin_2021}, where the results are obtained in the case of the simple random walk (which corresponds to $\alpha = 2$). 
Our results are therefore more general, and answer completely the questions of \cite[Sec.~2.2]{berger_lacoin_2021} (note that the roles of \(\gamma\) and \(\alpha\) are reversed).

Finally, let us conclude by mentioning the following results in the cases where in \eqref{eq: scaling polymer} one has \(\hat{\beta}=0\) (weak disorder) or \(\hbeta=+\infty\) (strong disorder).

\begin{theorem}{\label{thm: disorder relevance for the polymer}}
	Let the random walk \((S_n)_{n\geq 0}\) satisfy the assumptions of Section~\ref{subsec: longrange RW} and assume that \(1+ \frac{\alpha}{d} > \gamma-1\).	
	\begin{enumerate}
		\item \label{res: disorder relevance on the left} 
		If $\lim_{N \to +\infty} \beta_N V_N (a_N)^{-d} = 0$, then $Z_N^{\omega,\beta_N}$ goes to $1$ in probability (and in $L^1$).
		
		\item \label{res: disorder relevance on the right} 
		If $\lim_{N \to +\infty} \beta_N V_N (a_N)^{-d} = +\infty$, then $Z_N^{\omega,\beta_N}$ goes to $0$ in probability. 
	\end{enumerate}
\end{theorem}

\noindent
This shows that the disorder relevance properties discussed in Section \ref{subsec: disorder relevance} hold for the long-range directed polymer model.

\subsubsection{Further remarks and perspectives}
\label{subsec: remark SPDE polymer}

For this model, the point-to-point partition function is defined for every $1 \le M \le N$ and $x,y \in \bbZ^d$ by
\[ Z_{M,N}^{\omega,\beta,c}(x,y) := \E\Big[\prod\limits_{n=M+1}^N (1+\beta \omega_{n,S_n}) \Ind_{ S_N = y} \mid S_M = x \Big] \,.\]
For $x \in \bbR^d$, we denote by $\floor{x}$ a nearest neighbor of $x$ in $\bbZ^d$.
Our techniques adapt to prove that there exist a process $\big(\cZ_{s,t}^{\bzeta,\hbeta}(x,y)\big)_{0 \le s \le t \le 1\,; \, x,y \in \bbR^d}$ such that 
\begin{equation}{\label{eq: process convergence of the flow}}
	(a_N)^{-1} Z_{\floor{Ns},\floor{Nt}}^{\omega,\beta_N,c}\big(\floor{a_N x}, \floor{a_N y}\big) \xrightarrow[N \to +\infty]{(d)} \cZ_{s,t}^{\bzeta,\hbeta}(x,y)\,,
\end{equation} 
in the sense of finite-dimensional distributions.
The limiting process is a candidate solution for the flow associated with the following \textit{stochastic fractional heat equation}  
\begin{equation}{\label{eq: stochastic fractional heat equation}}
	\partial_t u(t,x) = \mathcal{L}_{\alpha} u(t,x) + \hat\beta \, \bzeta(t,x) \cdot u(t,x)\,, \quad t \ge 0, x \in \bbR^d \,.
\end{equation}
Here, $\mathcal{L}_{\alpha}$ denotes the fractional Laplacian of index $\alpha$, and $\bzeta$ is a $\gamma$-stable noise. 
When $\alpha = 2$, $\mathcal{L}_{\alpha}$ reduces to the standard Laplacian  $\Delta = \sum_{j=1}^d \partial_{x_j,x_j}$, and~\eqref{eq: stochastic fractional heat equation} becomes the stochastic heat equation (SHE) with multiplicative noise. 
This equation has been extensively studied when the noise $\bzeta$ is Gaussian.
In this case, the Cole-Hopf transform of the solution of the SHE is expected to solve the KPZ equation, a singular SPDE which modeling the evolution of random growth fronts; recent breakthroughs were achieved in this direction~\cite{csz_2020,csz_2023}.
Another important motivation for the study of the SHE is the phenomenon of \textit{localization}, or \textit{intermittency}~ \cite{intermittency_in_random_media,carmona_molchanov_1994}. 
Informally, this refers to the appearance of sharp peaks at large times in the random field $(u(t,x))_{t \ge 0, x \in \bbR^d}$, where $u$ is a solution of~\eqref{eq: stochastic fractional heat equation}. 

Early results for the SHE driven by Lévy noise were obtained in~\cite{saintloube_1998}, but relied on assumptions on the Lévy measure that excluded $\gamma$-stable noise.
In the $\gamma$-stable case,  existence,uniqueness and intermettency were established in ~\cite{berger_chong_lacoin_2023} (see also \cite{chong_kevei_2023}).

As already mentioned for the stochastic Volterra equation, it would be interesting to strengthen the convergence~\eqref{eq: process convergence of the flow} to a convergence in a suitable functional space. 
We again expect càdlàg regularity in time. 
Spatial regularity is a more delicate and less explored issue, particularly in view of intermittency. 
We refer to~\cite{chong_dalang_humeau_2018} for related results in this direction.

\section{Important \texorpdfstring{$\mathds{L}^p$}{Lp} estimates for chaos}{\label{section: Lp estimates}}

In this section we prove useful and general moment inequalities for discrete and continuous chaos. 
They illustrate in particular why we needed to introduce symmetric norms in Section~\ref{eq: def of symmetric norms}.
All the setting and notations are the same as in Section~\ref{section: main results}.

A function $\bff : \Omega^k \mapsto \mathbb{R}$ is said to vanish on the diagonals if $\bff(x_1,\dots,x_k) = 0$ whenever there are two indices $i \neq j$ such that $x_i = x_j$. 
We now state two results for chaos based on symmetric functions that vanish on the diagonals: one for the discrete chaos and one for the continuous chaos. 
In both cases, the meaning of the inequalities we state is that the left-hand side is well-defined if the right-hand side is finite, and the inequality holds when this is the case. 
We will then apply these inequalities to \(k\)-point correlation functions~\(\frac{1}{k!}\bpsi^{(k)}\).

\begin{theorem}{\label{thm: Lp estimates for discrete chaos}}
	Assume that \(\omega\) is centered and verifies~\eqref{eq: definition of a general heavy-tailed random variable} with \(\gamma \in (1,2)\).  
	For every $p \in (1,\gamma)$ and $q \in (\gamma,2]$, there exists a constant $C_0=C_0(p,q,\gamma,\Omega)$ such that, for every $k \ge 1$,~$\delta \in (0,1)$ and $f: \Omega_{\delta}^k \to \mathbb{R} $ which is symmetric and vanishes on the diagonal:
	\begin{equation}{\label{eq: Lp estimates for discrete chaos}}
		\bigg\| \frac{1}{k!} \sum\limits_{x_1,\dots,x_k \in \Omega_{\delta}} f(x_1,\dots,x_k) \prod\limits_{j=1}^k \frac{\omega_{x_j}}{V_{\delta}}\bigg\|_p \le (C_0)^k \big\| \overline{f} \big\|_q \,.
	\end{equation}
	Here, we recall that \(V_{\delta}\) is defined (up to asymptotic equivalence) by the relation~\eqref{eq: general scale of the noise} and that $\overline{f}$ is the piecewise extension of \(f\) as introduced in \eqref{eq: piecewise extension of f}. 
\end{theorem}

We refer to Remark~\ref{rem: constant C0} (see~\eqref{eq: the constant C}) for the way that the constant \(C_0\) depends on \(p,q,\gamma,\Omega\).

\begin{theorem}{\label{thm: Lp estimates for continuous chaos}}
	Assume that \(\zeta\) is a \(\gamma\)-stable Lévy white noise, with \(\gamma\in (1,2)\).
	For every $p \in (1,\gamma)$ and $q \in (\gamma,2]$, the constant $C_0$ of Theorem \ref{thm: Lp estimates for discrete chaos} is such that for every $ k \ge 1$, $a \in (0,1]$ and $\bff: \Omega^k \mapsto \mathbb{R}$ which is symmetric and vanishes on the diagonals, 
	\begin{equation}{\label{eq: Lp estimates for continuous chaos}}
		\bigg\|\frac{1}{k!} \int_{\Omega^k} \bff(x_1,\dots,x_k) \prod\limits_{j=1}^k \bzeta^{(a)}(\dd x_j) \bigg\|_p \le (C_0)^k \|\bff\|_q. 
	\end{equation}
\end{theorem}

\subsection{Decoupling inequalities and some consequences}{\label{subsec: decoupling inequalities}}

A crucial tool for the proof of Theorems~\ref{thm: Lp estimates for discrete chaos} and \ref{thm: Lp estimates for continuous chaos} is so-called \textit{decoupling inequalities}.
We refer to \cite{delapena_gine_2000} for an overview of these techniques.
For pedagogical reasons, we recall the main definitions and inequalities that we use in our setting.
In this section only, instead of working on a general summation index $\Omega_{\delta}$, we sum over sets of the form $\intg{1}{N} = \{1,\ldots, N\}$, where $N$ is the number of points in $\Omega_{\delta}$.
Also, instead of working with symmetric functions null on the diagonals, we use tetrahedral functions, \textit{i.e.}\ functions defined on \((n_1,\ldots, n_k)\) with \(n_1<\ldots <n_k\).
The main result of this section (Theorem~\ref{thm: Lp decoupling inequality} below) can then easily be transferred to the setting of Theorem~\ref{thm: Lp estimates for discrete chaos}, the estimates being uniform in~$N$ --- this is done in Section~\ref{section: Lp discrete chaos}.
The adaptation to continuous chaos is carried out in Section~\ref{section: Lp continuous chaos}.

\subsubsection{Statement of the result}

Let $k \ge 1$ and for every $1 \le j \le k$, let $\omega^{(j)} = (\omega_n^{(j)})_{n \ge 1}$ be a sequence of independent and centered random variables.
We do not assume that the sequences $\omega^{(j)}$ are independent or identically distributed, but we require that they are jointly \textit{level independent}.
This means that, if for \(n\in \bbN\) one lets $\cL_n = \sigma(\omega_n^{(j)}, 1 \le j \le k)$ be the ``level \(n\)'' \(\sigma\)-algebra, then $(\cL_n)_{n \ge 1}$ forms an independent family. 
For example, let $\omega=(\omega_n)_{n \ge 1}$ be a sequence of independent random variables and define $\omega^{(j)} = \omega$ for every $1 \le j \le k$: then, the sequences \(\omega^{(j)}\) are identical, but level independent nevertheless. 

For any $k \ge 1$ and $N \ge 1$, let $\cT_N^k$ denote the set of functions defined on the tetrahedra 

\[\{ (n_1,\dots,n_k) \in \intg{1}{N}\,, 1 \le n_1 < \dots < n_k \le N\}.
\]
For $f \in \cT_N^k$, define 
\begin{equation*}
X_N(f; \omega^{(1)},\dots,\omega^{(k)}) := \sum\limits_{1 \le n_1 < \dots < n_k \le N} f(n_1,\dots,n_k) \prod\limits_{j=1}^k \omega_{n_j}^{(j)}.
\end{equation*}

Decoupling inequalities compare the law of the random variable $X_N(f;\omega^{(1)},\dots,\omega^{(k)})$ with the law of the same random variable where the random sequences $\omega^{(1)},\dots,\omega^{(k)}$ are replaced by a version where they are now independent.  

\begin{theorem}[\(\mathds{L}^p\) decoupling inequality]
	{\label{thm: Lp decoupling inequality}}
	There exists a constant $C_1$ such that for every $ p \in (1,2]$, for any $k \ge 1$, $N \ge 1$ and any tetrahedral function $f \in \cT_N^k$,
	\begin{equation}
	\norme{X_N(f; \omega^{(1)},\dots,\omega^{(k)})}_p \le \Big( \frac{C_1}{p-1} \Big)^k \norme{X_N(f; \tomega^{(1)},\dots,\tomega^{(k)})}_p,
	\end{equation}
	where the random sequences $\tomega^{(1)},\dots,\tomega^{(k)}$ are independent and for each $1 \le j \le k$, $\tomega^{(j)}$ has the same distribution as $\omega^{(j)}$. 
\end{theorem}

Let us note that the random variable $X_N(f; \omega^{(1)},\dots,\omega^{(k)})$ is an example of U-statistics, for which decoupling inequalities have a long history; we refer to \cite{delapena_gine_2000} for an introduction to this topic. 
Let us stress however that the decoupling inequality of Theorem~\ref{thm: Lp decoupling inequality} has the remarkable feature that the dependence in $k$ is of the form $C_1^k$ with a universal constant $C_1$. 
This is not true for general U-statistics and the method used here to derive this inequality is quite different from the techniques usually employed to obtain decoupling inequalities. 
For instance, in \cite{delapena_montgomery_1995}, the authors obtain more general decoupling inequalities but the dependence in $k$ is of the type $C^{k^2}$.

We postpone the proof of Theorem~\ref{thm: Lp decoupling inequality} to Section~\ref{subsec: proof decoupling} below, and first derive some corollaries.

\subsubsection{Some useful corollaries}

From Theorem~\ref{thm: Lp decoupling inequality} we can deduce the following corollary, which is the most useful result of this section for our purposes. 

\begin{corollary}{\label{cor: Lp decoupling inequality}}
	Assume that $\omega^{(j)} = (\omega_n^{(j)})_{n \in \bbN }$ are sequences of independent and centered random variables.
	There exist a constant $C_2$ such that for every $ p \in (1,2]$, for every $k \ge 1$, $N \ge 1$ and any tetrahedral function $f \in \cT_N^k$, 
	\begin{equation*}
		\big\|X_N(f; \omega^{(1)},\dots,\omega^{(k)})\big\|_p \le \Big( \frac{C_2}{p-1} \Big)^k  \bigg(\sum\limits_{1 \le n_1 <\dots<n_k \le N} |f(n_1,\dots,n_k)|^p \prod\limits_{j=1}^k \|\omega_{n_j}^{(j)}\|_p^p\bigg)^{1/p}. 
	\end{equation*}
\end{corollary}

\begin{proof}
First, apply Theorem \ref{thm: Lp decoupling inequality}: we obtain	
\begin{equation}
	\label{eq: apply decoupling first}
	\big\| X_N(f; \omega^{(1)},\dots,\omega^{(k)}) \big\|_p \le \Big( \frac{C_1}{p-1} \Big)^k \big\| X_N(f; \tomega^{(1)},\dots,\tomega^{(k)}) \big\|_p.
\end{equation}
Let us now freeze the $k-1$ first variables $\tomega^{(1)},\dots,\tomega^{(k-1)}$ in $X_N(f; \tomega^{(1)},\dots,\tomega^{(k)})$ and denote \(\cF_n^{(k)} = \sigma(\tomega_{i}^{(k)}, i\leq n )\) for \(n\geq 1\). 
Then, we consider the \(\cF_n^{(k)}\)-martingale
\begin{equation*}
	X_N(f;\tomega^{(1)},\dots,\tomega^{(k)}) = \sum\limits_{n = 1}^N  F(n) \tomega_{n}^{(k)}
\end{equation*}
with for $n \ge 1$,
\begin{equation*}
	F(n) := \sum\limits_{1 \le n_1 < \dots < n_{k-1} < n} f(n_1,\dots,n_{k-1},n) \prod\limits_{j=1}^{k-1} \tomega_{n_j}^{(j)} \,,
\end{equation*}
and note that the \(F(n)\) are random variables independent of $(\cF_n^{(k)})_{n \ge 1}$. 
Applying the Burkholder--Davis--Gundy (BDG) inequality~\cite{burkholder_davis_gundy_1972} (see also Proposition~\ref{prop: BDG} below), this yields that 
\begin{equation*}
\| X_N(f; \tomega^{(1)},\dots,\tomega^{(k)})\|_p \leq C_p \, \Es\Bigg[ \bigg(\sum\limits_{n=1}^N F(n)^2 (\tomega_{n}^{(k)})^2 \bigg)^{p/2}\Bigg]^{1/p} \,,
\end{equation*}
for some constant \(C_p\) (which is bounded by a universal constant uniformly for \(p \in [1,2]\)).
Since \(p/2 \in [0,1]\), we can use the inequality \((\sum_{i} a_i)^{p/2} \leq \sum_{i} a_i^{p/2}\) for any non-negative \(a_i\)'s to get that the above is bounded by
\begin{equation*}
C_p\, \Big(\sum\limits_{n = 1}^N \Es\big[|F(n)|^p\big]\Es\big[|\omega_n^{(k)}|^p\big] \Big)^{1/p},
\end{equation*}
using also that \((\tomega_i^{(k)})_{i\geq 1}\) has the same distribution as \((\omega_i^{(k)})_{i\geq 1}\).

All together, if for any $1 \le n_k \le N$ we denote $f_{n_k}(n_1,\dots,n_{k-1}) = f(n_1,\dots,n_k)$, then $f_{n_k} \in \cT_{n_k-1}^{(k-1)}$ and recalling the definition of \(F(n)\) we have that \(F(n_k)=X_{n_k-1}(f_{n_k}; \tomega^{(1)},\dots,\tomega^{(k-1)})\), so that 
\begin{equation*}
\norme{X_N(f; \tomega^{(1)},\dots,\tomega^{(k)})}_p^p
\le (C_p)^p  \sum_{n_k=1}^N \norme{X_{n_k-1}(f_{n_k}; \tomega^{(1)},\dots,\tomega^{(k-1)})}_p^p \big\| \omega_{n_k}^{(k)} \big\|_p^p \,.
\end{equation*}
We can then iterate this inequality to finally obtain 
\begin{equation}
\label{eq: corollary Lp decoupling}
\norme{X_N(f; \tomega^{(1)},\dots,\tomega^{(k)})}_p^p
\le (C_p)^{p k}  \sum_{1\leq n_1< \cdots <n_k \leq N} |f(n_1,\dots,n_k)|^p \prod\limits_{j=1}^k \big\| \omega_{n_j}^{(j)} \big\|^p_p\,.
\end{equation}
This concludes the proof of Corollary \ref{cor: Lp decoupling inequality} follows, recalling also~\eqref{eq: apply decoupling first}.
\end{proof}

\begin{remark}
Let us stress that Corollary~\ref{cor: Lp decoupling inequality} will be very useful in what follows but it is not completely sharp, contrary to Theorem~\ref{thm: Lp estimates for discrete chaos}.
For instance, consider the case $k=1$ and i.i.d.\ random variables $\omega$ as in \eqref{eq: definition of a general heavy-tailed random variable} with a pure polynomial tail, \textit{i.e.}\ \(\Pro(|\omega|>t) \sim t^{-\gamma}\) as \(t\to\infty\).
It is a classical fact, see \cite[XVII.7]{feller_1966}, that
\[
\sum\limits_{1\le n \le N} f(n) \omega_n \propto \Big( \sum\limits_{1 \le n \le N} |f(n)|^{\gamma} \Big)^{1/\gamma} 
,\]
in the sense that the ratio of the left-hand side by the right-hand side converges in distribution to a non-degenerate (\(\gamma\)-stable) random variable; this corresponds to the heavy-tailed analogue of the central limit theorem. 
Then, let us observe that Corollary~\ref{cor: Lp decoupling inequality} does not capture this correct scaling, since one cannot apply the inequality with \(p=\gamma\). 
The limitation of Corollary~\ref{cor: Lp decoupling inequality} is that it does not exploit the full heavy-tail structure of the \(\omega\)'s and in fact holds for any random variable $\omega$ with a finite $\mathds{L}^p$ moment. 
\end{remark}

Before turning to the proof of Theorem \ref{thm: Lp decoupling inequality}, in view of the above remark, let us therefore state a slightly stronger version of Corollary~\ref{cor: Lp decoupling inequality}. 
In order to fully take advantage of the heavy tail structure, we will need to perform $\mathds{L}^p$ estimates for some families of random variables $\omega^{(j)}$, and $\mathds{L}^q$ estimates with $p \neq q$ for the others: in the statement below, the set $I \subset \intg{1}{k}$ is thought as the set of indices for which $\mathds{L}^p$ estimates are to be used. 

\begin{corollary}{\label{cor: less naive Lp estimate}}
	There exist a constant $C_2$ such that for every $p\in (1,2]$ and $p\leq q\leq 2$, for every $k \ge 1$, $N \ge 1$ and $I \subset \intg{1}{k}$, for any tetrahedral function $f \in \cT_N^k$,
	\begin{equation}
		{\label{eq: less naive Lp estimate}}
		\begin{split}
		\left\|X_N(f; \omega^{(1)}, \dots,\omega^{(k)})\right\|_p & \\
		\le \Big( \frac{C_2}{p-1} \Big)^k & \Bigg( \sum\limits_{n_i, i \in I} \prod\limits_{i \in I} \| \omega_{n_i}^{(i)} \|_p^p \bigg(  \!\!\!\sumtwo{n_i, i \notin I}{1 \le n_1 < \dots <n_k \le N} \!\!\!  |f(n_1,\dots,n_k)|^q \prod\limits_{i \notin I} \| \omega_{n_i}^{(i)} \|_q^q \bigg)^{p/q}\Bigg)^{1/p}.
		\end{split}
	\end{equation}
	Additionally, applying Jensen's inequality (to the first sum) shows that the above is bounded by
	\begin{equation}
		{\label{eq: less naive Lp estimate with Jensen}}
	\Big( \frac{C_2}{p-1} \Big)^k N^{|I|(\frac1p-\frac1q)}  \bigg(\sum\limits_{1 \le n_1 < \dots < n_k \le N} |f(n_1,\dots,n_k)|^q  \prod\limits_{i \in I} \| \omega_{n_i}^{(i)} \|_p^q \prod\limits_{i \notin I} \| \omega_{n_i}^{(i)} \|_q^q\bigg)^{1/q}.
	\end{equation}
\end{corollary}

\begin{proof}
By the decoupling inequality of Theorem~\ref{thm: Lp decoupling inequality}, we only need to treat \(X_N(f; \tomega^{(1)}, \dots,\tomega^{(k)})\), that is when the sequences \(\tomega^{(j)}\) are independent.

Let \(I\subset \intg{1}{k}\) and define, for every indices $(n_i)_{i \in I}$, the random variable 
\begin{equation*}
g((n_i)_{i \in I}) =\sumtwo{n_i, i \notin I}{1 \le n_1 < \dots <n_k \le N} f(n_1,\dots,n_k) \prod\limits_{i \notin I} \tomega_{n_i}^{(i)}.
\end{equation*}
Note that random variables \(g((n_i)_{i \in I})\) are independent of the random sequences~$\tomega^{(i)}$ for $i \in I$.
Rewriting 
\[
X_N(f; \tomega^{(1)}, \dots,\tomega^{(k)}) = \sum\limits_{n_i, i \in I} g((n_i)_{i \in I}) \prod\limits_{i \in I} \tomega_{n_i}^{(i)} \,,
\]
we can therefore apply Corollary~\ref{cor: Lp decoupling inequality}, or more precisely~\eqref{eq: corollary Lp decoupling}, conditionally to the random variables $\tomega^{(i)}$ for $i \notin I$.
We obtain 
\begin{equation*}
\left\|X_N(f; \tomega^{(1)}, \dots,\tomega^{(k)})\right\|_p
\leq (C_p)^{|I|}  \bigg( \sum\limits_{n_i, i \in I} \prod_{i \in I} \| \omega_{n_i}^{(i)} \|_p^p \Es \Big[\big|g( (n_i)_{i \in I})\big|^p \Big] \bigg)^{1/p} \,.
\end{equation*}
Now, using Jensen's inequality, since \(p\leq q\), we get that 
\begin{equation*}
\Es \Big[\big|g( (n_i)_{i \in I})\big|^p \Big] \leq \Es \Big[\big|g( (n_i)_{i \in I})\big|^q \Big]^{p/q} \,.
\end{equation*}
We now apply again Corollary \ref{cor: Lp decoupling inequality} (more precisely~\eqref{eq: corollary Lp decoupling}), to get that
\begin{equation*}
\Es \Big[\big|g( (n_i)_{i \in I})\big|^q \Big]^{1/q}  \le (C_q)^{k-|I|} \bigg( \sumtwo{n_i, i \notin I}{1 \le n_1 < \dots <n_k \le N} |f(n_1,\dots,n_k)|^q \prod\limits_{i \notin I} \| \omega_{n_i}^{(i)} \|_q^q \bigg)^{1/q} \,.
\end{equation*}
Putting everything together, recalling~\eqref{eq: apply decoupling first} and the fact that \(C_p,C_q\) are bounded by a universal constant, this concludes the proof.
\end{proof}

\subsubsection{Proof of Theorem~\ref{thm: Lp decoupling inequality}}
\label{subsec: proof decoupling}

The proof is based on two elementary results from the theory of martingales and is a direct adaptation of the proof for the continuous version which can be found in \cite{berger_chong_lacoin_2023}.

Let us state two results that we use in the following. 
Let $(\Omega,\mathcal{F},\Pro) $ be a probability space and let $(\mathcal{F}_n)_{n \ge 1} $ be a filtration.
The first result is the BDG inequality for discrete martingales, see \cite[Theorem~1.1]{burkholder_davis_gundy_1972}.

\begin{proposition}[BDG inequality]
	\label{prop: BDG}
	Let $(M_n)_{n \ge 1}$ be a martingale with respect to $(\mathcal{F}_n)_{n \ge 1}$ such that $M_0 = 0$. 
	Denote $\langle M \rangle$ its quadratic variation, \textit{i.e.}\ \(\langle M \rangle^2  = \sum_{n \ge 1} (M_n - M_{n-1})^2 \) and $M^* = \sup_{n \ge 1} |M_n|$. 
	Then, for any $p \ge 1$, there is a constant $C_p$ which depends only on $p$ (and is bounded by a constant uniformly in \(p\in [1,2]\)) such that 
	\begin{equation}
	\frac{1}{C_p} \, \Es\big[\langle M \rangle^{p/2}\big] \le \Es\big[(M^*)^p\big] \le C_p\, \Es\big[\langle M \rangle^{p/2}\big].
	\end{equation}
\end{proposition}

The second result deals with tangent processes. 
Two adapted processes $(d_n)_{n\ge 1}$ and $(e_n)_{n \ge 1} $ are said to be \textit{tangent} if for all $n \ge 1$ and all $t \in \mathbb{R}$, a.s.\ $\Pro[e_n \ge t | \mathcal{F}_{n-1}] = \Pro[d_n \ge t | \mathcal{F}_{n-1}]$. 
The proof of the following result can be found in \cite[Theorem~2]{hitczenko_1988}. 

\begin{proposition}[Comparing \(\mathds{L}^p\) norms of tangent processes]
	\label{prop: tangent processes}
	Let $(e_n)_{n \ge 1}$ and $(d_n)_{n \ge 1}$ be two non-negative processes which are tangent.
	Then for any $p > 0$, there is a constant \(C_p'\) such that 
	\begin{equation*}
	\Es\bigg[\Big(\sum\limits_{n \ge 1} e_n\Big)^p\bigg] \le C_p'\, \Es\bigg[\Big(\sum\limits_{n \ge 1} d_n\Big)^p\bigg] \,.
	\end{equation*}
\end{proposition}

\noindent
The remarkable feature of this last proposition is that it is also valid for $p \in(0,1]$ and in particular does not use convexity; in fact, for $p \in(0,1]$ (more generally for concave functions) one can take \(C_p=6\), see \cite[p.228]{hitczenko_1988}. 
This is crucial for us since we will apply it for $p/2 \leq 1$. 

\begin{remark}
	Both of the above results are actually valid in a more general setting which strictly includes $\mathds{L}^p$ norms.
	One could thus expect that the result of Theorem~\ref{thm: Lp estimates for discrete chaos} holds for more general functionals. 
	However we will ultimately use Doob's inequality and this is where we have to restrict ourselves to $\mathds{L}^p$ norms. 
\end{remark}

\begin{proof}[Proof of Theorem~\ref{thm: Lp decoupling inequality}]
Our main goal is to show that there exists a constant $C_1$ such that for every $1 \le i \le k$,
\begin{equation*}
\norme{X_N(f; \omega^{(1)},\dots,\omega^{(i)},\tomega^{(i+1)},\dots,\tomega^{(k)})}_p \le C_1\, \norme{X_N(f; \omega^{(1)},\dots,\omega^{(i-1)},\tomega^{(i)},\tomega^{(i+1)},\dots,\tomega^{(k)})}_p \,.
\end{equation*}
The proof then follows easily by iteration.
Additionally, by conditioning with respect to the random variables $\tomega^{(i+1)},\dots,\tomega^{(k)}$ it is actually enough to prove this for $i=k$.

For \(N\geq 1\), define the \(\sigma\)-algebra $\cF_N = \sigma(\omega_n^{(j)}, 1 \le j \le k, 1 \le n \le N; \tomega_n^{(k)}, 1 \le n \le N)$. 
Then the processes $(X_N(f; \omega^{(1)},\dots,\omega^{(k)}))_{N \ge 1}$ and $(X_N(f;\omega^{(1)},\dots,\omega^{(k-1)},\tomega^{(k)})_{N \ge 1}$ are martingales with respect to the filtration $(\cF_N)_{N \ge 1}$.
Therefore, by the BDG inequality (Proposition~\ref{prop: BDG}), we have
\begin{equation*}
\norme{X_N(f;\omega^{(1)},\dots,\omega^{(k)})}_p \le C_p \: \Es\left[ \big\langle X(f;\omega^{(1)},\dots,\omega^{(k)})\big\rangle_N^{p/2}\right]^{1/p}.
\end{equation*}
The increments of the process $(\langle X(f;\omega^{(1)},\dots,\omega^{(k)})\rangle_N)_{N \ge 1}$ are given by 
\begin{equation*}
e_N = \bigg( \sum\limits_{1 \le n_1 < \dots < n_{k-1} < N} f(n_1,\dots,n_{k-1},N) \prod\limits_{j=1}^{k-1} \omega_{n_j}^{(j)} \omega_N^{(k)} \bigg)^2 \,,
\end{equation*}
which are non-negative.
Since $\tomega^{(k)}$ is identically distributed to $\omega^{(k)}$ and thanks to the joint level independence property, $(e_N)_{N \ge 1}$ is tangent to $(d_N)_{N \ge 1}$, where 
\begin{equation*}
d_N = \bigg( \sum\limits_{1 \le n_1 < \dots < n_{k-1} < N} f(n_1,\dots,n_{k-1},N) \prod\limits_{j=1}^{k-1} \omega_{n_j}^{(j)} \tomega_N^{(k)}\bigg)^2 \,.
\end{equation*}
Let us now observe that $(d_N)_{N \ge 1}$ is exactly the process of the increments of the quadratic variation of $(X_N(f;\omega^{(1)},\dots,\omega^{(k-1)},\tomega^{(k)}))_{N \ge 1}$.
Therefore, by Proposition~\ref{prop: tangent processes} we have
\begin{equation*}
\begin{split}
	\left\|X_N(f;\omega^{(1)},\dots,\omega^{(k)})\right\|_p  &\le C_p C_{p/2}' \, \Es\left[ \big\langle X(f;\omega^{(1)},\dots,\tomega^{(k)})\big\rangle_N^{p/2}\right]^{1/p}\\
	& \le C_p^2 C_{p/2}' \: \left\|\sup\limits_{1 \le n \le N} X_n(f;\omega^{(1)},\dots,\tomega^{(k)})\right\|_p \\
	& \le C_p^2 C_{p/2}' \frac{p}{p-1} \: \left\| X_N(f;\omega^{(1)},\dots,\tomega^{(k)})\right\|_p\,,
\end{split}
\end{equation*}
where we have applied the reversed BDG for the second inequality and finally Doob's inequality for the last one.
This concludes the proof of Theorem~\ref{thm: Lp decoupling inequality}, with \(C_1= \sup_{p\in [1,2]} \{p C_p^2 C_{p/2}'\}\).
\end{proof}

\subsection{Application to discrete chaos: proof of Theorem~\ref{thm: Lp estimates for discrete chaos}}
\label{section: Lp discrete chaos}

First of all, let us reformulate Corollary~\ref{cor: less naive Lp estimate} in the context of polynomial chaos.
For $k \ge 1$, let $\omega^{(j)} = (\omega_{x}^{(j)})_{x \in \Omega_{\delta}}$, $1 \le j \le k$, be sequences of i.i.d.\ and centered random variables jointly level independent (recall the beginning of Section~\ref{subsec: decoupling inequalities} for a definition). 
We obtain the following, as an easy consequence of Corollary~\ref{cor: less naive Lp estimate}.

\begin{proposition}{\label{prop: general Lp-Lq inequality}}
	There exist a constant $C_2$ such that for every $ p \in (1,2]$ and $p\leq q\leq 2$, for every $k \ge 1$, $\delta > 0$ and $J \subset \intg{1}{k}$, for any symmetric function $f: \Omega_{\delta}^k \to \bbR$ that vanishes on the diagonals,
	\begin{equation*}
		{\label{eq: general Lp-Lq inequality}}
		\bigg\| \frac{1}{k!} \sum\limits_{x_1,\dots,x_k \in \Omega_{\delta}} f(x_1,\dots,x_k) \prod\limits_{j=1}^k \omega_{x_j}^{(j)} \bigg\|_p 
		\le \Big(\frac{C_2}{p-1} \Big)^k |\Omega|^{|J|(\frac1p-\frac1q)} \big\|\overline{f}\big\|_q \prod\limits_{j \in J} \frac{\| \omega^{(j)} \|_p }{v_{\delta}^{1/p}} \prod\limits_{j \notin J} \frac{ \| \omega^{(j)} \|_q}{v_{\delta}^{1/q}} \,.
	\end{equation*}	
	Recall here that \(\overline{f}\) is the piecewise constant extension of \(f\) and \(v_{\delta}\) is the volume of any cell of \(\Omega_{\delta}\).
\end{proposition}
	Note that we may also bound \(|\Omega|^{|J|(\frac1p-\frac1q)}\leq (|\Omega|\vee 1)^{k(\frac1p-\frac1q)}\) since \(\frac1p-\frac1q \geq 0\) in the above.

\begin{proof}
	Since the set $\Omega_{\delta}$ is finite, it is enough to show the result for \(\Omega_{\delta} = \intg{1}{N}\) with~$N$ the number of elements of \(\Omega_{\delta}\) (which is also the number of cells).	
	Then, letting $\mathfrak{S}_k$ be the symmetric group of $\intg{1}{k}$, we can write
	\[
	\frac{1}{k!} \sum\limits_{n_1,\dots,n_k \in \intg{1}{N}} f(n_1,\dots,n_k) \prod\limits_{j=1}^k \omega_{n_j}^{(j)} = \frac{1}{k!} \sum\limits_{\sigma \in \mathfrak{S}_k} \sum\limits_{1 \le n_1 < \dots <n_k \le N} f(n_1,\dots,n_k) \prod\limits_{j=1}^k \omega_{n_j}^{(\sigma(j))} \,.
	\]
	We can then apply Corollary~\ref{cor: less naive Lp estimate}-\eqref{eq: less naive Lp estimate with Jensen} for every fixed \(\sigma \in \mathfrak{S}_k\) and $I = \sigma^{-1}(J)$, to obtain that 
	\[
	\begin{split}
		&\bigg\| \frac{1}{k!} \sum\limits_{x_1,\dots,x_k \in \Omega_{\delta}} f(x_1,\dots,x_k) \prod\limits_{j=1}^k \omega_{x_j}^{(j)} \bigg\|_p \\
		& \qquad \qquad \leq \Big( \frac{C_2}{p-1} \Big)^k  \prod\limits_{j \in J} \| \omega^{(j)} \|_p  \prod\limits_{j \notin J} \| \omega^{(j)} \|_q 
		\times N^{|J|(\frac1p-\frac1q)} \bigg(\sum\limits_{1 \le n_1 < \dots < n_k \le N} |f(n_1,\dots,n_k)|^q\bigg)^{1/q}	\,,
	\end{split}
	\]
	where we have used the fact that for every fixed $j \in \intg{1}{k}$, the random variables $(\omega_n^{(j)})$ are identically distributed. 
	Notice now that, by symmetry of~$f$ and recalling the formula~\eqref{eq: norm of the piecewise extension of f} of \(\|\overline{f}\|_q\), we have
	\[ 
	N^{|J|(\frac1p-\frac1q)} \bigg(\sum\limits_{1 \le n_1 < \dots < n_k \le N} |f(n_1,\dots,n_k)|^q \bigg)^{1/q} = |\Omega|^{|J|(\frac1p-\frac1q)} (v_{\delta})^{-|J|/p} (v_{\delta})^{-(k-|J|)/q} \big\| \overline{f} \big\|_q\,,
	\]
	having also used the fact that $|\Omega| = N v_{\delta}$.
	This concludes the proof.
\end{proof}

	
	%

We are now ready to complete the proof of Theorem \ref{thm: Lp estimates for discrete chaos}. 

\begin{proof}[Proof of Theorem \ref{thm: Lp estimates for discrete chaos}]
	Define, for every $x \in \Omega_{\delta}$, the truncated random variables
	\begin{equation}{\label{eq: def of the big jump rv}}
	\homega_x = \omega_x \, \Ind_{|\omega_x| > V_{\delta}} - \Es[\omega_x \, \Ind_{|\omega_x| > V_{\delta}}],
	\end{equation}
	and 
	\begin{equation}{\label{eq: def of the small jump rv}}
	\bomega_x = \omega_x \, \Ind_{|\omega_x| \le V_{\delta}} - \Es[\omega_x \, \Ind_{|\omega_x| \le V_{\delta}}] .
	\end{equation}
	In particular, we have \(\omega_x = \homega_x + \bomega_x\) and \(\Es[\homega_x] = \Es[\bomega_x] = 0\).
	Expanding the product $\prod_{j=1}^k (\homega_{x_j} + \bomega_{x_j})$ and applying Minkowski's inequality, we obtain
	\begin{equation*}
	\bigg\|\frac{1}{k!} \sum\limits_{x_1,\dots,x_k \in \Omega_{\delta}} f(x_1,\dots,x_k) \prod\limits_{j=1}^k \omega_{x_j} \bigg\|_p 
	\le \sum\limits_{J \subset \intg{1}{k}} \bigg\| \frac{1}{k!} \sum\limits_{x_1,\dots,x_k \in \Omega_{\delta}} f(x_1,\dots,x_k) \prod\limits_{j \in J} \homega_{x_j} \prod\limits_{j \notin J} \bomega_{x_j} \bigg\|_p \,.
	\end{equation*}
	For every subset $J \subset \intg{1}{k}$, we apply Proposition~\ref{prop: general Lp-Lq inequality}, using~$\mathds{L}^p$ estimates for the variables $\homega$ and~$\mathds{L}^q$ estimates for the variables $\bomega$ with $p \in (1,\gamma)$ and $q \in (\gamma,2]$. 
	This gives
	\begin{equation*}
	\bigg\| \frac{1}{k!} \sum\limits_{x_1,\dots,x_k \in \Omega_{\delta}} f(x_1,\dots,x_k) \prod\limits_{j \in J} \homega_{x_j} \prod\limits_{j \notin J} \bomega_{x_j} \bigg\|_p 
	\le \Big( \frac{C_2}{p-1} (|\Omega|\vee 1)^{\frac1p-\frac1q} \Big)^k  \|\overline{f}\|_{q} \bigg( \frac{\| \homega \|_p}{v_{\delta}^{1/p}}\bigg)^{|J|} \bigg(\frac{\|\bomega\|_q}{v_{\delta}^{1/q}}\bigg)^{k-|J|} \,.
	\end{equation*}

	Then, summing over $J \subset \intg{1}{k}$ yields 
	\begin{equation}
		{\label{eq: intermediate step for discrete estimates}}
		\bigg\|\frac{1}{k!} \sum\limits_{x_1,\dots,x_k \in \Omega_{\delta}} f(x_1,\dots,x_k) \prod\limits_{j=1}^k \omega_{x_j} \bigg\|_p 
		\le \Big( \frac{C_2}{p-1}(|\Omega|\vee 1)^{\frac1p-\frac1q} \Big)^k \, \|\overline{f}\|_{q} \, \bigg( \frac{\| \homega \|_p}{v_{\delta}^{1/p}} + \frac{\|\bomega\|_q}{v_{\delta}^{1/q}} \bigg)^{k}.
	\end{equation}
	Using the fact that the random variable $|\omega|$ has a tail given by \eqref{eq: definition of a general heavy-tailed random variable}, one can easily check by properties of regularly varying functions that for any fixed \(p<\gamma\) and \(q>\gamma\), as \(t\to\infty\)  we have
	\[
	\Es[|\omega|^p \Ind_{|\omega| > t}] \sim \frac{\gamma}{\gamma-p}\, t^{p-\gamma} \varphi(t)
	\qquad \text{ and } \qquad
	\Es[|\omega|^q \Ind_{|\omega| \le t}] \sim \frac{\gamma}{q-\gamma}\, t^{\gamma-q} \varphi(t) \,.
	\] 
	In particular, recalling the definition~\eqref{eq: general scale of the noise} of \(V_{\delta}\), we get that for any $a > 0$
	\begin{equation}
	\label{eq: asymptotic behavior for moments of order p-q}
	\begin{split}
	\text{for \(0< p<\gamma\),} \qquad 
	& \Es[|\omega|^p \Ind_{|\omega| > a V_{\delta}}] \underset{\delta \to 0}{=} (1+o(1))\frac{\gamma}{\gamma-p} a^{p-\gamma} (V_{\delta})^p v_{\delta} ,\\
	\text{for \(q>\gamma\) }\qquad 
	& \Es[|\omega|^q \Ind_{|\omega| \le a V_{\delta}}] \underset{\delta \to 0}{=} (1+o(1)) \frac{\gamma}{q - \gamma} a^{q-\gamma} (V_{\delta})^q v_{\delta} \,.
	\end{split}
	\end{equation}
	(We have included a parameter \(a>0\) for future use, but in this proof we only use it for \(a=1\).)
	In particular, we get that there is some constant $C$ (which does not depend on $p,q$, or $\gamma$) such that, for every $\delta > 0$,
	\begin{equation}{\label{eq: bound for Lpq moments}}
	\norme{\homega}_p \le \frac{C}{(\gamma-p)^{1/p}}  v_{\delta}^{1/p} V_{\delta}
	\qquad \text{ and } \qquad 
	\norme{\bomega}_q \le \frac{C}{(q-\gamma)^{1/q}} v_{\delta}^{1/q} V_{\delta} \,.
	\end{equation}
	Plugging these estimates into \eqref{eq: intermediate step for discrete estimates} yields the upper bound
	\(
	(C_0)^k (V_{\delta})^k \|\overline{f}\|_{q}
	\),
	with the constant
	\[
	C_0 =  \frac{2 C C_2}{p-1} \max\{(\gamma-p)^{-1/p},(q-\gamma)^{-1/q}\} \, (|\Omega|\vee 1)^{\frac1p - \frac1q} \,,
	\]
	as claimed in Remark~\ref{rem: constant C0}.
	This concludes the proof.
\end{proof}

\subsection{Application to continuous chaos: proof of Theorem~\ref{thm: Lp estimates for continuous chaos}}
\label{section: Lp continuous chaos}

In this section we prove Theorem~\ref{thm: Lp estimates for continuous chaos}. 
Our strategy is to first establish \ref{eq: Lp estimates for continuous chaos} for simple functions and then conclude by a density argument. 
Fix $k \ge 1$. 
A simple symmetric function $\bff:\Omega^k \to \bbR$ is a measurable function of the form 
\begin{equation}
	\bff = \sum\limits_{1 \le i_1,\dots,i_k \le M} \alpha_{i_1,\dots,i_k} \Ind_{A_{i_1} \times \dots \times A_{i_k}}\,,
\end{equation}
where the measurable sets $(A_i)_{1 \le i \le M}$ form a partition of $\Omega$, and the coefficients $\alpha_{i_1,\dots,i_k}$ are symmetric and such that $\alpha_{i_1,\dots,i_k} = 0$, if for some $l \neq m$, $i_l = i_m$. 
In particular, a symmetric simple function vanishes on the diagonals. 
We set for every $a \in (0,1)$,
\begin{equation*}
	\bX^{(a)}(\bff) = \frac{1}{k!} \int_{\Omega^k} \bff(x_1,\dots,x_k) \prod\limits_{j=1}^k \bzeta^{(a)}(\dd x_j)\,.
\end{equation*}
Define, for every $0 \le a \le b \le +\infty$ and $p \ge 1$, the truncated moments of the noise $\bzeta$,
\begin{equation}{\label{eq: def of the truncated moments of the noise}}
	m_p(a,b) = \int_{\bbR} |z|^p \Ind_{a < |z| \le b} \lambda(\dd z)\,.
\end{equation}
In our case, the Lévy measure $\lambda$ is explicit, and we can compute the truncated moments but the following result is more general (see Theorem~\ref{thm: continuous polynomial chaos for general noise} and the discussion before). 

\begin{proposition}{\label{prop: Lp estimates for simple functions}}
	There exists a universal constant $C_3$ such that, for every $1 < p < q \le 2$, for every $a \in (0,1)$ and every simple symmetric function $\bff: \Omega^k \to \bbR$,
	\begin{equation*}
		\bigg\|\bX^{(a)}(\bff)\bigg\|_p \le \Big(\frac{C_3}{p-1} (|\Omega|\vee 1)^{1/p-1/q}\Big)^k\Big(m_p(1,+\infty)^{1/p}+m_q(a,1)^{1/q}\Big)^k \|\bff\|_q.
	\end{equation*}
\end{proposition}

\begin{proof}[Proof of Theorem \ref{thm: Lp estimates for continuous chaos}]
	If the Lévy measure is $\lambda(\dd z) = (c_+ \Ind_{z > 0}+c_- \Ind_{z < 0})\gamma |z|^{-1-\gamma} \dd z$, then for every $p \in (1,\gamma)$ and $q \in (\gamma,2)$,
	\begin{equation*}
		m_p(1,+\infty) \le \frac{2 \gamma }{\gamma-p}\,, \text{ and } \quad m_q(a,1) \le m_q(0,1) \le \frac{2 \gamma }{q-\gamma}\,.
	\end{equation*}
	This proves Theorem \ref{thm: Lp estimates for continuous chaos} for simple functions and with a constant which is bounded by $C_0$.
	
	It remains to deal with a general function \(\bff\in L^q_s(\Omega^k)\), using a density argument that we now explain.
	For any \(\gep \in (0,1)\), let $\bff^{(\gep)}$ be a simple symmetric function which approximates $\bff$ in $L^q_s(\Omega^k)$ as $\gep \downarrow 0$.
	Then, applying \eqref{eq: Lp estimates for continuous chaos} to \(\bff^{(\gep)}-\bff^{(\gep')}\) with $\gep, \gep' \in(0,1)$, we obtain
	\begin{equation*}
		\big\|\bX^{(a)}(\bff^{(\gep)}) - \bX^{(a)}(\bff^{(\gep')})\big\|_p  \le (C_0)^k \big\|\bff^{(\gep)} - \bff^{(\gep')}\big\|_q \,.
	\end{equation*}
	Since $\bff^{(\gep)}$ converges in $L^q_s(\Omega^k)$ as $\gep \downarrow 0$, the family of random variables $\bX^{(a)}(\bff^{(\gep)})$ is a Cauchy sequence in~$\mathds{L}^p$. 
	Therefore, it converges to a random variable $\bX^{(a)}(\bff) :=\frac{1}{k!} \int_{\Omega^k} \bff(x_1,\dots,x_k) \prod_{j=1}^k \bzeta^{(a)}(\dd x_j)$ for which~\eqref{eq: Lp estimates for continuous chaos} holds. 
	The definition of \(\bX^{(a)}(\bff)\) does not depend on the choice of the approximations $\bff^{(\gep)}$ nor on $p$.

\end{proof}

\begin{proof}[Proof of Proposition \ref{prop: Lp estimates for simple functions}]
	Denote for every $i \in \intg{1}{k}$, $\omega_i = \bzeta^{(a)}(A_i)$.
	The properties of the coefficients $\alpha_{i_1,\dots,i_k}$ imply that 
	\begin{equation*}
		\begin{split}
			\frac{1}{k!} \int_{\Omega^k} \bff(x_1,\dots,x_k) \prod\limits_{j=1}^k \bzeta^{(a)}(\dd x_j) &= \frac{1}{k!} \sum\limits_{1 \le i_1,\dots,i_k \le M} \alpha_{i_1,\dots,i_k} \prod\limits_{j=1}^k \bzeta^{(a)}(A_{i_j}) \\
			&= \sum\limits_{1 \le i_1 < \dots < i_k \le M} \alpha_{i_1,\dots,i_k} \prod\limits_{j=1}^k \omega_{i_j}\,.
		\end{split}
	\end{equation*}
	Define, for every $i \in \intg{1}{k}$, the truncated random variables
	\begin{equation}{\label{eq: def of the continuous big jump rv}}
		\homega_i = \sum\limits_{(x,z) \in \Lambda} \Ind_{A_i}(x) z \Ind_{|z| > 1} - |A_i| \int_{\bbR} z \Ind_{|z| > 1} \lambda(\dd z)\,,
	\end{equation}
	and 
	\begin{equation}{\label{eq: def of the continous small jump rv}}
		\bomega_i = \sum\limits_{(x,z) \in \Lambda} \Ind_{A_i}(x) z \Ind_{a < |z| \le 1} - |A_i| \int_{\bbR} z \Ind_{a < |z| \le 1} \lambda(\dd z)\, .
	\end{equation}
	In particular, we have \(\omega_i = \homega_i + \bomega_i\) and \(\Es[\homega_i] = \Es[\bomega_i] = 0\).
	Expanding the product $\prod_{j=1}^k (\homega_{i_j} + \bomega_{i_j})$ and applying Minkowski's inequality, we obtain
	\begin{equation*}
		\bigg\| \sum\limits_{1 \le i_1 < \dots < i_k \le M} \alpha_{i_1,\dots,i_k} \prod\limits_{j=1}^k \omega_{i_j} \bigg\|_p 
		\le \sum\limits_{J \subset \intg{1}{k}} \bigg\| \sum\limits_{1 \le i_1 < \dots < i_k \le M} \alpha_{i_1,\dots,i_k} \prod\limits_{j \in J} \homega_{i_j} \prod\limits_{j \notin J} \bomega_{i_j} \bigg\|_p \,.
	\end{equation*}
	For every subset $J \subset \intg{1}{k}$, we apply Corollary~\ref{cor: less naive Lp estimate}, using~$\mathds{L}^p$ estimates for the variables $\homega$ and~$\mathds{L}^q$ estimates for the variables $\bomega$ with $p \in (1,\gamma)$ and $q \in (\gamma,2]$. 
	This gives
	\begin{equation}{\label{eq: intermediate step for continuous estimates}}
		\begin{split}
			&\bigg\| \sum\limits_{1 \le i_1 < \dots < i_k \le M} \alpha_{i_1,\dots,i_k} \prod\limits_{j \in J} \homega_{i_j} \prod\limits_{j \notin J} \bomega_{i_j} \bigg\|_p \le\\ &\qquad \qquad \Big( \frac{C_2}{p-1} \Big)^k  \Bigg( \sum\limits_{i_j, j \in J} \prod\limits_{j \in J} \| \homega_{i_j} \|_p^p \bigg(  \!\!\!\sumtwo{i_j, j \notin J}{1 \le i_1 < \dots <i_k \le M} \!\!\!  |\alpha_{i_1,\dots,i_k}|^q \prod\limits_{j \notin J} \| \bomega_{i_j} \|_q^q \bigg)^{p/q}\Bigg)^{1/p}.
		\end{split}
	\end{equation}
	It is proved in \cite[Section 6.2]{berger_chong_lacoin_2023} that there is a universal constant $C$ such that for every measurable function $g : \Omega \to \bbR$, every $p \in (1,2]$ and $0 < a \le b \le +\infty$, 
	\begin{equation*}
		\Big\| \sum\limits_{(x,z) \in \Lambda} g(x) z \Ind_{a < |z| \le b} - \int_{\Omega \times \bbR} g(x) z \Ind_{a < |z| \le b} \dd x \lambda(\dd z) \Big\|_p \le C \Big(\int_{\Omega \times \bbR} |g(x)|^p |z|^p \Ind_{a < |z| \le b} \dd x \lambda(\dd z)\Big)^{1/p} \,.
	\end{equation*}
	This inequality can also be seen as the Poisson analogue of Corollary \ref{cor: Lp decoupling inequality} (for an order $1$ chaos). 
	Therefore, for every $i \in \intg{1}{k}$,
	\begin{equation*}
		\| \homega_i \|_p \le C |A_i|^{1/p} m_p(1,+\infty)^{1/p}\,, \text{ and } \quad \|\bomega_i\|_q \le C |A_i|^{1/q} m_q(a,1)^{1/q}\,.
	\end{equation*}
	Plugging these estimates into \eqref{eq: intermediate step for continuous estimates} yields, 
	\begin{equation*}
		\begin{split}
			&\bigg\| \sum\limits_{1 \le i_1 < \dots < i_k \le M} \alpha_{i_1,\dots,i_k} \prod\limits_{j \in J} \homega_{i_j} \prod\limits_{j \notin J} \bomega_{i_j} \bigg\|_p \le\\ &\quad \Big( \frac{C_2'}{p-1} \Big)^k m_p(1,+\infty)^{|J|/p} m_q(a,1)^{k-|J|/q} \Bigg( \sum\limits_{i_j, j \in J} \prod\limits_{j \in J} |A_{i_j}| \bigg(  \!\!\!\sumtwo{i_j, j \notin J}{1 \le i_1 < \dots <i_k \le M} \!\!\!  |\alpha_{i_1,\dots,i_k}|^q \prod\limits_{j \notin J} |A_{i_j}| \bigg)^{p/q}\Bigg)^{1/p}.
		\end{split}
	\end{equation*}
	Applying Jensen's inequality to the first sum, we finally obtain,
	\begin{equation*}
		\begin{split}
			&\bigg\| \sum\limits_{1 \le i_1 < \dots < i_k \le M} \alpha_{i_1,\dots,i_k} \prod\limits_{j \in J} \homega_{i_j} \prod\limits_{j \notin J} \bomega_{i_j} \bigg\|_p \le\\
			&\qquad \Big( \frac{C_2'}{p-1} \Big)^k m_p(1,+\infty)^{|J|/p} m_q(a,1)^{k-|J|/q} |\Omega|^{|J|(1/p-1/q)} \Bigg(\sum\limits_{1 \le i_1< \dots <i_k \le M} |\alpha_{i_1,\dots,i_k}|^q \prod\limits_{j =1}^k |A_{i_j}| \Bigg)^{1/q}.
		\end{split}
	\end{equation*}
	Summing over $J \subset \intg{1}{k}$ concludes the proof since the symmetric $L^q$ norm of $\bff$ is given by 
	\begin{equation*}
		\|\bff\|_q = \Big(\frac{1}{k!} \sum\limits_{1 \le i_1,\dots,i_k \le M} |\alpha_{i_1,\dots,i_k}|^q \prod\limits_{j=1}^k |A_{i_j}| \Big)^{1/q}\,.
	\end{equation*}
\end{proof}

\section{Multiple Lévy integrals and scaling limits of polynomial chaos}{\label{section: general polynomial chaos}}

	This section is dedicated to the proof of Theorems \ref{thm: continuous polynomial chaos} and \ref{thm: scaling limit of discrete polynomial chaos}. 	
	One of the strengths of our method, compared with the proof of Berger and Lacoin in \cite{berger_lacoin_2021} and \cite{berger_lacoin_2022}, is that we can proceed chaos by chaos. 
	This was true for the $\mathds{L}^p$ estimates, as illustrated in Section \ref{section: Lp estimates}, but it also holds for scaling limits. 
	We therefore postpone the full proof of Theorems \ref{thm: continuous polynomial chaos} and \ref{thm: scaling limit of discrete polynomial chaos} to the last Section~\ref{subsec: polynomial chaos}.
	Before that, in Sections~\ref{subsec: chaos continu ordre k} and~\ref{subsec: polynomial chaos}, we work with a chaos of a fixed order~$k \ge 1$.

\subsection{Defining the continuous chaos of order \texorpdfstring{\(k\)}{k}}
\label{subsec: chaos continu ordre k}

Let $k \ge 1$, \(q\in (\gamma,2)\) and let $\bff \in L^q_s(\Omega^k)$ be a symmetric function.
Recall the definition~\eqref{eq: definition of the general truncated noise} of the truncated noise \(\bzeta^{(a)}\), built thanks to a Poisson point process~\(\Lambda\) on \(\bbR\times \Omega\) of intensity \(\mu(\dd x\, \dd z) = \dd x \lambda(\dd z)\) with \(\lambda(\dd z) = (c_+ \Ind_{z>0}+ c_- \Ind_{z<0})\gamma |z|^{-1-\gamma} \dd z\).

Denote, for every $a\in (0,1]$, 
\begin{equation}
	\bX^{(a)}(\bff) = \frac{1}{k!} \int_{\Omega^k} \bff(x_1,\dots,x_k) \prod\limits_{j=1}^k \zeta^{(a)}(\dd x_j) \,,
\end{equation} 
which is measurable with respect to \(\cF^{(a)} := \sigma( \Lambda \cap (\Omega \times (-\infty,a] \cup [a,\infty))\).

\begin{proposition}{\label{prop: martingale property of multiple integrals}}
The process $(\bX^{(a)}(\bff))_{a\in (0,1]}$ is a time-reversed càdlàg martingale with respect to the filtration $\cF = (\cF^{(a)})_{a\in (0,1]}$.
Furthermore it is bounded in $\mathds{L}^p$ for every $p \in (1,\gamma)$ and, for every $a\in (0,1]$ we have,
\begin{equation}
	\|\bX^{(a)}(\bff)\|_{p} \le (C_0)^k \|\bff\|_{q},
\end{equation}
where the constant $C_0 = C_0(p,q,\gamma,\Omega)$ is given by Theorem \ref{thm: Lp estimates for discrete chaos}.
In particular, we have the following limit \(\bX(\bff):= \lim_{a\downarrow 0} \bX^{(a)}(\bff)\) \(\bbP\)-a.s.\ and in \(\mathds{L}^p\).
\end{proposition}

\begin{proof}
The function $\bff$ has a finite $L^q_s$ norm therefore, according to Theorem~\ref{thm: Lp estimates for continuous chaos}, the random variable $\bX^{(a)}(\bff)$ is well-defined and the family $(\bX^{(a)}(\bff))_{a\in (0,1]}$ is bounded in $\mathds{L}^p$, with \(\|\bX^{(a)}(\bff)\|_{p} \le (C_0)^k \|\bff\|_{q}\).

It only remains to show that it is a martingale with respect to the filtration \((\cF^{(a)})_{a\in (0,1]}\).
Recalling the definition~\eqref{eq: definition of the general truncated noise}, let us start by writing
\begin{equation*}
	\bX^{(a)}(\bff) =  \frac{1}{k!} \int_{\Omega^k \times \mathbb{R}^k} \bff(x_1,\dots,x_k) \prod\limits_{j=1}^k z_j \Ind_{|z_j| \ge a} \prod\limits_{j=1}^k (\delta_{\Lambda} - \mu)(\dd x_j \dd z_j).
\end{equation*}
Now, for \(b \in (a,1]\), write  $z\Ind_{|z| \ge a} = z \Ind_{|z| \ge b} + z \Ind_{a \le |z| < b}$ and expand the product. 
Note that the terms of the form $z_j \Ind_{|z_j| \ge b}$, when integrated against the point process \(\Lambda\) are measurable with respect to $\cF^{(b)}$ and those of the form $z_j \Ind_{a \le |z_j| < b }$ are independent of $\cF^{(b)}$. 
Therefore, we have that
\begin{equation*}
	\begin{split}
	\Es[ \bX^{(a)}(\bff) \mid \cF^{(b)} ] =
	\frac{1}{k!} \sum_{I \subset \intg{1}{k}} \int_{\Omega^{|I|}\times \bbR^{|I|}} \!\! \Es\bigg[\int_{\Omega^{k-|I|}\times \bbR^{k-|I|}} \!\!  \bff(x_1,\dots,x_k) & \prod\limits_{i \notin I} z_i \Ind_{a \le |z_i| < b} (\delta_{\Lambda} - \mu)(\dd x_i \dd z_i)\bigg] \\
	& \cdot \prod\limits_{i \in I} z_i \Ind_{|z_i| \ge b}\prod\limits_{i \in I} (\delta_{\Lambda}-\mu)(\dd x_i \dd z_i) \,.
	\end{split}
\end{equation*}
By Mecke's multivariate formula~\cite[Theorem 4.4]{last_penrose_2018}, when $I \neq \intg{1}{k}$ one has
\[
	\Es\bigg[\int_{\Omega^{k-|I|}\times \bbR^{k-|I|}} \!\! \bff(x_1,\dots,x_k) \prod\limits_{i \notin I} z_i \Ind_{a \le |z_i| < b} (\delta_{\Lambda} - \mu)(\dd x_i \dd z_i)\bigg] = 0 \,.
\]
Therefore only the term $I = \intg{1}{k}$ in the above sum remains, which yields $\Es[\bX^{(a)}(\bff) \mid \cF^{(b)} ] = \bX^{(b)}(\bff)$ and concludes the proof.
\end{proof}

\subsection{Scaling limit of a discrete chaos of order \texorpdfstring{\(k\)}{k}}
{\label{subsec: scaling limit of discrete chaos}}

We now address the question of convergence of discrete multiple integrals to their continuous counterparts. 
Let us denote, for any function $f : \Omega_{\delta}^k \to \bbR$,
\begin{equation*}
	X_{\delta}(f) := \frac{1}{k!} \sum\limits_{x_1,\dots,x_k \in \Omega_{\delta}} f(x_1,\dots,x_k) \prod\limits_{j=1}^k \frac{\omega_{x_j}}{V_{\delta}} = \frac{1}{k!} \int_{x_1,\dots,x_k \in \Omega} \overline{f}(x_1,\dots,x_k) \prod_{j=1}^k \zeta_{\delta}(\dd x_j) \,,
\end{equation*}
where for the second identity we recall the definition~\eqref{eq: definition of the discrete noise} of the discrete noise \(\zeta_{\delta}\) and $\overline{f}$ denotes the piecewise extension (constant on every cell) of $f$ to $\Omega^k$.

\begin{proposition}{\label{prop: convergence of discrete multiple integrals}}
	Assume that \(\omega\) is centered and satisfies~\eqref{eq: definition of a general heavy-tailed random variable} with \(\gamma\in (1,2)\).
	Let $\bff : \Omega^k \to \bbR$ be a symmetric function and let for every $\delta > 0$, $f_{\delta} : \Omega_{\delta}^k \to \bbR$ be a symmetric function that vanishes on the diagonals, and assume that there is some $q \in (\gamma,2]$ such that
	\begin{equation*}
	\lim\limits_{\delta \downarrow 0} \| \bff - \overline{f}_{\delta} \|_q = 0.
	\end{equation*}
	Then for every $s > \frac{D}{2}$, we have the following joint convergence in distribution in $H^{-s}_{\rm loc}(\Omega) \times \mathbb{R}$: as~\(\delta\downarrow 0\),
	\begin{equation*}
		(\zeta_{\delta}, X_{\delta}(f_{\delta})) \xrightarrow{(d)} (\bzeta, \mathbf{X}(\bff)),
	\end{equation*}
	where \(\bzeta\) is a \(\gamma\)-stable Lévy noise and $\mathbf{X}(\bff) = \int_{\Omega^k} \bff(x_1,\ldots, x_k) \prod_{j=1}^k \bzeta(\dd x_j)$ is defined in Proposition~\ref{prop: martingale property of multiple integrals} (and is \(\bzeta\)-measurable).
	The conclusion extends to the joint convergence of a finite collection of chaos $(X_{\delta}(f_{j,\delta}))_{1 \le j \le M}$, provided for each $j \in \intg{1}{M}$, $f_{j,\delta}:\Omega^{k_j} \to \bbR$ ($k_j \ge 1$) is symmetric, vanishes on the diagonals, and converges in $L^q_s(\Omega^{k_j})$ to some function $\bff_{j}$. 
\end{proposition}

\begin{proof}
Our strategy to prove Proposition \ref{prop: convergence of discrete multiple integrals} will rely on two approximations: 
\textit{(i)} we approximate the function $\bff$ by a smooth (and compactly supported) one;
\textit{(ii)} we approximate the discrete noise \(\zeta_{\delta}\) by a truncated one.
The proof for a single chaos readily extends to the case of several chaos.

For technical reasons, we need to introduce the following notation: given $\bff : \Omega^k \to \bbR$ and $\delta > 0$, let~$\bff_{\delta}$ denote the piecewise constant extension to $\Omega$ of $\bff_{\upharpoonright \Omega_{\delta}}$, \textit{i.e.}\, $\bff_{\delta}(y) =\bff(x)$ for every $y \in \cC_{\delta}(x)$.
If the function $\bff$ is continuous and compactly supported in $\Omega^k$, $\bff_{\delta}$ converges uniformly to $\bff$ as $\delta \downarrow 0$.

\subsubsection*{(i) Reducing to smooth functions \texorpdfstring{\(\bff\)}{bff} with compact support.}

Let $(\bff^{(\gep)})_{\gep \in (0,1]}$ be infinitely differentiable functions compactly supported in $\Omega^{k,*} = \{ (x_1,\dots,x_k) \in \Omega^k, \forall i \neq j, x_i \neq x_j\}$, which approximate $\bff$ in $L^q_s(\Omega^k)$ as $\gep \downarrow 0$\footnote{For the existence of such approximating functions \(\bff^{(\gep)}\), first approximate $\bff(x_1,\dots,x_k)$ by $\bff(x_1,\dots,x_k) \Ind_{\forall i \neq j, |x_i - x_j| > \gep}$ and then approximate this last function in $L^q_s(\Omega^{k,\gep})$ where $\Omega^{k,\gep} = \{ (x_1,\dots,x_k) \in \Omega^k, \forall i \neq j, |x_i - x_j| > \gep \}$.}. 
Let $F$ be a bounded and Lipschitz function on the space $H^{-s}_{\rm loc}(\Omega)\times \mathbb{R}$, with Lipschitz constant $\| \nabla F \|_{\infty}$.
Then, we have
\begin{multline*}
\left|\Es\left[F(\bzeta,\mathbf{X}(\bff))\right]-\Es\left[F(\zeta_{\delta}, X_{\delta}(f_{\delta}))\right] \right| \le \| \nabla F \|_{\infty} \Es\left[|\mathbf{X}(\bff) - \mathbf{X}(\bff^{(\gep)})|\right]
\\ +\left|\Es\left[F(\bzeta,\mathbf{X}(\bff^{(\gep)}))\right]-\Es\left[F(\zeta_{\delta}, X_{\delta}(\bff^{(\gep)}))\right] \right|+ \| \nabla F \|_{\infty} \Es\left[|X_{\delta}(\bff^{(\gep)})- X_{\delta}(f_{\delta}) |\right].
\end{multline*}
By Proposition \ref{prop: martingale property of multiple integrals}, the first term is bounded by $(C_0)^k \| \bff - \bff^{(\gep)} \|_{q}$ and by Theorem \ref{thm: Lp estimates for discrete chaos} the third term is bounded by $(C_0)^k \| \bff^{(\gep)}_{\delta} - \overline{f_{\delta}} \|_{q}$.
But by the triangular inequality we have
\begin{equation*}
	\| \bff^{(\gep)}_{\delta} - \overline{f_{\delta}} \|_{q} \le  \|\bff^{(\gep)}_{\delta} - \bff\|_{q} + \| \bff - \overline{f_{\delta}}\|_q. 
\end{equation*}
Therefore, since $\bff^{(\gep)}_{\delta}$ converges uniformly to $\bff^{(\gep)}$, and by assumption \(\| \bff - \overline{f_{\delta}}\|_q \to 0\) as \(\delta\downarrow 0\), we have
\begin{multline*}
 	\limsup_{\delta \downarrow 0} \big|\Es\left[F(\bzeta,\mathbf{X}(\bff))\right]-\Es\left[F(\zeta_{\delta}, X_{\delta}(f_{\delta}))\right] \big| \le 2 \|\nabla F \|_{\infty} \| \bff - \bff^{(\gep)}\|_q
 	\\ + \limsup_{\delta \downarrow 0} \left|\Es\left[F(\bzeta,\mathbf{X}(\bff^{(\gep)}))\right]-\Es\left[F(\zeta_{\delta}, X_{\delta}(\bff^{(\gep)}))\right] \right| \,.
\end{multline*}
One therefore only needs to show that, for any fixed \(\gep\in (0,1)\), the last term goes to \(0\) as \(\delta \downarrow 0\). Then, letting \(\gep\downarrow0\) afterwards gives the conclusion.

In other words, it is enough to show that for any infinitely differentiable and compactly supported function $\bff$, we have the following  convergence in distribution: as \(\delta\downarrow 0\)
\begin{equation}
	{\label{eq: convergence of chaos for a test function}}
	(\zeta_{\delta}, X_{\delta}(\bff)) \xrightarrow{\ (d)\ } (\bzeta, \bX(\bff)) \,.
\end{equation}

\subsubsection*{(ii) Truncating the noise}

In order to prove \eqref{eq: convergence of chaos for a test function}, let us introduce a cut-off on the random variables~$\omega$.
For $a > 0$ and $x \in \Omega_{\delta}$, define
\begin{equation}
	\label{eq: definition cut-off omega}
	\omega_x^{(a)} = \omega_x \Ind_{|\omega_x| > a V_{\delta}} - \kappa_{\delta}(a)\,,
\end{equation}
with $\kappa_{\delta}(a) = \Es[\omega \Ind_{|\omega| > a V_{\delta}}]$.
The random variables $(\omega_x^{(a)})_{x \in \Omega_{\delta}}$ are i.i.d.\ and centered. 
Then define, for every $\delta > 0$ and \(a\in (0,1)\): 
\begin{equation*}
	X_{\delta}^{(a)}(\bff) = \frac{1}{k!} \sum\limits_{x_1,\dots,x_k \in \Omega_{\delta}} \bff(x_1,\dots,x_k) \prod\limits_{j=1}^k \frac{\omega_{x_j}^{(a)}}{V_{\delta}}\,,
\end{equation*}

\begin{lemma}{\label{lemma: uniform approximability}}
If $\bff$ is a smooth function with compact support in $\Omega^{k,*}$, the following uniform approximation holds:
	\begin{equation*}
		\lim\limits_{a \downarrow 0} \displaystyle\limsup_{\delta \downarrow 0} \big\| X_{\delta}(\bff)-X_{\delta}^{(a)}(\bff) \big\|_1 = 0. 
	\end{equation*}
\end{lemma}

Let us first see how this lemma implies~\eqref{eq: convergence of chaos for a test function}, \textit{i.e.}\ Proposition~\ref{prop: convergence of discrete multiple integrals}. 
Let $F$ be a bounded and Lipschitz function on the space $H^{-s}_{\rm loc}(\Omega)\times \mathbb{R}$, with Lipschitz constant $\| \nabla F \|_{\infty}$. 
Then 
\begin{multline*}
	\left|\Es\big[F(\zeta_{\delta}, X_{\delta}(\bff))\big] - \Es\big[F(\bzeta,\bX(\bff))\big]\right| \le \|\nabla F\|_{\infty}  \Es\big[|X_{\delta}(\bff)-X_{\delta}^{(a)}(\bff)|\big]
	\\ + \big|\Es\big[F(\zeta_{\delta}, X_{\delta}^{(a)}(\bff))\big] - \Es\big[F(\bzeta,\bX^{(a)}(\bff))\big]\big|
	+ \|\nabla F\|_{\infty} \Es\big[|\bX^{(a)}(\bff) - \bX(\bff)|\big].
\end{multline*}
By Proposition \ref{prop: convergence of chaos for truncated noise}, the second term goes to \(0\) as \(\delta\downarrow 0\), so we obtain
\begin{multline*}
	\limsup_{\delta \downarrow 0} \big|\Es\big[F(\zeta_{\delta}, X_{\delta}(\bff))\big] - \Es\big[F(\bzeta,\bX(\bff))\big]\big| 
	\\ \le  \|\nabla F\|_{\infty} \limsup_{\delta \downarrow 0} \|X_{\delta}(\bff)-X_{\delta}^{(a)}(\bff)\|_1 + \|\nabla F\|_{\infty} \|\bX^{(a)}(\bff) - \bX(\bff)\|_1 \,.
\end{multline*}
Then, Lemma~\ref{lemma: uniform approximability} shows that the first term vanishes as $a \downarrow 0$. 
Additionally, Proposition \ref{prop: martingale property of multiple integrals} shows that $\lim_{a\downarrow0}\left\|\bX^{(a)}(\bff) - \bX(\bff)\right\|_1 = 0$, which concludes the proof of~\eqref{eq: convergence of chaos for a test function} and thus of Proposition~\ref{prop: convergence of discrete multiple integrals}.

\begin{proof}[Proof of Lemma \ref{lemma: uniform approximability}]
Again, we need to decompose the random variables $\omega$ into two (centered) components: 
recalling the definition~\eqref{eq: definition cut-off omega} of \(\omega_x^{(a)}\), we have
\begin{equation*}
	\omega_x = \omega_x^{(a)} + \bomega_x^{(a)},
\end{equation*}
with $\bomega_x^{(a)} = \omega_x \Ind_{|\omega_x| \le a V_{\delta}} + \kappa_{\delta}(a)$.
Then, developing the product $\prod_{j=1}^k (\omega_{x_j}^{(a)} + \bomega_{x_j}^{(a)})$ yields
\begin{equation*}
	X_{\delta}(\bff) - X_{\delta}^{(a)}(\bff) = \sum_{J \subsetneq \intg{1}{k}}  \frac{1}{(V_{\delta})^k}\sum\limits_{x_1,\dots,x_k \in \Omega_{\delta}} f_{\delta}(x_1,\dots,x_k) \prod\limits_{j \in J} \omega_{x_j}^{(a)}\prod\limits_{j \notin J} \bomega_{x_j}^{(a)}.
\end{equation*}
By Jensen's and Minkowski inequality, we obtain for \(p >1\)
\begin{equation*}
	\big\|X_{\delta}(\bff) - X_{\delta}^{(a)}(\bff) \big\|_1 \le \sum_{J \subsetneq \intg{1}{k}} \frac{1}{(V_{\delta})^k} \bigg\| \sum\limits_{x_1,\dots,x_k \in \Omega_{\delta}} \bff(x_1,\dots,x_k) \prod\limits_{j \in J} \omega_{x_j}^{(a)}\prod\limits_{j \notin J} \bomega_{x_j}^{(a)} \bigg\|_p . 
\end{equation*}

Choose \(p\in (1,\gamma)\) and \(q\in (\gamma,2]\) such that \(\frac{p(q-\gamma)}{\gamma(q-p)} > \frac{k-1}{k}\) (\textit{i.e.}\ \(p\) sufficiently close to $\gamma$).
For every $J \subset \intg{1}{k}$, we apply Proposition~\ref{prop: general Lp-Lq inequality} with $\mathds{L}^p$ estimates for the random variables $\omega^{(a)}$ and $\mathds{L}^q$ estimates for the random variables~$\bomega^{(a)}$.
Thus, the terms in the sum above are bounded by
\begin{equation*}
	\frac{1}{(V_{\delta})^k} \Big(\frac{C_2}{p-1} (|\Omega|\vee 1)^{\frac1p - \frac1q} \Big)^k \|\bff_{\delta}\|_q  \bigg(\frac{\|\omega^{(a)}\|_p}{v_{\delta}^{1/p}} \bigg)^{|J|} \bigg(\frac{\|\bomega^{(a)}\|_q}{v_{\delta}^{1/q}}\bigg)^{k-|J|}  \,,
\end{equation*}
where we recall that~$\bff_{\delta}$ denotes the piecewise constant extension to $\Omega$ of $\bff_{\upharpoonright \Omega_{\delta}}$, \textit{i.e.}\, $\bff_{\delta}(y) =\bff(x)$ for every $y \in \cC_{\delta}(x)$.

Additionally, using \eqref{eq: asymptotic behavior for moments of order p-q}, we get that there is a constant \(C>0\) such that for any \(a\in (0,1)\), for \(\delta\) small enough we have
\begin{equation*}
\frac{\|\omega^{(a)}\|_p}{v_{\delta}^{1/p} V_{\delta}} \leq  \frac{C}{(\gamma-p)^{1/p}} a^{1-\frac{\gamma}{p}} 
\quad\text{ and }\quad
\frac{\|\bomega^{(a)}\|_q}{v_{\delta}^{1/q}V_{\delta}} \leq \frac{C}{(q-\gamma)^{1/q}}  a^{1-\frac{\gamma}{q}}\,. 
\end{equation*}
All together, we obtain that 
\begin{equation*}
	\begin{split}
		\big\|X_{\delta}(\bff) - X_{\delta}^{(a)}(\bff) \big\|_1 &\leq  (C')^k \, a^{ (1-\frac{\gamma}{q}) k} \|\bff_{\delta}\|_q \sum_{J \subsetneq \intg{1}{k}} a^{-(\frac{\gamma}{p}-\frac{\gamma}{q})  |J| } \\
		&\leq (C')^k a^{ (1-\frac{\gamma}{q}) k -  (\frac{\gamma}{p}-\frac{\gamma}{q}) (k-1)} \|\bff^{(\delta)}\|_q \,.
	\end{split}
\end{equation*}
where for the second inequality we have used that \(\frac{1}{p}-\frac{1}{q} >0\) and that \(|J|\leq k-1\) in the sum.
This concludes the proof, since we took \(p\) sufficiently close to \(1\) so that \(\frac{p(q-\gamma)}{\gamma (q-p)}> \frac{k-1}{k}\), \textit{i.e.}\ \((1-\frac{\gamma}{q}) k >  (\frac{\gamma}{p}-\frac{\gamma}{q}) (k-1)\).
\end{proof}

\end{proof}
\subsection{Polynomial chaos: proof of Theorems~\ref{thm: continuous polynomial chaos} and~\ref{thm: scaling limit of discrete polynomial chaos}}
{\label{subsec: polynomial chaos}}

We conclude this section with the proof of Theorems \ref{thm: continuous polynomial chaos} and \ref{thm: scaling limit of discrete polynomial chaos}. 
The proofs are straightforward and consist of assembling the results from Section \ref{section: Lp estimates} with those of this Section. 

\begin{proof}[Proof of Theorem \ref{thm: continuous polynomial chaos}]
	Let $\bpsi^{(k)}$, $k \ge 0$, be symmetric functions defined respectively on $\Omega^k$ such that~\eqref{hyp: infinite radius of convergence for Lq norms} holds.
	Then, by Theorem~\ref{thm: Lp estimates for continuous chaos}, for every $a \in (0,1]$ the random variable $\Phi^{(a)}$ from~\eqref{def: Phia} is well-defined and for every $p \in (1,\gamma)$ we have
	\begin{equation*}
		\big\|\bPhi^{(a)}\big\|_p \le \sum\limits_{k \ge 0} C_0^k \big\|\bpsi^{(k)}\big\|_q < +\infty \,.
	\end{equation*}
	(The fact that the sum is finite comes from the assumption~\eqref{hyp: infinite radius of convergence for Lq norms}.)
	Additionally, by Proposition~\ref{prop: martingale property of multiple integrals}, $(\bPhi^{(a)})_{a\in (0,1]}$ is a càdlàg time-reversed martingale with respect to the filtration $(\cF_a)_{a\in (0,1]}$.
	Since for \(p\in (1,\gamma)\) it is bounded in \(\mathds{L}^p\), it converges a.s.\ and in \(\mathds{L}^p\) to some well-defined random variable \(\bPhi\).
\end{proof}

We now move to the proof of Theorem~\ref{thm: scaling limit of discrete polynomial chaos}. 

\begin{proof}[Proof of Theorem \ref{thm: scaling limit of discrete polynomial chaos}]
	Let $\psi_{\delta}^{(k)}$, \(k\geq 0\), be symmetric functions defined on $\Omega_{\delta}^k$ satisfying the assumptions of Theorem \ref{thm: scaling limit of discrete polynomial chaos}. 
	For every $k \ge 0$, denote $\bpsi^{(k)}$ the limit of $\overline{\psi}_{\delta}^{(k)}$ as $ \delta \downarrow 0$. 
	Then Assumption \ref{hyp: truncation} tells that for every $C \ge 0$ we have
	\[
	\sum\limits_{k \ge 0} C^k \big\|\bpsi^{(k)}\big\|_q < +\infty \,.
	\] 
	Therefore, the family of functions $(\bpsi^{(k)})_{k \ge 0}$ satisfies the assumption of Theorem \ref{thm: continuous polynomial chaos}, which ensures that the following continuous chaos expansion is well-defined:
	\begin{equation*}
		\bPhi = \sum\limits_{k \ge 0} \frac{1}{k!} \int_{\Omega^k} \bpsi^{(k)}(x_1,\dots,x_k) \prod\limits_{j=1}^k \bzeta(\dd x_j) \,.
	\end{equation*}

	We are going to show that $\Phi_{\delta}$ converges in law to $\bPhi$. 
	The proof essentially consists in proving the term by term convergence of the chaos. 
	First, let us truncate the expansion: for $M \in \mathbb{N}$, define
	\begin{equation*}
		\Phi_{\delta}^{(M)} := \sum\limits_{k=0}^M \frac{1}{k!} \int_{\Omega^k} \psi_{\delta}^{(k)}(x_1,\dots,x_k) \prod\limits_{j=1}^k \zeta_{\delta}(\dd x_j),
		\ \ \text{ and } \ \
		\bPhi^{(M)} := \sum\limits_{k=0}^M \frac{1}{k!} \int_{\Omega^k} \bpsi^{(k)}(x_1,\dots,x_k) \prod\limits_{j=1}^k \bzeta(\dd x_j).
	\end{equation*}
	By Theorem \ref{thm: Lp estimates for discrete chaos}, for $p \in (1,\gamma)$ we have
	\begin{equation*}
		\big\| \Phi_{\delta}-\Phi_{\delta}^{(M)} \big\|_p \le \sum\limits_{k > M} (C_0)^k \big\| \psi_{\delta}^{(k)}\big\|_q .
	\end{equation*}
	Hence, using Assumption~\ref{hyp: truncation}, we obtain
	\begin{equation}{\label{eq: discrete truncation estimate}}
		\lim_{M \to +\infty} \limsup\limits_{\delta \downarrow 0} \big\|\Phi_{\delta}-\Phi_{\delta}^{(M)}\big\|_p = 0 \,.
	\end{equation}
	Similarly, using Theorem~\ref{thm: Lp estimates for continuous chaos}, for the continuous counterpart and Assumption~\ref{hyp: truncation}, we obtain
	\begin{equation}{\label{eq: truncation estimate}}
		\lim_{M \to +\infty} \big\|\bPhi-\bPhi^{(M)}\big\|_p = 0 \,.
	\end{equation}	

	Finally, Proposition \ref{prop: convergence of discrete multiple integrals} shows that the following convergence in distribution holds: for any fixed \(M\), as \(\delta\downarrow 0\),
	\begin{equation*}
		\begin{split}
			&\bigg(\zeta_{\delta}, \int_{\Omega} \psi_{\delta}^{(1)}(x) \zeta_{\delta}(\dd x),\dots,\int_{\Omega^M} \psi_{\delta}^{(M)}(x_1,\dots,x_M) \prod\limits_{j =1}^M \zeta_{\delta}(\dd x_j)\bigg) \\
			&\qquad \qquad \xrightarrow{\ (d)\ } \bigg(\bzeta, \int_{\Omega} \bpsi^{(1)}(x) \bzeta(\dd x),\dots,\int_{\Omega^M} \bpsi^{(M)}(x_1,\dots,x_M) \prod\limits_{j =1}^M \bzeta(\dd x_j)\bigg).
		\end{split}
	\end{equation*}
	Therefore $\Phi_{\delta}^{(M)}$ converges in distribution as $\delta \downarrow 0$ to $\bPhi^{(M)}$, jointly with the noise.
	Together with \eqref{eq: discrete truncation estimate} and \eqref{eq: truncation estimate}, this concludes the proof of Theorem~\ref{thm: scaling limit of discrete polynomial chaos}. 
\end{proof}

\section{Convergence of the disordered Gibbs measure and disorder relevance}{\label{section: statistical mechanics proofs}}

In this section, we prove the convergence of the disordered Gibbs measure, \textit{i.e.}\ Theorem~\ref{thm: convergence of the disordered probability}, in Section~\ref{section: convergence Gibbs proof}.
We also address the question of disorder relevance, \textit{i.e.}\ Theorem~\ref{thm: disorder relevance}, in Section~\ref{section: disorder relevance proof}

\subsection{Proof of Theorem \ref{thm: convergence of the disordered probability}}
\label{section: convergence Gibbs proof}

For every bounded, continuous function $G$ on $E$, define 
\begin{equation*}
	Z_{\Omega_{\delta}}^{\omega,\beta_{\delta}}(G) := \E_{\delta}^{\mathrm{ref}}\Big[ G(\sigma) \prod\limits_{x \in \Omega_{\delta}} (1+\beta_{\delta} \omega_x) \Big]\,,
\end{equation*}
so that with this notation we have 
\(
\P_{\Omega_{\delta}}^{\omega,\beta_{\delta}}(G) =   Z_{\Omega_{\delta}}^{\omega,\beta_{\delta}}(G) / Z_{\Omega_{\delta}}^{\omega,\beta_{\delta}} \,.
\)
Given a distribution $T \in \cD'(\Omega)$ and a test function $\bff \in \cD(\Omega)$ (infinitely differentiable with compact support), we denote by $\langle T, \bff \rangle$ the action of $T$ on $\bff$.

The following lemma is adapted from \cite{berger_lacoin_2021} and stated in a slightly more general context. 
For the sake of completeness, we include the proof here. 
It provides sufficient conditions for the discrete probability measure $\P_{\Omega_{\delta}}^{\omega,\beta_{\delta}}$ to admit a scaling limit.
An important feature of the proof is that it does not require the state space $\Omega$ to be bounded; this allows us to apply it to the long-range polymer model in the proof of Theorem~\ref{thm: convergence of the polymer measure}. 

\begin{lemma}{\label{lemma: how to prove convergence of disordered measures}}
	Suppose that there exist $\beta_{\delta}$ such that the following conditions are satisfied:
	\begin{enumerate}[label = {(\alph*)}]
		\item{\label{cond: i}}
		$\P_{\Omega_{\delta}}^{\mathrm{ref}}$ converges in $\cM_1(E)$ to a probability distribution $\bP_{\Omega}$;
		\item{\label{cond: ii}}
		for every function $G \in \cC_b(E)$, for every $\bff \in \cD(\Omega)$, there exist a random variable $\cZ_{\Omega}^{\bzeta}$ measurable with respect to $\bzeta$ such that, as $\delta \downarrow 0$,
		\[
		\big(\langle \bzeta_{\delta},\bff \rangle, Z_{\Omega_{\delta}}^{\omega,\beta_{\delta}}(G)\big) \xrightarrow{\ (d)\ } \big(\langle \bzeta,\bff \rangle, \cZ_{\Omega}^{\bzeta}(G)\big)\,.
		\]
		\item{\label{cond: iii}}
		The random variable $\cZ_{\Omega}^{\bzeta}(1)$ is almost surely strictly positive.
	\end{enumerate}
	Then there exist a random probability measure \(\bP_{\Omega}^{\bzeta} \in \cM_1(E)\) such that we have the following joint convergence in distribution in $H_{\rm loc}^{-s}(\Omega) \times \cM_1(E)$, $s > \frac{D}{2}$,
	\begin{equation*}
		(\zeta_{\delta}, \P_{\Omega_{\delta}}^{\omega,\beta_{\delta}}) \xrightarrow{(d)} (\bzeta, \bP_{\Omega}^{\bzeta}) \quad\text{ as } \delta\downarrow 0 \,. 
	\end{equation*} 
\end{lemma}

\begin{proof}
	First we prove convergence of finite-dimensional marginals.
	Since $\zeta_{\delta}$ and $\P_{\Omega_{\delta}}^{\omega,\beta_{\delta}}$ can be both viewed as linear forms, it is actually enough to prove convergence of one-dimensional marginals, which is exactly condition~\ref{cond: ii}.
	
	Second, we prove tightness of the couple $(\zeta_{\delta},\P_{\Omega_{\delta}}^{\omega,\beta_{\delta}})$.
	The family of random distributions $\{\zeta_{\delta}\}_{\delta \in (0,1)}$ is tight in $H^{-s}_{\rm loc}(\Omega)$, see Theorem~\ref{thm: functional convergence for the noise} in Appendix.
	Therefore, it only remains to prove that the family of random probability measures $(\P_{\Omega_{\delta}}^{\omega,\beta_{\delta}})_{\delta \in (0,1)}$ is tight. 
	By conditions~\ref{cond: ii} and~\ref{cond: iii}, $ Z_{\Omega_{\delta}}^{\omega,\beta_{\delta}}\xrightarrow{(d)} \mathcal{Z}_{\Omega}^{\bzeta,\hbeta}$ with $\mathcal{Z}_{\Omega}^{\bzeta,\hbeta}>0$ almost surely, hence there exists a sequence $(\gep_m)_{m \ge 0}$ going to zero such that, for all $\delta > 0$ and $ m \ge 0$
	\begin{equation*}
		\Pro[Z_{\Omega_{\delta}}^{\omega,\beta_{\delta}} \le \gep_m] \le 2^{-m}\,.
	\end{equation*}
	Since by assumption~\ref{cond: i}, the family of probability measures $\P_{\delta}^{\mathrm{ref}}$ converge as $\delta \downarrow 0$ in the space $\mathcal{M}_1(E)$, it is a tight family of $\cM_{1}(E)$.
	We can therefore find compact subsets $(K_m)_{m\geq 1}$ of $E$ such that for all $\delta > 0$ and $m \ge 0$, 
	\begin{equation*}
		\P_{\delta}^{\mathrm{ref}}(K_m) \ge 1- \frac{\gep_m}{4^m}\,.
	\end{equation*}
	
	Now let us define, for every $n \ge 1$
	\begin{equation*}
		\cK_n = \big\{ \mu \in \cM_1(E)  \;;\; \forall m > n,\, \mu(K_m^{c}) \le 2^{-m} \big\}\,.
	\end{equation*}
	The set $\cK_n \subset \cM_1(E)$ is closed, and any sequence in $\cK_n$ is tight, therefore it is a compact subset of $\cM_1(E)$.
	Additionally, by subadditivity, we have	
	\begin{equation*}
		\Pro\big[ \P_{\Omega_{\delta}}^{\omega,\beta_{\delta}} \notin \cK_n \big] \le \sum_{m > n} \Pro[\P_{\Omega_{\delta}}^{\omega,\beta_{\delta}}(K_m^{c}) > 2^{-m}]\,.
	\end{equation*}
	Using Markov inequality and the fact that $\Es[Z_{\Omega_{\delta}}^{\omega,\beta_{\delta}}(\Ind_{K_m^{c}})] = \P_{\delta}^{\mathrm{ref}}(K_m^c) \leq \gep_m 4^{-m}$, we get
	\begin{equation*}
		\begin{split}
			\Pro[\P_{\delta,\beta_{\delta}}^{\omega}(K_m^{c}) > 2^{-m}] 
			& \le  \Pro[Z_{\delta,\beta_{\delta}}(\Ind_{K_m^{c}}) > \gep_m 2^{-m}] + \Pro[Z_{\delta,\beta_{\delta}} \le \gep_m] \\
			& \le 2^{m} \gep_m^{-1} \mathrm P_{\delta}(K_m^c) + 2^{-m} \le 2 \cdot 2^{-m} \,.
		\end{split}
	\end{equation*}
	This shows that $\sup_{\delta \in (0,1)} \Pro[\P_{\Omega_{\delta}}^{\omega,\beta_{\delta}} \notin \cK_n ] \leq 2 \cdot 2^{-n}$, which can be made arbitrarily small by taking~\(n\) large. 
	This concludes the proof.
\end{proof}

We now apply Lemma~\ref{lemma: how to prove convergence of disordered measures} to the proof of Theorem~\ref{thm: convergence of the disordered probability}. 

\begin{proof}[Proof of Theorem~\ref{thm: convergence of the disordered probability}]
	All we need is to check that condition~\ref{cond: ii} of Lemma~\ref{lemma: how to prove convergence of disordered measures} is satisfied. 
	
	Let $G \in \cC_b(E)$.
	Arguing as in \eqref{eq: chaos expansion for L2 computation} and \eqref{eq: chaos expansion of the partition function}, the random variable $Z_{\Omega_{\delta}}^{\omega,\beta_{\delta}}(G)$ has the following polynomial chaos expansion:
	\begin{equation}
		Z_{\Omega_{\delta}}^{\omega,\beta_{\delta}}(G) = \E_{\delta}^{\mathrm{ref}}[G(\sigma)] +  \sum\limits_{k \ge 1} \frac{\hbeta^k}{k!} \sum\limits_{x_1,\dots,x_k \in \Omega_{\delta}} \psi_{\delta}^{(k)}(x_1,\dots,x_k;G) \prod\limits_{j=1}^k \frac{\omega_{x_j}}{V_{\delta}} \,,
	\end{equation} 
	recalling the definition~\eqref{def: generalized correlation functions} of the \(G\)-weighted \(k\)-point correlation function \(\psi_{\delta}^{(k)}(x_1,\dots,x_k;G)\).
	
	The function $G$ being bounded, we therefore have, for every $k \ge 1$ and $x_1,\dots,x_k \in \Omega_{\delta}$, 
	\[
	|\psi_{\delta}^{(k)}(x_1,\dots,x_k;G)| \le \|G\|_{\infty}  \psi_{\delta}^{(k)}(x_1,\dots,x_k) \,.
	\]
	This implies that the functions $\psi_{\delta}^{(k)}(\cdot,G)$ satisfy the assumption~\ref{hyp: truncation} of Theorem \ref{thm: scaling limit of discrete polynomial chaos}.
	By condition~\ref{cond: convergence general correlations}, they also satisfy the assumption~\ref{hyp: convergence of correlation functions} of Theorem \ref{thm: convergence of the disordered probability}. 
	Note that the convergence also holds for $k = 0$ thanks to condition~\ref{cond: invariance principle}.
	
	Therefore, Theorem~\ref{thm: scaling limit of discrete polynomial chaos} applies and yields condition~\ref{cond: ii}, using also that the evaluation against a test function is a continuous functional from $H^{-s}_{\rm loc}(\Omega)$ to $\mathbb{R}$.
\end{proof}

\subsection{Disorder relevance: proof of Theorem \ref{thm: disorder relevance}}
\label{section: disorder relevance proof}

The proof follows a standard change-of-measure argument.
Recall that in Theorem~\ref{thm: disorder relevance}, we assume that the random variable $J_{\delta} v_{\delta} \sum_{x \in \Omega_{\delta}} \sigma_x$ converges as $\delta \to 0$ to a strictly positive random variable.

\begin{lemma}
	Assume that $\hat{\beta}_{\delta} :=\beta_{\delta}  V_{\delta} J_{\delta}^{-1}$ verifies \(\lim_{\delta\downarrow 0} \hat{\beta}_{\delta}=+\infty\), and define the following event (for the environment $\omega$): 
	\[
	A_{\delta} = \{ \omega \colon \forall x \in \Omega_{\delta},\, |\omega_x| < a_{\delta} V_{\delta} \}\,, \quad \text{with} \quad a_{\delta} := (\hat{\beta}_{\delta})^{4/(\gamma-1)} \,.
	\]
	Then, we have that
	\begin{enumerate}
		\item \label{item: PAc}
		$ \lim\limits_{\delta \to 0} \Pro[(A_{\delta})^{c}] = 0$;
		\item \label{item: PtildeA}
		$ \lim\limits_{\delta \to 0} \Es[Z_{\Omega_{\delta}}^{\omega,\beta_{\delta}} \Ind_{A_{\delta}} ] = 0$. 
	\end{enumerate}
	This implies that $Z_{\Omega_{\delta}}^{\omega,\beta_{\delta}} \to 0$ in probability.
\end{lemma}

\begin{proof} 
	First of all, notice that by definition we have $\lim_{\delta \downarrow0}a_{\delta} =+\infty$ and also \(\lim_{\delta \downarrow0} a_{\delta}^{-(\gamma-1)/2} \hat{\beta}_{\delta} =+\infty\).
	Notice also that by Potter's bound (see \cite[Thm 1.5.4]{bingham_goldie_teugels_1987}) for any $\gep > 0$, there are constants~$C_{\gep}$ and~$c_{\gep}$ such that for every $a \ge 1$ and $\delta \in (0,1)$, 
	\begin{equation}{\label{eq: potter's bound}}
	c(\gep) a^{-\gep}  \varphi(V_{\delta}) \le \varphi(a V_{\delta}) \le C(\gep) a^{\gep} \varphi(V_{\delta}) \,.
	\end{equation}

	\smallskip
	Let us start with $(i)$. 
	Using subadditivity then Potter's bound~\eqref{eq: potter's bound} (say with \(\gep=\gamma/2\)), we have that 
	\begin{equation*}
		\Pro[(A_{\delta})^c] \le \mathrm{Card}(\Omega_{\delta}) \: \Pro[|\omega| \ge a_{\delta} V_{\delta}] \le C\, a_{\delta}^{- \gamma/2} |\Omega| \: \frac{1}{v_{\delta}} \Pro[|\omega| \ge V_{\delta}]\,,
	\end{equation*}
	where we have also used that \(|\Omega|= v_{\delta} \mathrm{Card}(\Omega_{\delta})\).
	Recalling~\eqref{eq: general scale of the noise} we have that \(\frac{1}{v_{\delta}} \Pro[|\omega| \ge V_{\delta}] \to 1\), so the above goes to $0$ as $\delta \downarrow 0$, since $\lim_{\delta \downarrow0}a_{\delta} =+\infty$.
	
	\smallskip
	For $(ii)$, by Fubini's theorem we have
	\begin{equation}{\label{eq: representation with size-bias}}
			\Es[Z_{\Omega_{\delta}}^{\omega,\beta_{\delta}} \Ind_{A_{\delta}} ] = \E_{\delta}^{\mathrm{ref}}\Big[ \bbE\Big[\Ind_{A_{\delta}} \prod_{x \in \Omega_{\delta}} (1+\beta_{\delta} \omega_x \sigma_x) \Big] \Big] = \E_{\delta}^{\mathrm{ref}}\Big[ \Pro_{\delta}^{\sigma}[A_{\delta}]\Big]\,,
	\end{equation}
	where, given a realization of the field $\sigma$, $\Pro_{\delta}^{\sigma}$ is the size-biased measure defined by 
	\begin{equation*}
		\frac{\dd \Pro_{\delta}^{\sigma}}{\dd \Pro}(\omega) = \prod_{x \in \Omega_{\delta}} (1+\beta_{\delta} \omega_x \sigma_x)\,.
	\end{equation*}
	Note that \(\Pro_{\delta}^{\sigma}\) is a probability distribution since \(\Es[\prod_{x \in \Omega_{\delta}} (1+\beta_{\delta} \omega_x \sigma_x)] = 1\) and that under $\Pro_{\delta}^{\sigma}$, the random variables $(\omega_{x})_{x \in \Omega_{\delta}}$ are independent.
	Furthermore, if $\sigma_x = 0$, the law of $\omega_x$ is the same as under~$\Pro$, and if $\sigma_x = 1$, the law of $\omega_x$ is $\tilde{\Pro}_{\delta}$, which is defined by
	\(\dd \tilde{\Pro}_{\delta}(\omega) = (1+\beta_{\delta} \omega) \dd\Pro(\omega)\).

	Now, we bound \ref{eq: representation with size-bias} by observing that $A_{\delta} \subset \{ \forall x \in \sigma, |\omega_x| \le a_{\delta} V_{\delta}\}$: we get that 
	\begin{equation*}
		\begin{split}
		\Es[Z_{\Omega_{\delta}}^{\omega,\beta_{\delta}} \Ind_{A_{\delta}} ] \le \E_{\delta}^{\mathrm{ref}}\Big[ \Pro_{\delta}^{\sigma}\big[\forall x \in \sigma, |\omega_x| \le a_{\delta} V_{\delta}\big] \Big] 
		& = \E_{\delta}^{\mathrm{ref}}\Big[ \prod\limits_{x \in \sigma} (1-\tilde{\Pro}_{\delta}(|\omega_x| > a_{\delta} V_{\delta})) \Big] \\
		& \le \E_{\delta}^{\mathrm{ref}}\bigg[\exp\Big(-\tilde{\Pro}_{\delta}\big(|\omega| > a_{\delta} V_{\delta}\big)\sum_{x \in \Omega_{\delta}} \sigma_x\Big) \bigg],
		\end{split}
	\end{equation*} 
	where, for the last inequality, we have simply used that $1-x \le e^{-x}$ for any \(x\). 
	Now, using the definition of $\tilde{\Pro}_{\delta}$, we obtain that
	\begin{equation*}
		\tilde{\Pro}_{\delta}\big(|\omega| > a_{\delta} V_{\delta}\big) =\Pro\big(|\omega| > a_{\delta} V_{\delta}\big) + \beta_{\delta}  \Es\big[ \omega \Ind_{\{|\omega| > a_{\delta} V_{\delta}\}}\big] \sim \frac{\gamma}{\gamma-1} \beta_{\delta} \varphi(a_{\delta} V_{\delta}) (a_{\delta} V_{\delta})^{-(\gamma -1)}\quad \text{ as } \delta \downarrow 0\,.
	\end{equation*} 
	Together with Potter's lower bound~\eqref{eq: potter's bound} (say with \(\gep =(\gamma-1)/2\)), this yields 
	\begin{equation*}
		\tilde{\Pro}_{\delta}\big(|\omega| > a_{\delta} V_{\delta}\big) \ge  c'\, \beta_{\delta} a_{\delta}^{-(\gamma-1)/2} \Pro\big(|\omega| > V_{\delta}\big) \geq c''\hat{\beta}_{\delta} a_{\delta}^{-(\gamma-1)/2} J_{\delta} v_{\delta} \,,
	\end{equation*}
	recalling the definition of \(\hat{\beta}_{\delta}\) and the fact that \( \Pro[|\omega| \ge V_{\delta}] \sim v_{\delta}\) (see~\eqref{eq: general scale of the noise}).
	We end up with 
	\begin{equation*}
		\Es[Z_{\Omega_{\delta}}^{\omega,\beta_{\delta}} \Ind_{A_{\delta}} ] \leq \E_{\delta}^{\mathrm{ref}}\bigg[\exp\Big(-c'' \hat{\beta}_{\delta} a_{\delta}^{-(\gamma-1)/2} J_{\delta} v_{\delta}\sum_{x \in \Omega_{\delta}} \sigma_x\Big) \bigg] \,,
	\end{equation*}
	which goes to \(0\) as \(\delta\downarrow 0\) since \(\lim_{\delta\downarrow0} \hat{\beta}_{\delta} a_{\delta}^{-(\gamma-1)/2} =+\infty\) and $J_{\delta}v_{\delta} \sum_{x \in \Omega_{\delta}} \sigma_x$ converges in distribution to a \textit{positive} random variable.
	This concludes the proof. 
\end{proof}

\section{Proof for disordered pinning model}{\label{section: proofs pinning}}
{\label{section: proof pinning}}

In this section we prove Theorem \ref{thm: convergence of the partition function of the pinning model} and Theorem \ref{thm: disorder relevance for the pinning model}. 
We recall that we have assumed that $1-\alpha< \frac{1}{\gamma}$, which is equivalent to $\gamma (1-\alpha) < 1$.

Consider for $N \in \bbN$ the discretization $\Omega_N := \{ \frac{n}{N}, 1 \le n \le N\}$ of $\Omega := (0,1)$. 
In this section, instead of working with a continuous discretization parameter $\delta$ as in Sections \ref{section: introduction} and \ref{section: main results}, we prefer the discrete parameter $N$ (with $\delta = N^{-1}$).
The volume of each cell is $v_N = N^{-1}$, and recall that $V_N$ is defined by the asymptotic relation~\eqref{eq: scale of the noise for the pinning}.
We also set \(\hat \beta_N := \beta_N u(N)V_N\), and recall that we assume that \(\lim_{N\to\infty} \hat \beta_N  = \hat \beta \in [0,\infty)\).

\subsection{Proof of Theorem \ref{thm: convergence of the partition function of the pinning model}}

A consequence of the strong renewal theorem~\eqref{eq: renewal theorem} is that there exists a constant $C$ such that for every $N \ge 1$,
\begin{equation}{\label{eq: uniform bound on the renewal function}}
\frac{1}{C} \le u(N) \ell(N) N^{1-\alpha} \le C. 
\end{equation}
Furthermore, denoting for every $N \ge 1$ and $t \in [0,1]$,
\begin{equation}{\label{eq: rescaled renewal function}}
	U_N(t) := \frac{u(\floor{N t})}{u(N)},
\end{equation}
the functions $(U_N)_{N \ge 1}$ converge uniformly on every compact subset of $(0,1]$ to $U(t) := t^{\alpha -1}$.

\smallskip
Arguing as in \eqref{eq: first chaos expansion of the partition function} and \eqref{eq: chaos expansion of the partition function}, we rewrite the partition function $Z_{N,h_N}^{\omega,\beta_N}$ as a polynomial chaos with respect to the variables $\omega$:
\begin{equation*}
	\begin{split}
		Z_{N,h_N}^{\omega,\beta_N} &= 1 + \sum\limits_{k \ge 1} \frac{(\beta_N)^k}{k!} \sum\limits_{t_1,\dots,t_k \in \Omega_N} \E_{N,h_N}\Big[\prod\limits_{j=1}^k \Ind_{N t_j \in \tau}\Big] \prod\limits_{j=1}^k \omega_{N t_j} \\
		&= 1 + \sum\limits_{k \ge 1} \frac{(\hbeta_N)^k}{k!} \sum\limits_{t_1,\dots,t_k \in \Omega_N^k} \psi_N^{(k)}(t_1,\dots,t_k) \prod\limits_{j=1}^k \frac{\omega_{N t_j}}{V_N}\,.
	\end{split}
\end{equation*}
Here, for every $k \ge 1$, $\psi_N^{(k)} : (\Omega_N)^k \to \bbR$ is the symmetric function which vanishes on the diagonals, defined for $0 = t_0 < t_1 < \dots < t_k \le 1$ (in $\Omega_N$) by
\begin{equation*}
	\psi_{N}^{(k)}(t_1,\dots,t_k) := u(N)^{-k} \, \P_{N,h_N}\left[\{N t_1,\dots,N t_k\} \subset \tau\right].
\end{equation*}
Its piecewise constant extension to $\Omega^k$ is such that for every $t_0:= 0 < t_1 < \dots < t_k \le 1$, 
\begin{equation*}
	\overline{\psi}_{N}^{(k)}(t_1,\dots,t_k) =  u(N)^{-k} \, \P_{N,h_N}\left[\{\floor{N t_1},\dots,\floor{N t_k}\} \subset \tau\right].
\end{equation*}
We now verify that the family of functions $(\psi_N^{(k)})_{k \ge 1}$ satisfy the assumptions of Theorem \ref{thm: scaling limit of discrete polynomial chaos} and identify their scaling limit. 

We recall the definition of the discrete homogeneous constrained partition function:  
\begin{equation*}
	Z_{M,h}^{c} := \E\Big[\exp\Big(h \sum\limits_{n=1}^M \Ind_{n \in \tau} \Big) \Ind_{ M \in \tau} \Big] . 
\end{equation*}
By the Markov property of the renewal process $\tau$, we obtain that for every $t_0 := 0 < t_1 < \dots < t_k \le 1$
\begin{equation*}
	\overline{\psi}_{N}^{(k)}(t_1,\dots,t_k) = \frac{1}{Z_{N,h_N}} \prod\limits_{j=1}^k \bigg(\frac{ Z_{\floor{N t_j} - \floor{N t_{j-1}},h_N}^{c}}{u(N)}\bigg) Z_{N - \floor{N t_k},h_N}.
\end{equation*}
For every $\alpha \in (0,1)$, the (modified) Mittag--Leffler function is defined by
\begin{equation*}
	E_{\alpha}(z) := \sum\limits_{k = 0}^{+\infty} \frac{z^{k} \Gamma(\alpha)^k}{\Gamma(\alpha k + 1)}.
\end{equation*}
The proof of the following lemma is essentially technical and is therefore postponed to Appendix~\ref{appendix: homogeneous partition functions}.
\begin{lemma}{\label{lemma: scaling limit of the homogeneous partition functions}}
	Recall the definitions~\eqref{eq: def of the continuum homogeneous partition function} of the (free and conditioned) continuum partition functions~\(\cZ_{t,\hh}\) and \(\cZ_{t,\hh}^{c}\).
	If $\lim_{N \to +\infty} \frac{M_N}{N} = t \in (0,1]$, then 	
	\begin{equation*}
		\lim_{N \to +\infty} Z_{M_N,h_N} = \cZ_{t,\hh} := E_{\alpha}(\hh t^{\alpha}),
		\qquad \text{and} \qquad
		\lim_{N \to +\infty} \frac{Z_{M_N,h_N}^c}{u(N)} = \cZ_{t,\hh}^{c} := \alpha t^{\alpha - 1} E_{\alpha}'(\hh t^{\alpha}),
	\end{equation*} 
	These convergences hold uniformly in $t$ on every compact subset of $(0,1]$.
	Furthermore there exists a constant $C=C_{\hat{h}}$ such that for every $1 \le M \le N$, 
	\begin{equation}{\label{eq: uniform bound on the conditioned partition function}}
		Z_{M,h_N} \leq C \quad \text{and} \quad Z_{M,h_N}^{c} \le C u(M). 
	\end{equation}
\end{lemma}

\noindent 
Therefore, for every $k \ge 1$, $(\overline{\psi}_N^{(k)})_{N \ge 1}$ converges to $\bpsi_{\hh}^{(k)}$ (defined in \eqref{eq: continuum correlation functions of the pinning model}) uniformly on every compact subset of $\{(t_1,\dots,t_k) \in (0,1)^k, \forall i \neq j, t_i \neq t_j \}$. 
Furthermore, the bound \eqref{eq: uniform bound on the conditioned partition function} implies that for every $t_0 := 0 < t_1 < \dots < t_k \le 1$,
\begin{equation*}
	\psi_N^{(k)}(t_1,\dots,t_k) \le C^k \prod\limits_{j=1}^k \frac{u(\floor{N t_j}-\floor{N t_{j-1}})}{u(N)},
\end{equation*}
which thanks to~\eqref{eq: uniform bound on the renewal function} gives
\begin{equation}{\label{eq: uniform bound on the correlation functions of the pinning}}
	\psi_N^{(k)}(t_1,\dots,t_k) \leq (C')^k u(N)^{-k} \prod\limits_{j=1}^k \frac{(\floor{N t_j}-\floor{N t_{j-1}})^{-(1-\alpha)}}{\ell(\floor{N t_j}-\floor{N t_{j-1}})}.
\end{equation}
Taking the limit $N \to +\infty$ yields, for every $k \ge 1$,
\begin{equation*}
	\bpsi_{\hh}^{(k)}(t_1,\dots,t_k) \le (C')^k \prod\limits_{j=1}^k \frac{1}{(t_j-t_{j-1})^{1-\alpha}}.
\end{equation*}

Let $q \in (\gamma, \frac{1}{1-\alpha})$.
We now use these inequalities to verify that 
\begin{enumerate}
	\item[1.] for every $k \ge 1$, $\bpsi_{\hh}^{(k)} \in L^q_s((0,1)^k)$.
	\item[2.] the functions $\bpsi_N^{(k)}$ satisfy hypothesis~\ref{hyp: truncation}.
	\item[3.] for every $k \ge 1$, $\bpsi_N^{(k)}$ converges in $L^q_s((0,1)^k)$ to $\bpsi_{\hh}^{(k)}$. 
\end{enumerate}

\paragraph{1} 
Denoting $\xi := 1-q(1-\alpha) > 0$, we obtain
\begin{equation*}
	\begin{split}
	\frac{1}{k!} \int_{(0,1)^k} \bpsi_{\hh}^{(k)}(t_1,\dots,t_k)^q \prod\limits_{j=1}^k \dd t_j &= \int_{0 < t_1 < \dots < t_k < 1} \bpsi_{\hh}^{(k)}(t_1,\dots,t_k)^q \prod\limits_{j=1}^k \dd t_j \\
	&\le (C')^{kq} \int_{0 < t_1 < \dots < t_k < 1} \prod\limits_{j=1}^k (t_j - t_{j-1})^{\xi - 1} \dd t_j.
	\end{split} 
\end{equation*}
Therefore, using the standard computation (see \cite[Lemma A.3]{berger_lacoin_2022}), for $\xi > 0$ we have
\begin{equation}{\label{eq: gamma integrals}}
	\int_{t_0:=0 < t_1 < \dots < t_k < t_{k+1} := t} \prod\limits_{j=1}^{k+1} (t_j - t_{j-1})^{\xi - 1} \prod\limits_{j=1}^k \dd t_j = t^{(k+1)\xi - 1} \frac{\Gamma(\xi)^{k+1}}{\Gamma((k+1) \xi)}\,,
\end{equation}
from which we deduce that 
\begin{equation}{\label{eq: estimates of the moments of the correlation functions of the pinning model}}
	\bigg(\frac{1}{k!} \int_{(0,1)^k} \bpsi_{\hh}^{(k)}(t_1,\dots,t_k)^q \prod\limits_{j=1}^k \dd t_j \bigg)^{1/q} \le (C')^k \bigg(\frac{\Gamma(\xi)^k}{\Gamma(\xi k+1)} \bigg)^{1/q} \,.
\end{equation} 

\paragraph{2} Let us now estimate the $L^q$ moment of $\overline{\psi}_{N}^{(k)}$ (we recall that the function $\overline{\psi}_{N}^{(k)}$ vanishes on the diagonals): we have
\begin{equation*}
	\begin{split}
		\frac{1}{k!} \int_{(0,1)^k} \overline{\psi}_{N}^{(k)}(t_1,\dots,t_k)^q \prod\limits_{j=1}^k \dd t_j &= \int_{0 < t_1 < \dots < t_k < 1} \overline{\psi}_{N}^{(k)}(t_1,\dots,t_k)^q \prod\limits_{j=1}^k \dd t_j \\
		&= \frac{1}{N^k} \sum\limits_{1 \le n_1 < \dots < n_k \le N} \psi_N^{(k)}\bigg(\frac{n_1}{N},\dots,\frac{n_k-n_{k-1}}{N}\bigg)^q \,.
	\end{split}
\end{equation*}
By \eqref{eq: uniform bound on the correlation functions of the pinning}, this is bounded by
\begin{equation*}
		\frac{(C')^k}{N^k u(N)^{qk}} \sum\limits_{1 \le n_1 < \dots < n_k \le N} \prod\limits_{j=1}^k \frac{(n_j-n_{j-1})^{\xi-1}}{\ell(n_j-n_{j-1})^q},
\end{equation*}
with $\xi = 1 - q(1-\alpha) > 0$ as above. 
Let $\gep > 0$ be such that $\xi > \gep$. 
Applying Proposition \ref{prop: technical estimate on slow varying functions} (say with \(\gep=\xi/2\)), there exist a constant $C$ such that for every $k \ge 1$ and $N \ge 1$,
\begin{equation*}
	 \sum\limits_{1 \le n_1 < \dots < n_k \le N} \prod\limits_{j=1}^k \frac{(n_j-n_{j-1})^{\xi-1}}{\ell(n_j-n_{j-1})^q} \le \frac{C^k}{\ell(N)^{kq}} N^{k \xi} \int_{0 < t_1 < \dots < t_k < 1} \prod\limits_{j=1}^k (t_j - t_{j-1})^{\frac12\xi-1} \dd t_j. 
\end{equation*}
Therefore, using the lower bound of \eqref{eq: uniform bound on the renewal function}, we finally obtain
\begin{equation}{\label{eq: bounded moment of pinning correlation functions}}
	\begin{split}
	\bigg(\frac{1}{k!} \int_{(0,1)^k} (\overline{\psi}_N^{(k)}(t_1,\dots,t_k))^q \prod\limits_{j=1}^k \dd t_j \bigg)^{1/q} &\le (C'')^k \bigg(\int_{0 < t_1 < \dots < t_k < 1} \prod\limits_{j=1}^k (t_j - t_{j-1})^{\frac12\xi-1} \dd t_j\bigg)^{1/q} \\
	& \leq (C'')^k \bigg(\frac{\Gamma(\frac12\xi)^k}{\Gamma( \frac12 \xi k+1)} \bigg)^{1/q} \,,
	\end{split}
\end{equation}
with the same calculation as in~\eqref{eq: estimates of the moments of the correlation functions of the pinning model}.
This decreases faster than exponentially as $k \to +\infty$, so the hypothesis~\ref{hyp: truncation} of Theorem \ref{thm: scaling limit of discrete polynomial chaos} is verified. 

\paragraph{3} For every $k \ge 1$, the functions $(\psi_N^{(k)})_{N \ge 1}$ converge uniformly on every compact subset of $\{(t_1,\dots,t_k) \in (0,1)^k, \forall i \neq j, t_i \neq t_j \}$. 
Furthermore, they are uniformly bounded in $L^q_s((0,1)^q)$ for every $q \in (\gamma, \frac{1}{1-\alpha})$.
By dominated convergence, this shows the convergence in $L^q_s((0,1)^q)$ for every $q \in (\gamma, \frac{1}{1-\alpha})$.
\qed

\subsection{Proof of Theorem \ref{thm: convergence of the pinning measure}}

We apply the general Theorem~\ref{thm: convergence of the disordered probability}. 

It is proved in \cite[Theorem A.8]{giacomin_2007} that the rescaled renewal process 
\[\tau^{(N)} := \frac{1}{N} \tau \subset \bbR_+\]
converges in distribution with respect to the Matheron topology, to the $\alpha$-stable regenerative set $\cA_{\alpha}$.
Hence the assumption~\ref{cond: invariance principle} is verified.

We now verify assumption~\ref{cond: convergence general correlations}. 
Let $G$ be a bounded continuous function on the space $\cC_{\infty}$. 
Adapting the definition~\ref{eq: general rescaled correlation function} to the pinning model, we define for every fixed $k \ge 1$, for every $0 < t_1 < \dots < t_k < 1$,
\begin{equation*}
	\psi_N^{(k)}(t_1,\dots,t_k;G) := u(N)^{-k} \E_{N,h_N}\Big[G\big(\frac{\tau}{N}\big) \Ind_{\floor{Nt_1},\dots,\floor{N t_k} \in \tau}\Big] \,.
\end{equation*}
which can be rewritten as 
\begin{equation*}
	\psi_N^{(k)}(t_1,\dots,t_k;G) = \psi_N^{(k)}(t_1,\dots,t_k) \E_{N,h_N}\Big[G\big(\frac{\tau}{N}\big)\mid \floor{Nt_1},\dots,\floor{N t_k} \in \tau\Big]\,.
\end{equation*}
We must show that the symmetric function $\psi_N^{(k)}(\, \cdot \,;G)$ converges in $L^q_s((0,1)^k)$ for every $q \in (\gamma,\frac{1}{1-\alpha})$.
Since $G$ is bounded, 
\[ \psi_N^{(k)}(t_1,\dots,t_k;G) \le \|G\|_{\infty} \psi_N^{(k)}(t_1,\dots,t_k)\,,\]
which implies that $\psi_N^{(k)}(\, \cdot \,;G)$ is uniformly bounded in $L^q_s((0,1)^k)$ by \eqref{eq: bounded moment of pinning correlation functions}.
Moreover, by \cite[Proposition A.8]{csz_2016_continuum}, for every fixed $T > 0$, the random closed set $\tau^{(N)} \cap [0,T]$ conditioned on $\floor{N T} \in \tau$, converges in distribution (for the Matheron topology) to the $\alpha$-stable regenerative set $\cA_{\alpha}$, conditioned on $T \in \cA_{\alpha}$. 
Therefore, for every fixed $0 < t_1 < \dots < t_k < 1$,
\[\E_{N,h_N}\Big[G\big(\frac{\tau}{N}\big) \mid \floor{Nt_1},\dots,\floor{N t_k} \in \tau\Big] \xrightarrow[N \to +\infty]{} \bE_{\hh}\Big[G(\cA_{\alpha}) \mid t_1,\dots,t_k \in \cA_{\alpha}\Big]\,.\]
Here, the conditional expectation $\bE_{\hh}[\, \cdot \,\mid t_1,\dots,t_k \in \cA_{\alpha}]$ is understood as the law of the concatenation of $k$ independent $\alpha$-stable regenerative sets $(\cA_{\alpha}^{(i)}\cap [0,t_i-t_{i-1}])_{1 \le i \le k}$ respectively conditioned to $t_i-t_{i-1} \in \cA_{\alpha}^{(i)}$ for every $i \in \intg{1}{k}$ and one unconditioned $\alpha$-stable regenerative set $\cA_{\alpha}^{(k+1)}\cap [0,1-t_k]$.

Finally, it is proved in~\cite{faugere_lacoin} that the continuum partition function $\cZ_{\hh}^{\bzeta,\hbeta}$ is almost surely stricly positive, which concludes the proof.

\subsection{Proof of Theorem \ref{thm: disorder relevance for the pinning model}}

By Theorem \ref{thm: disorder relevance}, it is enough to show that the random variable 
\[ 
	D_N := \frac{1}{N u(N)} \sum\limits_{n=1}^N \Ind_{n \in \tau} 
\]
converges in distribution as $N \to +\infty$ to a strictly positive random variable.
The Laplace transform of~$D_N$ is given by 
\begin{equation*}
	\E_{N,h_N}[\expo^{\lambda D_N} ] = \E_{N,h_N}\Big[ \exp\Big(\lambda_N \sum\limits_{n=1}^N \Ind_{n \in \tau}\Big)\Big] = \frac{Z_{N,h_N+\lambda_N}}{Z_{N,h_N}}\,,
\end{equation*}
with $\lambda_N := \frac{\lambda}{N u(N)}$. 
By Lemma \ref{lemma: scaling limit of the homogeneous partition functions} (or \cite[Thm.~3.1]{sohier_2009}), we have that 
\[
\lim_{N\to\infty} \E_{N,h_N}[\expo^{\lambda D_N} ] =\frac{\cZ_{1,\hh+\lambda}}{\cZ_{1,\hh}} = \bE_{\alpha,\hh}\Big[ \expo^{\lambda L_1^{(\alpha)}} \Big] \,,
\]
where \(\dd \bP_{\alpha,\hh} = \frac{1}{\cZ_{1,\hh}} \expo^{\hh L_1^{(\alpha)}} \dd \bP\) is the law of the continuum homogeneous pinning model, with \(\bP\) the law of an \(\alpha\)-stable subordinator \(\sigma^{(\alpha)}\) and $L_1^{(\alpha)} = \inf \{ s \ge 0, \sigma_s^{(\alpha)} \ge t \}$ its local time (see Section~\ref{subsec: homogeneous pinning} for an overview).
The right-hand side is the Laplace transform of the local time  under \(\bP_{\alpha,\hh}\), which is a.s.\ strictly positive.
This concludes the proof.
\qed

\section{Proofs for long-range directed polymer}{\label{section: proofs polymer}}

In this section, we prove Theorems \ref{thm: convergence of the polymer partition function}, Theorem~\ref{thm: convergence of the polymer measure} and \ref{thm: disorder relevance for the polymer}. 
Let us stress that in this case, one cannot directly apply Theorems~\ref{thm: convergence of the partition function}-\ref{thm: convergence of the disordered probability} since \(\Omega = (0,1) \times \bbR^{d}\) is not bounded.
We therefore introduce some truncation argument, considering the random walk only in a bounded strip.

Denote for every $N \ge 1$,  $S_N^* = \max_{1 \le n \le N} \frac{|S_n|}{a_N}$. 
Let $f : \bbR_+ \to [0,1] $ be a smooth non-increasing function such that $f(x) = 1$ if $x \le 1$ and  $f(x) = 0$ if $x > 2$, and for every $A > 0$, let $f_A(x) = f(x/A)$.
Then, define the truncated partition function
\begin{equation*}
	Z_{N,A}^{\omega,\beta_N} := \E\Big[f_A(S_N^*) \prod\limits_{n=1}^N (1+\beta_N \omega_{n,S_n}) \Big].
\end{equation*}

Let us recall that \((\bX_t)_{t\geq 0}\) is the multivariate \(\alpha\)-stable Lévy process that arises as the scaling limit of the random walk, with density \(g_t(x)\) which is continuous and bounded (see e.g.\ \cite{watanabe_2007}, which also give their rate of decay).
We also denote by \(\bX^* = \sup_{t\in [0,1]} |\bX_t|\), and we introduce the formal notation 
\[
\bP\big( \, \cdot\, \big|\,  \bX_{t_1} = x_1 ,\ldots, \bX_{t_k} =x_k \big)
\]
for the law of the concatenation of bridges of~\(\bX\) from~\(x_{i-1}\) to~\(x_{i}\) between times~\(t_{i-1}\) and~\(t_i\) (this can be defined as the scaling limit of a random walk \(( \frac{1}{a_N} S_{\lfloor tN \rfloor})_{t\geq 0}\) conditioned to have \(S_{\lfloor t_i N \rfloor} =  \lfloor x_i a_N \rfloor\) for \(1\leq i \leq k\), see~\cite{liggett_1970}).

For a bounded continuous function $G : \cD \to \bbR$, define 
\[ Z_N^{\omega,\beta_N} := \E\Big[G(S^{(N)}) \prod\limits_{n=1}^N (1+\beta_N \omega_{n,S_n})\Big]\,.\]
In section~\ref{sec: extension to functional partition function}, we explain how the proof of Theorem~\ref{thm: convergence of the polymer partition function} can be adapted to prove the convergence in law of $Z_N^{\omega,\beta_N}(G)$.

\subsection{Proof of Theorem \ref{thm: convergence of the polymer partition function}}

The core of the proof is the convergence of the truncated partition function $Z_{N,A}^{\omega,\beta_N}$.
\begin{lemma}{\label{lemma: convergence of the truncated polymer partition function}}
	Assume that $\lim_{N\to\infty} \beta_N V_N (a_N)^{-d}= \hbeta \in [0,\infty)$.
	Then, for every $A > 0$, we have the following convergence in distribution: 
	\begin{equation}
	Z_{N,A}^{\omega,\beta_N} \xrightarrow[\ N\to\infty\ ]{(d)}
	\cZ_A^{\bzeta, \hbeta} = \sum\limits_{k \ge 0} \hbeta^k \int_{0 < t_1 < \dots < t_k < 1} \int_{(\bbR^{d})^k} \bpsi_A^{(k)}\big((t_1,x_1),\dots,(t_k,x_k)\big) \prod\limits_{j=1}^k \bzeta(\dd t_j \dd x_j),
	\end{equation}
	where $\bpsi_A^{(k)}\big((t_1,x_1),\dots,(t_k,x_k)\big) := \prod\limits_{j=1}^k g_{t_j-t_{j-1}}(x_j - x_{j-1})\, \E[f_A(\bX^*) \mid \bX_{t_1} = x_1, \dots , \bX_{t_k} = x_k]$. 

	\noindent
	The convergence also holds for the finite-dimensional marginal distributions of the process $(Z_{N,A}^{\omega,\beta_N})_{A > 0}$.
	
	\noindent
	Furthermore, for every $p \in [1,\gamma)$, $\lim\limits_{N \to +\infty} \Es[|Z_{N,A}^{\omega,\beta_N}|^p] = \Es[|\cZ_A^{\bzeta,\hbeta}|^p] < +\infty$.  	 
\end{lemma}

Before proving this lemma, let us first see how it implies Theorem \ref{thm: convergence of the polymer partition function}. 

\begin{proof}[Proof of Theorem \ref{thm: convergence of the polymer partition function}]
Since $f_A \le 1$, we have
\begin{equation}{\label{eq: uniform inequality on approximation of partition functions}}
	\Es\Big[|Z_N^{\omega,\beta_N}- Z_{N,A}^{\omega,\beta_N}|\Big] = \Es[Z_N^{\omega,\beta_N}] - \Es[Z_{N,A}^{\omega,\beta_N}] = \E[1-f_A(S_N^*)] \leq \P[S_N^* \geq A ]\,.
\end{equation} 
Now, by the invariance principle (see e.g.\ \cite[Thm.~7.1]{resnick_2007}), we have that~\(S_N^*\) converges in distribution to \(\bX^*\) as $N \to +\infty$, so we get 
\begin{equation}{\label{eq: uniform discrete approximation for the polymer partition function}}
	\limsup\limits_{N \to +\infty} \Es\Big[|Z_N^{\omega,\beta_N}- Z_{N,A}^{\omega,\beta_N}|\Big] \le  \bP( \bX^* \geq A) \xrightarrow[A\to\infty]{} 0 \,.
\end{equation}

For \(A\leq B\), notice that we have $0\leq Z_{N,A}^{\omega,\beta_N} \le Z_{N,B}^{\omega,\beta_N}$.
Hence, Lemma~\ref{lemma: convergence of the truncated polymer partition function} implies that $0\leq \cZ_A^{\bzeta,\hbeta} \le \cZ_B^{\bzeta,\hbeta}$ almost surely. 
In particular, we have that \(\cZ^{\bzeta,\hbeta} :=\lim_{A\to\infty}\cZ_A^{\bzeta,\hbeta}\) exists a.s.\
Let us now show that \(\cZ^{\bzeta,\hbeta} <+\infty\) a.s.\
For this, notice that since $\cZ_A^{\bzeta,\hbeta} \le \cZ_B^{\bzeta,\hbeta}$ we have
\begin{equation*}
	\Es[|\cZ_B^{\bzeta,\hbeta}-\cZ_A^{\bzeta,\hbeta}|] = \Es[\cZ_B^{\bzeta,\hbeta}] - \Es[\cZ_A^{\bzeta,\hbeta}] = \lim\limits_{N \to +\infty} \big(\Es[Z_{N,B}^{\omega,\beta_N}] - \Es[Z_{N,A}^{\omega,\beta_N}]\big)\,.
\end{equation*}
By the invariance principle (\cite[Thm.~7.1]{resnick_2007}), $S_N^*$ converges in distribution to $\bX^*$, so that 
\begin{equation}{\label{eq: expectation of the truncated partition function}}
	\lim\limits_{N \to +\infty} \Es[Z_{N,A}^{\omega,\beta_N}] = \lim\limits_{N \to +\infty} \E[f_A(S_N^*)] = \bE[f_A(\bX^*)]\,,
\end{equation}
and therefore
\begin{equation*}
	\Es[|\cZ_B^{\hbeta}-\cZ_A^{\hbeta}|] = \bE[f_B(\bX^*) - f_A(\bX^*)]\leq \bP( \bX^* \geq A) \xrightarrow[A\to\infty]{} 0 \,.
\end{equation*}
This shows that the family of random variables $(\cZ_A^{\hbeta})_{A \ge 1}$ is a Cauchy family in \(\mathds{L}^1\), so the limit \(\cZ^{\bzeta,\hbeta} :=\lim_{A\to\infty}\cZ_A^{\bzeta,\hbeta}\) also holds in \(\mathds{L}^1\).
Therefore, Lemma \ref{lemma: convergence of the truncated polymer partition function}, together with \eqref{eq: uniform discrete approximation for the polymer partition function}, concludes the proof of Theorem~\ref{thm: convergence of the polymer partition function}. 	
\end{proof}

\begin{proof}[Proof of Lemma \ref{lemma: convergence of the truncated polymer partition function}]
	A fundamental tool that we will need is the local limit theorem, see e.g.\ \cite{doney_1991,griffin_1986}, that we now recall for convenience:
	\begin{equation}{\label{eq: local limit theorem}}
		\sup\limits_{x \in \bbZ^{d}} \Big| (a_n)^{d}\, \P(S_n=x) - g_1\Big(\frac{x}{a_n}\Big) \Big| \xrightarrow[\ n \to +\infty\ ]{} 0 \,.
	\end{equation} 
	Since the density \(g_1\) is continuous and bounded, there exists a constant $C$ such that for every $n \in \bbN$ and $x \in \bbZ^d$,
	\begin{equation}{\label{eq: uniform bound on discrete density}}
		\P(S_n=x)\le \frac{C}{(a_n)^{d}}\,. 
	\end{equation}
	
	We are now going to apply Theorem \ref{thm: scaling limit of discrete polynomial chaos}, and to this end, we need to specify how the convergence of~$Z_{N,A}^{\omega,\beta_N}$ fits into its framework.  
	First of all, the polynomial chaos expansion of $Z_{N,A}^{\omega,\beta_N}$ is given by:
	\begin{equation*}
		Z_{N,A}^{\omega,\beta_N} = \E[f_A(S_N^*)] + \sum\limits_{k \ge 1} (\beta_N)^k \sumtwo{1 \le n_1 < \dots < n_k \le N}{x_1, \dots , x_k \in \bbZ^{d}} \E\big[f_A(S_N^*)\, \Ind_{\{S_{n_1} = x_1, \dots, S_{n_k} = x_k\}}\big] \prod\limits_{j=1}^k \omega_{n_j,x_j} \,. 
	\end{equation*}
	Set $\Omega_A = (0,1) \times \{ x \in \bbR^{d}, |x| < 2 A\}$, which is a bounded open set of $\bbR^{1+d}$ and for $N \ge 1$, let $\Omega_{N,A} = \{ (\frac{n}{N}, \frac{x}{a_N}),\, 1 \le n \le N, x \in \bbZ^{d} \text{ with } |x| < 2 A a_N \}$.  
	In this section, we also work with the discretization parameter $N$ instead of $\delta$ (with $\delta = N^{-1}$). 
	The discretization $\Omega_{N,A}$ naturally yields a tesselation of $\Omega$ with cells of volume $v_N = N^{-1} (a_N)^{-d}$; recall that $V_N$ is defined by the asymptotic relation \ref{eq: scale of the noise for polymers}. 
	For this model, $J_N = (a_N)^d$ by the local limit theorem~\eqref{eq: local limit theorem}. 
	Let us also set \(\hat \beta_N := \beta_N V_N (a_N)^{-d}\) and recall that we assume that \(\lim_{N\to\infty} \hat{\beta}_N = \hat \beta \in [0,\infty)\).
	
	We can therefore rewrite the polynomial chaos expansion as follows: 
	\begin{equation*}
		Z_{N,A}^{\omega,\beta_N} = \E[f_A(S_N^*)] + \sum\limits_{k \ge 1} \frac{(\hbeta_N)^k}{k!} \sum\limits_{(t_1,x_1),\dots,(t_k,x_k) \in \Omega_{N,A}} \psi_{N,A}^{(k)}\big((t_1,x_1),\dots,(t_k,x_k)\big) \prod\limits_{j=1}^k \frac{\omega_{N t_j, a_N x_j}}{V_N} \,.
	\end{equation*}
	Here, for every $k \ge 1$, $\psi_{N,A}^{(k)}: \Omega_{N,A}^k \to \bbR_+$ is the symmetric function vanishing on the diagonals defined, for every $(t_1,x_1),\dots,(t_k,x_k) \in \Omega_{N,A}$ with $t_1 < \dots < t_k$, by
	\begin{equation}
		\psi_{N,A}^{(k)}\big((t_1,x_1),\dots,(t_k,x_k)\big) := (a_N)^{kd} \E[f_A(S_N^*) \Ind_{\{S_{N t_1} = a_N x_1, \dots, S_{N t_k} = a_N x_k\}}] \,.
	\end{equation}
	Its piecewise extension to $\Omega_A$ is denoted $\overline{\psi}_{N,A}^{(k)}$.
	Now, we need to identify the scaling limits of the functions $\overline{\psi}_{N,A}^{(k)}$ and show that they satisfy the assumptions of Theorem \ref{thm: scaling limit of discrete polynomial chaos}. 

	Let us introduce the notation $\psi_N^{(k)} := \psi_{N,+\infty}^{(k)}$, with by convention $f_{+\infty} :=1 $: we can then write
	\begin{equation*}
		\psi_{N,A}^{(k)}\big((t_1,x_1),\dots,(t_k,x_k)\big) = \psi_N^{(k)}\big((t_1,x_1),\dots,(t_k,x_k)\big) \E\left[ f_A(S_N^*) \, \mid \, S_{N t_j} = a_N x_j \, \forall j \in \intg{1}{k}\right]\,.
	\end{equation*}
	By the Markov property of the random walk, when $t_1 < \dots < t_k$ we get
	\begin{equation*}
		\psi_N^{(k)}\big((t_1,x_1),\dots,(t_k,x_k)\big) = (a_N)^{kd}  \prod\limits_{j=1}^k \P\big( S_{N (t_j - t_{j-1})} =a_N (x_j - x_{j-1}) \big)\,.
	\end{equation*}
	Therefore, by the local limit theorem \eqref{eq: local limit theorem} and the invariance principle for conditioned random walk~\cite{liggett_1970}, for every $k \ge 1$, the functions $\overline{\psi}_N^{(k)}$ and  $\overline{\psi}_{N,A}^{(k)}$ converge pointwise as $N \to +\infty$ respectively to the symmetric functions $\bpsi^{(k)}$ and  $\bpsi_A^{(k)}$ given for every $0 < t_1 < \dots < t_k < 1$ and $x_1,\dots,x_k \in \bbR^{d}$ by
	\begin{equation}
		\label{def: continuous correlations long-range}
		\begin{split}
		\bpsi^{(k)}\big((t_1,x_1),\dots,(t_k,x_k)\big) 
		& = \prod_{j=1}^k g_{t_j - t_{j-1}}(x_j - x_{j-1})\,, \\
		\bpsi_A^{(k)}\big((t_1,x_1),\dots,(t_k,x_k)\big) 
		& = \prod_{j=1}^k g_{t_j - t_{j-1}}(x_j - x_{j-1}) \, \E\big[ f_A(\bX^*) \mid \bX_{t_j} = x_j \, \forall j \in \intg{1}{k} \big]\,,
		\end{split}
	\end{equation}
	where we recall that  $\bP(\, \cdot \, | \, \bX_{t_j} = x_j \, \forall j \in \intg{1}{k})$ is a shorthand for the concatenation of $k$ independent bridges connecting $(t_{j-1},x_{j-1})$ to $(t_j,x_j)$ for every $j \in \intg{1}{k}$. 
	
	In order to prove that this convergence actually holds in $L^q_s$ and that $(\psi_{N,A}^{(k)})$ satisfy condition~\ref{hyp: truncation} of Theorem~\ref{thm: scaling limit of discrete polynomial chaos}, we show the following.
	For every $ q \in (\gamma,1+\frac{\alpha}{d})$, denote $\xi = \xi_q := \frac{d}{\alpha}(1+\frac{\alpha}{d}-q) > 0$. 
	Then there exists a constant $C$ such that the following hold:
	\begin{enumerate}
		\item  For every $k \ge 1$, 
		\begin{equation}{\label{eq: bound on the continuous moments}}
			\frac{1}{k!} \int_{\Omega_A^k} \bpsi_A^{(k)} \big((t_1,x_1),\dots,(t_k,x_k)\big)^q \prod\limits_{j=1}^k \dd t_j \dd x_j \le \frac{C^k}{\Gamma(\xi k + 1)}\,.
		\end{equation}
		
		\item For every $k \ge 1$ and $N \ge 1$,
		\begin{equation}{\label{eq: bound on the discrete moments}}
			\frac{1}{k!} \int_{\Omega_A^k} \overline{\psi}_{N,A}^{(k)} \big((t_1,x_1),\dots,(t_k,x_k)\big)^q \prod\limits_{j=1}^k \dd t_j \dd x_j \le \frac{C^k}{\Gamma(\frac12 \xi k+1)}\,.
		\end{equation}
	\end{enumerate}
	Notice that, since $\psi_{N,A}^{(k)} \le \psi_N^{(k)}$ and $\bpsi_{A}^{(k)} \le \bpsi^{(k)}$, it is enough to work with $\psi_N^{(k)}$ and $\bpsi^{(k)}$. 
	Since these bounds hold for every $q \in (\gamma,1+\frac{\alpha}{d})$, the dominated convergence theorem implies that the pointwise convergence of  $\overline{\psi}_{N,A}^{(k)}$ to $\bpsi_{A}^{(k)}$ extends to convergence in $L^q_s$.
	Hence the condition~\ref{hyp: convergence of correlation functions} of Theorem~\ref{thm: scaling limit of discrete polynomial chaos} is also satisfied.

	\smallskip
	Let us start with~\eqref{eq: bound on the continuous moments}.
	For every $ k \ge 1 $, we have
	\begin{equation*}
		\frac{1}{k!} \int_{\Omega_A^k} \bpsi\big((t_1,x_1),\dots,(t_k,x_k)\big)^q \prod\limits_{j=1}^k \dd t_j \dd x_j 
		\le \hbeta^k \int_{0 < t_1 < \dots < t_k < 1} \int_{(\bbR^{d})^k}  \prod\limits_{j=1}^k g_{t_j-t_{j-1}}(x_j-x_{j-1})^q \prod\limits_{j=1}^k \dd t_j \dd x_j \,.
	\end{equation*}
	Now, the scale-invariance property of the function $g$ gives that for every $t > 0$,
	\begin{equation*}
		\int_{\bbR^d} g_t(x)^q \dd x = c_q\, t^{\xi-1} \,,
	\end{equation*}
	where $\xi = \xi_q:= \frac{d}{\alpha}(1+\frac{\alpha}{d}-q) > 0$ and $c_q = \| g \|_{L^q(\bbR^d)}^q$ is finite since $g \in L^1(\bbR^d) \cap L^{\infty}(\bbR^d)$.
	Therefore, the left-hand side of~\eqref{eq: bound on the continuous moments} is bounded by 
	\begin{equation*}
	(c_q)^k \int_{0 < t_1 < \dots < t_k < 1} \prod\limits_{j=1}^k (t_j-t_{j-1})^{\xi-1} \prod\limits_{j=1}^k \dd t_j = \frac{(c_q \Gamma(\xi))^k}{\Gamma( \xi k+1)} \,,
	\end{equation*}
	having used~\eqref{eq: gamma integrals}.
	This proves \ref{eq: bound on the continuous moments}.

	\smallskip
	We now turn to~\eqref{eq: bound on the discrete moments}.
	For every \(k\geq 1\), the left-hand side is equal to
	\begin{multline*}
		\int_{0 < t_1 < \dots < t_k} \int_{(\bbR^d)^k} \overline{\psi}_N^{(k)}\big((t_1,x_1),\dots,(t_k,x_k) \big)^q \prod\limits_{j=1}^k \dd t_j \dd x_j \\
		= (v_N (a_N)^{q d})^k \sumtwo{1 \le n_1 < \dots < n_k \le N}{x_1,\dots,x_k \in \bbZ^{d}} \prod\limits_{j=1}^k  \P(S_{n_j-n_{j-1}}=x_j - x_{j-1})^q \,.
	\end{multline*}
	Applying the uniform bound \eqref{eq: uniform bound on discrete density} to \(\P(S_{n_j-n_{j-1}}=x_j - x_{j-1})^{q-1}\) and using that $v_N = N^{-1} a_N^{-d}$, we get that this is bounded by
	\begin{equation*}
		C^k (a_N)^{-k(q-1)d} N^{-k}  \sum\limits_{1 \le n_1 < \dots < n_k \le N} \prod\limits_{j=1}^k \frac{1}{(a_{n_j-n_{j-1}})^{(q-1)d}} \sum_{x_1,\dots,x_k \in \bbZ^d} \prod\limits_{j=1}^k \P(S_{n_j-n_{j-1}}=x_j - x_{j-1}) .
	\end{equation*}
	Now we can use that $\sum_{x \in \bbZ^{d}} \P(S_n=x) = 1$ for every $ n \in \bbN$ to sum over \(x_1,\ldots, x_k\).
	Then, writing that $a_n = \ell(n) n^{1/\alpha}$ for some slowly-varying function \(\ell(\cdot)\) and recalling that \(\xi := \frac{d}{\alpha}(1+\frac{\alpha}{d}-q)\), we have that \((a_n)^{(q-1)d} = n^{-(\xi-1)} \ell(n)^{(q-1)d}\).
	We then obtain that the left-hand side of~\eqref{eq: bound on the discrete moments} is bounded by
	\begin{equation*}
		\begin{split}
			&C^k (a_N)^{k(q-1)d} N^{-k} \sum\limits_{1 \le n_1 < \dots < n_k \le N} \prod\limits_{j=1}^k \frac{(n_j-n_{j-1})^{\xi-1}}{\ell(n_j-n_{j-1})^{(q-1)d}} \\
	 		& \quad \leq (C')^k (a_N)^{k(q-1)d} N^{-k}\frac{N^{k\xi}}{\ell(N)^{k(q-1) d}} \int_{0 < t_1 < \dots < t_k < 1} \prod\limits_{j=1}^k (t_j-t_{j-1})^{\frac12 \xi-1} \prod\limits_{j=1}^k \dd t_j = (C')^k \frac{\Gamma(\frac12 \xi)^k}{\Gamma(\frac12\xi k +1)} \,,
		\end{split}
	\end{equation*}
	where we have used Proposition~\ref{prop: technical estimate on slow varying functions} (with \(\gep= \xi/2\)) for the first inequality and then~\eqref{eq: gamma integrals}.
	This yields~\eqref{eq: bound on the discrete moments} and concludes the proof.
\end{proof}

\subsection{Adapting the proof for a general function $G$}{\label{sec: extension to functional partition function}}

Without loss of generality, we assume that the function $G$ is positive. 
As we did in the case $G=1$, we define a truncated version of $Z_N^{\omega,\beta_N}(G)$ for every $A > 0$ by
\begin{equation*}
	Z_{N,A}^{\omega,\beta_N}(G) := \E\Big[f_A(S_N^*) G(S^{(N)}) \prod\limits_{n=1}^N (1+\beta_N \omega_{n,S_n}) \Big].
\end{equation*}
Lemma~\ref{lemma: convergence of the truncated polymer partition function} can be adapted as follows:

\begin{lemma}{\label{lemma: convergence of the truncated partition function with general G}}
	Assume that $\lim_{N\to\infty} \beta_N V_N (a_N)^{-d}= \hbeta \in [0,\infty)$.
	Then, for every $A > 0$, we have the following convergence in distribution: 
	\begin{equation*}
		Z_{N,A}^{\omega,\beta_N}(G) \xrightarrow[\ N\to\infty\ ]{(d)}
		\cZ_A^{\bzeta, \hbeta}(G) = \sum\limits_{k \ge 0} \hbeta^k \int_{0 < t_1 < \dots < t_k < 1} \int_{(\bbR^{d})^k} \bpsi_A^{(k)}\big((t_1,x_1),\dots,(t_k,x_k);G\big) \prod\limits_{j=1}^k \bzeta(\dd t_j \dd x_j),
	\end{equation*}
	where $\bpsi_A^{(k)}\big((t_1,x_1),\dots,(t_k,x_k);G\big) := \prod\limits_{j=1}^k g_{t_j-t_{j-1}}(x_j - x_{j-1})\, \E[f_A(\bX^*) G(\bX) \mid \bX_{t_1} = x_1, \dots , \bX_{t_k} = x_k]$. 
\end{lemma}
This lemma implies the convergence of $Z_N^{\omega,\beta_N}(G)$.
Indeed, since we also have for every $0 \le A \le B$, \[Z_{N,A}^{\omega,\beta_N}(G) \le Z_{N,B}^{\omega,\beta_N}(G)\,,\]
we get
\begin{equation*}
	\Es\Big[|Z_N^{\omega,\beta_N}(G)-Z_{N,A}^{\omega,\beta_N}(G)|\Big] = \E[G(S^{(N)})(1-f_A(S_N^*))] \le \|G\|_{\infty} \P[S_N^* \ge A]
\end{equation*}
and similarly,
\begin{equation*}
	\Es\Big[|\cZ_B^{\bzeta,\hbeta}(G)-\cZ_A^{\bzeta,\hbeta}(G)|\Big] \le \|G\|_{\infty} \bP[\bX^* \ge A]\,.
\end{equation*}
Therefore,
\begin{equation*}
	\lim\limits_{A \to +\infty} \limsup\limits_{N \to +\infty} \Es\Big[|Z_N^{\omega,\beta_N}(G)-Z_{N,A}^{\omega,\beta_N}(G)|\Big] = 0\,,
\end{equation*}
and the family $(\cZ_A^{\bzeta,\hbeta}(G))_{A \ge 1}$ is a Cauchy family in $\mathds{L}^1$, so the limit $\cZ^{\bzeta,\hbeta}(G) := \lim\limits_{A \to +\infty} \cZ_A^{\bzeta,\hbeta}(G)$ holds in $\mathds{L}^1$.

The polynomial chaos expansion of $Z_{N,A}^{\omega,\beta}(G)$ is given by 
\begin{equation*}
	\begin{split}
		&Z_{N,A}^{\omega,\beta}(G) = \E[f_A(S_N^*) G(S^{(N)})]+\\
		&\qquad \qquad \qquad \sum\limits_{k \ge 1} \frac{(\hbeta_N)^k}{k!} \sum\limits_{(t_1,x_1),\dots,(t_k,x_k) \in \Omega_{N,A}} \psi_{N,A}^{(k)}\big((t_1,x_1),\dots,(t_k,x_k);G\big) \prod\limits_{j=1}^k \frac{\omega_{N t_j, a_N x_j}}{V_N}\,,
	\end{split}
\end{equation*}
where, for every $k \ge 1$, $\psi_{N,A}^{(k)}(~\cdot~;G): \Omega_{N,A}^k \to \bbR_+$ is the symmetric function vanishing on the diagonals defined, for every $(t_1,x_1),\dots,(t_k,x_k) \in \Omega_{N,A}$ with $t_1 < \dots < t_k$, by
\begin{equation*}
	\psi_{N,A}^{(k)}\big((t_1,x_1),\dots,(t_k,x_k);G\big) := (a_N)^{kd} \E\Big[f_A(S_N^*)G(S^{(N)}) \Ind_{\{S_{N t_1} = a_N x_1, \dots, S_{N t_k} = a_N x_k\}}\Big] \,.
\end{equation*}
By the local limit theorem~\eqref{eq: local limit theorem} and the invariance principle for conditionned random walk~\cite{liggett_1970}, it converges pointwisely as $N \to +\infty$ to the symmetric function $\bpsi_A^{(k)}(~\cdot~;G)$ given for every $0 < t_1 < \dots < t_k < 1$ and $x_1,\dots,x_k \in \bbR^d$ by 
\begin{equation*}
	\bpsi_A^{(k)}\big((t_1,x_1),\dots,(t_k,x_k);G\big) = \prod\limits_{j=1}^k g_{t_j -t_{j-1}}(x_j-x_{j-1}) \bE\Big[f_A(\bX^*) G(\bX) \mid X_{t_1} = x_1,\dots, X_{t_k} = x_k \Big]\,.
\end{equation*}
Furthermore, for every $N \ge 1$, 
\[ \psi_{N,A}^{(k)}\big((t_1,x_1),\dots,(t_k,x_k);G\big) \le \|G\|_{\infty} \psi_{N,A}^{(k)}\big((t_1,x_1),\dots,(t_k,x_k)\big)\,,\]
and therefore, the integrability conditions proved for $G = 1$ in the proof of Lemma~\ref{lemma: convergence of the truncated polymer partition function} apply to a general bounded function $G$. 
\subsection{Proof of Theorem \ref{thm: convergence of the polymer measure}}

We apply Lemma~\ref{lemma: how to prove convergence of disordered measures} to prove Theorem \ref{thm: convergence of the polymer measure}.

Assumption \ref{cond: i} follows from the invariance principle, while assumption \ref{cond: ii} is established in Section~\ref{sec: extension to functional partition function}.
It therefore remains to show that $\cZ^{\bzeta,\hbeta}$ is a.s.\ strictly positive. 
This result follows by a straightforward adaptation of the proof in the case $\alpha = 2$ presented in \cite[Section~4.7]{berger_lacoin_2022}, so we have decided not to include it here.

\subsection{Proof of Theorem \ref{thm: disorder relevance for the polymer}}

Let us stress that item~\ref{res: disorder relevance on the left} is a direct consequence of Theorem~\ref{thm: convergence of the polymer partition function}: if $\lim_{N \to +\infty} \beta_N V_N (a_N)^{-d} = \hat{\beta}=0$ we have that \(Z_N^{\omega,\beta_N}\) converges in distribution to \(\cZ^{\bzeta,\hbeta=0} =1\).

For the proof of item~\ref{res: disorder relevance on the right}, this is also almost a consequence of Theorem~\ref{thm: disorder relevance}, after a truncation argument. 

\begin{lemma}
	Suppose $\lim_{N \to +\infty} \beta_N V_N (a_N)^{-d} = +\infty$, then for every fixed $A > 0$, $Z_{N,A}^{\omega,\beta_N}$ goes to $0$ in probability. 
\end{lemma}

\begin{proof}
	For a fixed $A > 0$, since $f_A(x) = 0 $ for $x>2$, we have 
	\[ Z_{N,A}^{\omega,\beta_N} \le \E\Big[\Ind_{\forall 1 \le n \le N, |S_n| \le 2 A a_N} \prod\limits_{n=1}^N (1 + \beta_N \omega_{n,S_n})\Big] \le \tilde{Z}_{N,A} := \E_{N,A}\Big[\prod\limits_{n=1}^N (1 + \beta_N \omega_{n,S_n})\Big]\,,\]
	where $\P_{N,A}(~\cdot~) = \P\Big[~\cdot~ \mid \forall 1 \le n \le N, |S_n| \le 2 A a_N\Big]$.
	In order to apply Theorem~\ref{thm: disorder relevance} to $\tilde{Z}_{N,A}$,
	we have to show that $(a_N)^d N^{-1} (a_N)^{-d} \sum_{n=1}^N \Ind_{|S_n| \le 2 A a_N}$ converges in distribution under $\P$ to a stricly positive random variable (recall that for the long-range polymer, $J_N = (a_N)^d$ and $v_N = N^{-1} (a_N)^{-d}$).
	But this is a consequence of the invariance principle for long-range random walks (see e.g.\ \cite[Thm. 7.1.]{resnick_2007}) since 
	\[ \frac{1}{N} \sum\limits_{n=1}^N \Ind_{|S_n| \le 2 A a_N} \xrightarrow{(d)} \int_{0}^1 \Ind_{|\bX_t| \le 2 A} \dd t \,,\]
	which is almost surely strictly positive by stochastic continuity of the process $(\bX_t)_{t \in [0,1]}$.
	
\end{proof}
Now, we deduce from this lemma the proof of item~\ref{res: disorder relevance on the right}. 
Fix any \(\gep>0\) and $\delta > 0$. 
For every $N \ge 1$ and $A > 0$, by~\eqref{eq: uniform inequality on approximation of partition functions}, we have 
\[ \Es[|Z_N^{\omega,\beta_N}-Z_{N,A}^{\omega,\beta_N}|] \leq \P[S_N^* \geq A]\,. \]
Choose $A > 0$ large enough so that 
\begin{equation}{\label{eq: uniform approximation}}
	\sup_{N \ge 1} \Es[|Z_N^{\omega,\beta_N}-Z_{N,A}^{\omega,\beta_N}|] \leq \delta\,.
\end{equation}
Then, for every $N \ge 1$, 
\[ \Pro[Z_N^{\omega,\beta_N} > 2\gep] = \Pro[(Z_N^{\omega,\beta_N}-Z_{N,A}^{\omega,\beta_N})+Z_{N,A}^{\omega,\beta_N} > 2\gep] \le \Pro[Z_{N,A}^{\omega,\beta_N} > \gep] + \Pro[Z_{N}^{\omega,\beta_N} - Z_{N,A}^{\omega,\beta_N} > \gep]\,.\]
Therefore, by Markov's inequality and~\eqref{eq: uniform approximation}, 
\[ \Pro[Z_N^{\omega,\beta_N} > 2\gep] \le \Pro[Z_{N,A}^{\omega,\beta_N} > \gep] + \delta \gep^{-1}\,,
\]
which yields, applying the lemma,
\( \limsup_{N \to +\infty} \Pro[Z_N^{\omega,\beta_N} > 2\gep] \le \delta \gep^{-1}\,.\)
Since \(\delta\) is arbitrary, this concludes the proof.




\appendix

\section{Lévy noises and Poisson convergence}{\label{appendix: noises and Poisson convergence}}

After having extensively dissected, partitioned and truncated our polynomial chaos, it becomes necessary, at some point, to establish a result of convergence in distribution. 
This is the purpose of this appendix. 

Let us recall our set up. 
The set $\Omega$ is bounded and open in $\bbR^D$, and for every $\delta \in (0,1)$, $\Omega_{\delta}$ is a finite subset of $\Omega$ associated to a tesselation $\cC_{\delta} = \{ \cC_{\delta}(x) \}_{x \in \Omega_{\delta}}$ of $\Omega$ (see Section~\ref{def of a tesselation} for a definition). We are given for every $\delta \in (0,1)$, i.i.d.\ and centered random variables $(\omega_{x})_{x \in \Omega_{\delta}}$ such that for every $t > 0$
\[ 
	\Pro[|\omega_x| > t ] = \varphi(t) t^{-\gamma} \,,
\]
where $\gamma \in (1,2)$ and $\varphi$ is a slowly-varying function. 
Furthermore, there are constants $c_+,c_- \in [0,1]$ with $c_+ + c_- =1 $ such that $c_+ = \lim_{t \to +\infty} \frac{\Pro[\omega_x > t]}{\Pro[|\omega_x| > t]}$. The discrete noise $\zeta_{\delta}$ is then defined by 
\[ 
	\zeta_{\delta} := \frac{1}{V_{\delta}} \sum\limits_{x \in \Omega_{\delta}} \omega_x \delta_x \,.
\]
Let $\Lambda$ be a Poisson point process on $\Omega \times \bbR$ of intensity $\mu(\dd x \dd z) = \dd x \lambda(\dd z)$ with $\lambda(\dd z) = \gamma (c_+ \Ind_{z > 0} + c_- \Ind_{z < 0}) |z|^{-1-\gamma} \dd z$. 
Our first step is to define the $\gamma$-stable noise $\bzeta$. 
For every $a \in (0,1]$, let $\Lambda^{(a)}$ denote the point process $\Lambda$ restricted to $\{ (x,z) \in \Omega \times \bbR \,, |z| > a \}$.  
This subset of $\Omega \times \bbR$ has a finite measure under $\mu$, therefore the point process $\Lambda^{(a)}$ is almost surely finite. 
This allows to define 
\[ 
	\bzeta^{(a)} = \sum\limits_{(x,z) \in \Lambda} z \Ind_{|z| > a} \delta_x - \kappa(a) \dd x\,,
\]
with $\kappa(a) = \int_{\bbR} z \Ind_{|z| > a} \lambda(\dd z)$. 
The truncated noises $(\bzeta^{(a)})_{a \in (0,1]}$ converge almost surely as $a \downarrow 0$ in some negative Sobolev spaces.  
In order to formulate our result, let us briefly introduce these spaces.

The space of infinitely differentiable functions compactly supported in $\Omega$ is denoted by $\cD_0(\Omega)$. 
Its dual space is the space of distributions $\cD'(\Omega)$, and the action of a distribution $T \in \cD'(\Omega)$ on an element $\bff \in \cD_0(\Omega)$ is denoted by $\langle T,\bff \rangle$. 
Let $s \in \mathbb{R}$, $H^s(\mathbb{R}^D)$ is the completion of $\cD_0(\Omega)$ for the Hilbertian norm
\begin{equation*}
\| \theta \|_{H^s} = \left(\int_{\mathbb{R}^D} (1+|z|^2)^s |\hat{\theta}(z)|^2 dz \right)^{1/2},
\end{equation*}
where $\hat{\theta}$ is the Fourier transform of $\theta$.
It is a linear subspace of $\cD'(\bbR^D)$.
The local version of this space is
\begin{equation*}
H_{\rm loc}^s(\Omega) = \{ \theta \in \cD'(\Omega) \,; \: \forall \psi \in \cD_0(\Omega), \psi \theta \in H^s(\bbR^D) \}.
\end{equation*}
It is a separable Fréchet space.  

\begin{theorem}{\label{thm: construction of the continuous noise}}
	$\bzeta^{(a)}$ converges almost surely in $H_{\rm loc}^{-s}(\Omega)$ for every $s >  \frac{D}{2}$. Its limit is denoted $\bzeta$ and its law is characterized by 
	\[ 
	\Es\Big[ \expo^{i \langle \bzeta, \bff \rangle} \Big] = \exp\bigg( \int_{\Omega \times \bbR} \Big( \expo^{i \bff(x) z} -1 -i \bff(x) z) \dd x \lambda(\dd z) \bigg) \,,
	\]
	for every function $\bff \in \cD_0(\Omega)$. 
\end{theorem}

The proof consists of showing that, for every function $\psi \in \cD_0(\Omega)$, almost surely, the family $(\psi \zeta^{(a)})_{a > 0}$ is Cauchy in $H^s$ as $a \to 0$.
Such a proof can be found in \cite[appendix~A]{berger_lacoin_2022}. 
Although the setting of their paper is slightly different - the authors consider more general measures $\lambda$, but supported only on $\mathbb{R}_+$)- there arguments extend to our framework without difficulty.
The set of real numbers $s$ such that  $s > \frac{D}{2} $ coincides exactly with the set of $s$ for which 
\[ 
	\int_{\mathbb{R}^D} (1+|z|^2)^{-s} dz < +\infty \,.
\]
This illustrates the fact that Lévy noises have the same Sobolev regularity as the Dirac distribution. 

\begin{theorem}{\label{thm: functional convergence for the noise}}
	Let $s > \frac{D}{2}$, then $(\zeta_{\delta})_{\delta > 0}$ converges in law as $\delta \to 0$ to $\bzeta$ for the topology of $H^{-s}_{\rm loc}(\Omega)$. 
\end{theorem}

A proof of this result can be found in \cite{berger_lacoin_2021} (Theorem 2.4). 
Again the setting is slightly different but there is no difficulty to adapt their proof. 
One needs to show that the family of random variables $(\zeta_{\delta})_{\delta > 0}$ is tight in $H^s_{\rm loc}(\Omega)$ (lemma 3.2), and then that for every $\bff \in \cD_0(\Omega)$, $\langle \zeta_{\delta}, \bff \rangle$ converges in law as $\delta \to 0$ to $\langle \bzeta, \bff \rangle$ which is a consequence of the convergence of binomials to Poisson laws. 

We now derive a useful consequence of this convergence in distribution. 
Let $k \ge 1$ and $a \in (0,1]$. 
For every symmetric function $\bff \in \cD_0(\Omega^k)$, define
\[ 
	X_{\delta}^{(a)}(\bff) = \frac{1}{k!} \sum\limits_{x_1,\dots,x_k \in \Omega_{\delta}} \bff(x_1,\dots,x_k) \prod\limits_{j=1}^k \frac{\omega_{x_j}^{(a)}}{V_{\delta}} \,,
\]
where we recall (see \eqref{eq: definition cut-off omega}) that $\omega_x^{(a)} = \omega_x \Ind_{|\omega_x| > a V_{\delta}} - \Es[\omega \Ind_{|\omega| > a V_{\delta}}]$. 

\begin{proposition}{\label{prop: convergence of chaos for truncated noise}}
	Let $\bff$ be symmetric and infinitely differentiable function with compact support in~ $\Omega^{k,*} = \{ (x_1,\dots,x_k) \in \Omega^k \,, \forall i \neq j, x_i \neq x_j \}$. 
	Then for every $s > \frac{D}{2}$ the following joint convergence in distribution holds in $H_{\rm loc}^{-s}(\Omega) \times \bbR$ as $\delta \downarrow 0$:
	\[ 
	(\zeta_{\delta}, X_{\delta}^{(a)}(\bff)) \xrightarrow{(d)} (\bzeta,\bX^{(a)}(\bff)) \,,
	\]
	where $\bX^{(a)}(\bff) = \frac{1}{k!} \int_{\Omega^k} \bff(x_1,\dots,x_k) \prod\limits_{j=1}^k \bzeta^{(a)}(\dd x_j)$. 
\end{proposition}

\begin{proof}
	Define for $a \in (0,1]$ and $\delta \in (0,1)$,
	\[ 
	\zeta_{\delta}^{(a)} := \frac{1}{V_{\delta}} \sum\limits_{x \in \Omega_{\delta}} \omega_x^{(a)} \delta_x = \frac{1}{V_{\delta}} \sum\limits_{x \in \Omega_{\delta}} \omega_x \Ind_{|\omega_x| > a V_{\delta}} \delta_x - \frac{1}{V_{\delta}} \Es[\omega \Ind_{|\omega| > a V_{\delta}}] \sum\limits_{x \in \Omega_{\delta}} \delta_x \,.
	\]
	Using the fact that $|\Omega| = v_{\delta} \: \# \Omega_{\delta}$ and the definition of $V_{\delta}$ (\eqref{eq: general scale of the noise}), the second term in the above expression (which is deterministic) converges (in any negative Sobolev space) to $\kappa(a) \dd x$ as $\delta \downarrow 0$. 
	Additionally, the proof of Theorem \ref{thm: functional convergence for the noise} adapts to the truncated random variables $\omega^{(a)}$, therefore for every fixed $a \in (0,1]$, $\zeta_{\delta}^{(a)}$ converges in distribution in $H^{-s}_{\rm loc }(\Omega)$ to $\bzeta^{(a)}$.  
	Now observe that 
	\[ 	X_{\delta}^{(a)}(\bff) = \frac{1}{k!} \int_{\Omega^k} \bff(x_1,\dots,x_k) \prod\limits_{j=1}^k \zeta_{\delta}^{(a)}(\dd x_j) \,.
	\]
	Therefore, the joint convergence of $(\zeta_{\delta},X_{\delta}^{(a)}(\bff))$ to $(\bzeta,\bX^{(a)}(\bff))$ as $\delta \to 0$ follows from the continuity of the map on $H^{-s}_{\rm loc}(\Omega)$ defined by
	\[ \theta \mapsto \int_{\Omega^k} \bff(x_1,\dots,x_k) \prod\limits_{j=1}^k \theta(\dd x_j)\,.\] 
	In fact, we prove the slightly stronger statement: the multilinear map 
	\[T : 
	\left\{ 
		\begin{array}{ccccl}
		H^{-s}_{\rm loc}(\Omega)^k \times H^{ks}(\mathbb{R}^{kd}) & \longrightarrow & \bbR \\
		(\theta_1,\dots,\theta_k;\bff) & \mapsto & \int_{\Omega^k} \bff(x_1,\dots,x_k) \prod\limits_{j=1}^k \theta_j(\dd x_j) \,, \\
	\end{array}
	\right.\]
	is continuous. 
	Let us first see why it is well-defined. 
	Given $x_2,\dots,x_k \in \Omega$, the function $\bff(\cdot,x_2,\dots,x_k)$ is an element of $\cD_0(\Omega)$, therefore $\langle \theta_1, \bff(\cdot,x_2,\dots,x_k) \rangle$ is well-defined and the map $(x_2,\dots,x_k) \mapsto \langle \theta_1, \bff(\cdot,x_2,\dots,x_k)$ is in $\cD_0(\Omega^{k-1})$. 
	Therefore one can apply $\theta_2$ with respect to the second variable $x_2$ as we did for the first variable. 
	Following this procedure iteratively gives a definition of the integral of $\bff$ with respect to $\theta_1,\dots,\theta_k$. 
	This definition is not ambiguous since $\bff$ is symmetric. 
	
	Assume that $\zeta_1,\dots,\zeta_k$ are in $H^{-s}(\bbR^D)$. 
	Then by Plancherel's identity and the definition of the Sobolev norm,
	\[ 
	|T(\theta_1,\dots,\theta_k; \bff)| \le \prod\limits_{j=1}^k \|\theta_j\|_{H^{-s}(\bbR^D)} \| \bff \|_{H^{ks}(\bbR^{kD})} \,.
	\]
	For the general case, let $K_1,...,K_k$ be compact sets of $\Omega$ such that the support of $\bff$ is included in $\prod\limits_{j=1}^k K_j$. 
	Take for every $1 \le j \le k$ a smooth and compactly supported function $\chi_j$ such that $\chi_j = 1$ on $K_j$. 
	Define $\chi(x_1,\dots,x_k) = \prod\limits_{j=1}^k \chi_j(x_j)$. 
	Then
	\[T(\theta_1,\dots,\theta_k;\bff) = T(\theta_1,\dots,\theta_k;\chi \bff) = T(\chi_1 \theta_1,\dots,\chi_k \theta_k; \bff) \le \prod\limits_{j=1}^k \| \chi_j \theta_j\|_{H^{-s}(\bbR^D)} \|\bff\|_{H^{ks}(\bbR^{kD})} \,,
	\]
	which concludes the proof.
\end{proof}

\section{About homogeneous partition functions}{\label{appendix: homogeneous partition functions}}

This Appendix is devoted to the proof of Lemma~\ref{lemma: scaling limit of the homogeneous partition functions}. 
First, we need the following uniform bound, that we also use in Section~\ref{section: proofs polymer}. 

\begin{proposition}{\label{prop: technical estimate on slow varying functions}}
	Let $L$ be a strictly positive slowly-varying function and let $\xi > 0$.
	For every $\gep > 0$, there exist a constant $C = C_{\gep}$ such that for every $k,N \in \bbN^*$,
	\begin{multline*}
		\sum\limits_{1 \le n_1 < \dots < n_k \le N} \prod\limits_{j=1}^k L(n_j-n_{j-1})(n_j-n_{j-1})^{\xi-1} 
		\\ \le C^k L(N)^k N^{k \xi} \int_{0 < t_1 < \dots < t_k < 1} \prod\limits_{j=1}^k (t_j-t_{j-1})^{\xi-\gep -1} \dd t_j \,,
	\end{multline*}
	and 
	\begin{multline*}
		\sum\limits_{1 \le n_1 < \dots < n_k < n_{k+1} := N} \prod\limits_{j=1}^{k+1} L(n_j-n_{j-1})(n_j-n_{j-1})^{\xi-1} \\
		\le C^{k+1} L(N)^{k+1} N^{k \xi + \xi -1} \int_{0 < t_1 < \dots < t_k < t_{k+1} =1} \prod\limits_{j=1}^{k+1} L(t_j-t_{j-1}) (t_j-t_{j-1})^{\xi -\gep - 1} \prod\limits_{j=1}^k \dd t_j \,.
	\end{multline*}
\end{proposition}

\begin{proof}
	By standard comparison arguments, there exist a constant $C$ such that for every $k,N \in \bbN^*$,
	\begin{multline*}
		\sum\limits_{1 \le n_1 < \dots < n_k \le N} \prod\limits_{j=1}^k L(n_j-n_{j-1})(n_j-n_{j-1})^{\xi-1} 
		\\ \le C^k \int_{0 < t_1 < \dots < t_k < N} \prod\limits_{j=1}^k L(t_j-t_{j-1}) (t_j - t_{j-1})^{\xi - 1} \dd t_j \,.
	\end{multline*}
	With a change of variable $t_j = N s_j$, we get that the right-hand side is equal to 
	\begin{equation*}
		C^k N^{k \xi} L(N)^k \int_{0 < s_1 < \dots < s_k < 1} \prod\limits_{j =1}^k \frac{L(N(s_j-s_{j-1}))}{L(N)} (s_j - s_{j-1})^{\xi -1} \dd s_j.
	\end{equation*}
	By Potter's bound (see \cite[Thm 1.5.6.]{bingham_goldie_teugels_1987}), for every $\gep > 0$, there exist a constant $C = C(\gep)$ such that for every $s \le t$, $L(s) \le L(t) \big(\frac{t}{s}\big)^{\gep}$. 
	This concludes the proof of the first inequality, and the proof of the second one follows the same lines.
\end{proof}

\begin{proof}[Proof of Lemma \ref{lemma: scaling limit of the homogeneous partition functions}]
	Introduce $\lambda = \expo^{h}-1$. 
	We can rewrite the conditioned partition function in terms of $\lambda$ as follows:
	\begin{equation*}
		Z_{M,h}^c = \expo^{h} \E\left[ \prod\limits_{n=1}^{M-1} (1+\lambda \Ind_{n \in \tau}) \Ind_{M \in \tau} \right]\,.
	\end{equation*}
	Expanding the product $\prod\limits_{n=1}^N (1+\lambda \Ind_{n \in \tau})$ and taking expectation yields
	\begin{equation*}
		Z_{M,h}^c = \expo^{h} \sum\limits_{k \ge 0} \lambda^k \sum\limits_{n_0 := 0 < n_1 < \dots < n_k < n_{k+1} := M} \prod\limits_{j=1}^{k+1} u(n_j - n_{j-1}).
	\end{equation*}
	First, let us prove the uniform bound~\eqref{eq: uniform bound on the conditioned partition function}. 
	By~\eqref{eq: uniform bound on the renewal function}, for every $k \ge 1$,
	\begin{equation*}
		\sum\limits_{n_0 := 0 < n_1 < \dots < n_k < n_{k+1} := M} \prod\limits_{j=1}^{k+1} u(n_j - n_{j-1}) \le C^{k+1} \sum\limits_{n_0 := 0 < n_1 < \dots < n_k < n_{k+1} := M} \prod\limits_{j=1}^{k+1} \frac{(n_j-n_{j-1})^{\alpha-1}}{\ell(n_j-n_{j-1})}\,.
	\end{equation*}
	Applying Proposition~\ref{prop: technical estimate on slow varying functions} and~\eqref{eq: gamma integrals}, we obtain that the left-hand side is bounded by
	\begin{equation*}
		C^{k+1} \frac{M^{k \alpha+\alpha -1 }}{\ell(M)^{k+1}} \frac{\Gamma(\alpha)^{k+1}}{\Gamma((k+1)\alpha)}.
	\end{equation*}
	Since $\lambda = e^h -1 \sim h$ as $h \to 0$, and $h_N \sim \hh N^{-1} u(N)^{-1}$ as $N \to +\infty$, by\eqref{eq: uniform bound on the renewal function}, there exist a constant $C = C_{\hh}$ such that
	\[ \lambda_N \le C \frac{\ell(N)}{N^{\alpha}}\,.\]
	Therefore, 
	\begin{equation*}
		Z_{M,h_N}^c \le C \frac{M^{\alpha-1}}{\ell(M)} \sum\limits_{k \ge 0} \bigg( C \frac{M^{\alpha} \ell(N)}{N^{\alpha} \ell(M)}\bigg)^k \frac{1}{\Gamma((k+1)\alpha)}\,.
	\end{equation*}
	By \cite[Thm 1.5.4.]{bingham_goldie_teugels_1987}, 
	\[\sup\limits_{N \ge 1} \sup\limits_{1 \le M \le N} \frac{M^{\alpha} \ell(N)}{N^{\alpha} \ell(M)} < +\infty
	.\]	
	Since the sequence $(\Gamma((k+1)\alpha))_{k \ge 0}$ decreases faster than any exponential rate, this concludes the proof of \eqref{eq: uniform bound on the conditioned partition function}. 
	
	Suppose now that $M_N/N \to t \in (0,1]$.
	For every $\gep > 0$ and $k \in \mathbb{N}$, we split the indexes of summations into 
	\begin{equation*}
	\cD_1^{\gep,N}(k) := \{ n_0 :=0 < n_1 < \dots < n_k < n_{k+1}:= M_N;\: \forall j \in \intg{1}{k+1},\, n_j-n_{j-1} \ge \gep N \}\,,
	\end{equation*}
	and 
	\begin{equation*}
		\cD_2^{\gep,N}(k) := \{ n_0 :=0 < n_1 < \dots < n_k < n_{k+1}:= M_N;\: \exists j \in \intg{1}{k+1},\, n_j-n_{j-1} < \gep N \}\,,
	\end{equation*}
	which gives a decomposition of $Z_{M_N,h_N}^c$ into two pieces: 
	\[ Z_{N,\gep}^{c,1} = \expo^{h_N} \sum\limits_{k \ge 0} (\lambda_N)^k \sum\limits_{n_1,\dots,n_k \in \cD_1^{\gep,N}(k)} \prod\limits_{j=1}^{k+1} u(n_j-n_{j-1})\,
	\] 
	and 
	\[Z_{N,\gep}^{c,2} = \expo^{h_N} \sum\limits_{k \ge 0} (\lambda_N)^k \sum\limits_{n_1,\dots,n_k \in \cD_2^{\gep,N}(k)} \prod\limits_{j=1}^{k+1} u(n_j-n_{j-1})\,.
	\]
	Rewriting $Z_{N,\gep}^{c,1}$ with the function $U_N$ (introduced in \eqref{eq: rescaled renewal function}), we get
	\begin{equation*}
		Z_{N,\gep}^{c,1} = \expo^{h_N} u(N) \sum\limits_{k \ge 0} (\lambda_N u(N) N)^k \int_{\underset{\forall j \in \intg{1}{k+1}\,, t_j - t_{j-1} \ge \gep}{0 < t_1 < \dots < t_k < \frac{M_N}{N}}} \prod\limits_{j=1}^{k+1} U_N\Big(\frac{\floor{N t_j}-\floor{N t_{j-1}}}{N} \Big) \prod\limits_{j=1}^k \dd t_j\,.
	\end{equation*}
	Since the function $U_N$ converges uniformly on $(0,1]$ to the function $U(t) = t^{\alpha -1}$, we get that for every fixed $\gep > 0$,
	\begin{equation*}
		\frac{Z_{N,\gep}^{c,1}}{u(N)} \xrightarrow[N \to +\infty]{} \sum\limits_{k \ge 0} \hh^k \int_{\underset{\forall j \in \intg{1}{k+1}\,, t_j - t_{j-1} \ge \gep}{0 < t_1 < \dots < t_k < t}} \prod\limits_{j=1}^{k+1} (t_j-t_{j-1})^{\alpha-1} \prod\limits_{j=1}^k \dd t_j \,,
	\end{equation*}
	and the right hand side converges as $\gep \to 0$ to $\cZ_{t,\hh}^c$ (recall \ref{eq: gamma integrals}). 
	It only remains to show that $Z_{N,\gep}^{c,2}$ is asymptotically negligible as $\gep \downarrow 0$. 
	Using again the bound  \eqref{eq: uniform bound on the renewal function} and the fact that $M_N \le N$, we get for every $ k \ge 0$, 
	\begin{equation*}
		\sum\limits_{n_1,\dots,n_k \in \cD_2^{\gep,N}(k)} \prod\limits_{j=1}^{k+1} u(n_j-n_{j-1}) \le C^{k+1} \sumtwo{1 \le n_1 < \dots < n_k < N}{\exists j \in \intg{1}{k+1}, n_j - n_{j-1} \le \gep N} \prod\limits_{j=1}^k \frac{(n_j-n_{j-1})^{\alpha-1}}{\ell(n_j-n_{j-1})} \,,
	\end{equation*}
	which is bounded (by Proposition \ref{prop: technical estimate on slow varying functions}) by
	\begin{equation*}
		C^{k+1} \frac{N^{k \alpha + \alpha -1}}{\ell(N)^{k+1}} \int_{\underset{\exists j \in \intg{1}{k+1}\,, t_j - t_{j-1} < \gep}{0 < t_1 < \dots < t_k < 1}} \prod\limits_{j=1}^{k+1} (t_j-t_{j-1})^{\alpha-1} \prod\limits_{j=1}^k \dd t_j \,.
	\end{equation*}
	Therefore, 
	\begin{equation*}
		\limsup\limits_{N \to +\infty} \frac{Z_{N,\gep}^{c,2}}{u(N)} \le C \sum\limits_{k \ge 0} C^k \int_{\underset{\exists j \in \intg{1}{k+1}\,, t_j - t_{j-1} < \gep}{0 < t_1 < \dots < t_k < 1}} \prod\limits_{j=1}^{k+1} (t_j-t_{j-1})^{\alpha-1} \prod\limits_{j=1}^k \dd t_j\,,
	\end{equation*}
	which goes to $0$ as $\gep \downarrow 0$. 
\end{proof}

\printbibliography

\end{document}